\documentclass[a4paper]{amsart}
\usepackage{graphicx}
\usepackage{amssymb}
\usepackage{amsmath}
\usepackage{amsthm}
\usepackage{amscd}
\usepackage{color}
\usepackage{colordvi}
\usepackage{tikz}
\usepackage[all,2cell]{xy}
\usepackage{picture}
\usepackage{multirow}
\usepackage{makecell}

\usepackage[pagebackref,colorlinks]{hyperref}

\usepackage{amsfonts,enumerate,verbatim,mathtools,tikz,bm,tikz-cd,hyperref,comment}
\hypersetup{
	colorlinks=true,
	linkcolor=blue,
	filecolor=magenta,
	urlcolor=cyan,
}

\UseAllTwocells \SilentMatrices
\newtheorem{thm}{Theorem}[section]

\newtheorem{cor}[thm]{Corollary}
\newtheorem{lem}[thm]{Lemma}

\newtheorem{theorem}{Theorem}
\newtheorem{example}{Example}
\newtheorem{conjecture}{Conjecture}
\newtheorem{lemma}{Lemma}

\newtheorem{thmy}{Theorem}
\newenvironment{thmx}{\stepcounter{thm}\begin{thmy}}{\end{thmy}}

\newtheorem{exm}[thm]{Example}

\newtheorem{prop}[thm]{Proposition}
\theoremstyle{definition}
\newtheorem{defn}[thm]{Definition}
\theoremstyle{remark}
\newtheorem{rem}[thm]{\bf Remark}
\numberwithin{equation}{section}

\newcommand{\sg}{\mathrm{sg}}
\newcommand{\HH}{\mathrm{HH}}

\newcommand{\ie}{{\em i.e.}\ }

\newcommand{\Z}{\mathbb{Z}}
\newcommand{\N}{\mathbb{N}}

\newcommand{\ca}{{\mathcal A}}
\newcommand{\cb}{{\mathcal B}}
\newcommand{\cc}{{\mathcal C}}
\newcommand{\cd}{{\mathcal D}}

\newcommand{\ch}{{\mathcal H}}

\newcommand{\cn}{{\mathcal N}}

\newcommand{\cp}{{\mathcal P}}
\newcommand{\cR}{{\mathcal R}}

\newcommand{\cs}{{\mathcal S}}
\newcommand{\ct}{{\mathcal T}}

\newcommand{\iso}{\xrightarrow{_\sim}}
\newcommand{\ko}{\: , \;}
\newcommand{\ten}{\otimes}
\newcommand{\op}{\mathrm{op}}

\newcommand{\Hom}{\mathrm{Hom}}
\newcommand{\Ext}{\mathrm{Ext}}
\newcommand{\lten}{\overset{\boldmath{L}}{\ten}}
\newcommand{\id}{\mathbf{1}}
\newcommand{\Si}{\Sigma}

\newcommand{\Mod}{\mathrm{Mod}\,}
\newcommand{\dgMod}{\mathrm{dgMod}\,}
\newcommand{\dgcat}{\mathrm{dgcat}}
\newcommand{\Com}{\mathrm{Com}\,}
\newcommand{\dgCom}{\mathrm{dgCom}\,}
\newcommand{\Inj}{\mathrm{Inj}\,}
\renewcommand{\mod}{\mathrm{mod}\,}
\newcommand{\com}{\mathrm{com}\,}
\newcommand{\per}{\mathrm{per}\,}
\newcommand{\thick}{\mathrm{thick}\,}
\newcommand{\pretr}{\mathrm{pretr}\,}
\newcommand{\gr}{\mathrm{gr}\,}

\begin{document}
\title[Leavitt algebra, singular Yoneda category and singularity]{The dg Leavitt algebra, singular Yoneda category and singularity category}
\author[Xiao-Wu Chen, Zhengfang Wang] {Xiao-Wu Chen and Zhengfang Wang,\\
 with an appendix  by Bernhard Keller and Yu Wang}

\thanks{}
\date{\today}


\maketitle

\dedicatory{}%
\commby{}%

\begin{abstract}
For any finite dimensional algebra $\Lambda$ given by a quiver with relations,  we prove that its dg singularity category is quasi-equivalent to the perfect dg derived category of a dg Leavitt path algebra. The result might be viewed as a deformed version of the known description of the dg singularity category of a radical-square-zero algebra in terms of a Leavitt path algebra with trivial differential.

The above result is achieved in two steps. We first introduce the singular Yoneda dg category of $\Lambda$, which is quasi-equivalent to the dg singularity category of $\Lambda$. The construction of this new dg category follows from a general operation for dg categories, namely an explicit dg localization inverting a natural transformation from the identity functor to a dg endofunctor. This localization turns out to be quasi-equivalent to a  dg quotient category. Secondly,  we prove that the endomorphism algebra of the quotient of $\Lambda$ modulo its Jacobson radical in the singular Yoneda dg category is  isomorphic to the dg Leavitt path algebra. The appendix is devoted to  an alternative proof of the result  using Koszul-Moore duality and derived localizations.
 \end{abstract}

\tableofcontents

\section{Introduction}

\subsection{The background and main results}

Let $\mathbb{K}$ be a field and $\Lambda$  a finite dimensional algebra over $\mathbb{K}$. The \emph{singularity category} $\mathbf{D}_{\rm sg}(\Lambda)$ of $\Lambda$ is defined as the Verdier quotient
of the bounded derived category $\mathbf{D}^b(\Lambda\mbox{-mod})$ of finitely generated left $\Lambda$-modules by the full subcategory of perfect complexes. This notion was first introduced in \cite{Buc} and then rediscovered in \cite{Orl} motivated by the homological mirror symmetry conjecture.  The singularity category measures the homological singularity of the algebra: it vanishes if and only if the algebra $\Lambda$ is of finite global dimension.

The homotopy category $\mathbf{K}_{\rm ac}(\Lambda\mbox{-Inj})$ \cite{Kra} of acyclic complexes of arbitrary injective $\Lambda$-modules is a \emph{compactly generated completion} of $\mathbf{D}_{\rm sg}(\Lambda)$. This means that $\mathbf{K}_{\rm ac}(\Lambda\mbox{-Inj})$ is compactly generated and that its full subcategory of compact objects is triangle equivalent to $\mathbf{D}_{\rm sg}(\Lambda)$. However, in general, we do not know whether $\mathbf{D}_{\rm sg}(\Lambda)$ determines  $\mathbf{K}_{\rm ac}(\Lambda\mbox{-Inj})$ uniquely as a triangulated category.

As is well known, triangulated categories arising naturally in algebra usually have a dg enhancement, that is, there is a pretriangulated dg category whose zeroth cohomology yields the given triangulated category \cite{LO}. For instance,  the \emph{dg singularity category} $\mathbf{S}_{\rm dg}(\Lambda)$  \cite{Kel18, BRTV, BrDy} is a canonical dg enhancement of $\mathbf{D}_{\rm sg}(\Lambda)$, which is defined to be the dg quotient of the bounded dg derived category $\mathbf{D}_{\rm dg}^b(\Lambda\mbox{-mod})$ by the full dg subcategory of perfect complexes.

In comparison with the singularity category, the dg singularity category contains more information and has more invariants. For example, the above completion $\mathbf{K}_{\rm ac}(\Lambda\mbox{-Inj})$ is uniquely determined by $\mathbf{S}_{\rm dg}(\Lambda)$: there is a triangle equivalence
\begin{align}\label{intro:1}
\mathbf{K}_{\rm ac}(\Lambda\mbox{-Inj})\simeq \mathbf{D}(\mathbf{S}_{\rm dg}(\Lambda)^{\rm op}),
\end{align}
where $\mathbf{D}(\mathbf{S}_{\rm dg}(\Lambda)^{\rm op})$ is the derived category of right dg $\mathbf{S}_{\rm dg}(\Lambda)$-modules; see \cite{Kra, CLiuW}.  The main theorem in \cite{Kel18}  states that under mild conditions the Hochschild cohomology of $\mathbf{S}_{\rm dg}(\Lambda)$ is isomorphic to the singular Hochschild cohomology \cite{ZF} of the algebra $\Lambda$.

We are interested in describing the (dg) singularity categories of $\Lambda$. Let us assume for a moment that $\Lambda=\mathbb{K}Q/I$, where $\mathbb{K}Q$ is the path algebra of a finite quiver $Q$ and $I$ is an admissible ideal of $\mathbb{K}Q$.

Recall from \cite{Smi} the description of $\mathbf{D}_{\rm sg}(\Lambda)$ when $\Lambda$ is radical square zero, i.e.\ the ideal $I$ contains all paths of length two. Then $\Lambda=\mathbb{K}Q_0\oplus \mathbb{K}Q_1$ has a basis given by vertices and arrows in $Q$. Denote by $Q^\circ$ the finite quiver without sinks, which is obtained from $Q$ by removing sinks repeatedly. The corresponding  Leavitt path algebra $L(Q^\circ )$ in the sense of \cite{AA, AGGP, AMP} is naturally graded, and is viewed as a dg algebra with trivial differential.  One of the main results in \cite{Smi} states a triangle equivalence
$$\mathbf{D}_{\rm sg}(\Lambda)\simeq \mathbf{per}(L(Q^\circ)).$$
Here, $\mathbf{per}(L(Q^\circ))$ denotes the perfect derived category of left dg $L(Q^\circ)$-modules. Indeed, by the work \cite{Li,CLW},  such a triangle equivalence lifts to a quasi-equivalence between the corresponding dg enhancements
$$\mathbf{S}_{\rm dg}(\Lambda)\simeq \mathbf{per}_{\rm dg}(L(Q^\circ)).$$
Combining (\ref{intro:1}) with the quasi-equivalence above, we  recover the following triangle  equivalence in \cite{CY}:
$$\mathbf{K}_{\rm ac}(\Lambda\mbox{-Inj})\simeq \mathbf{D}(L(Q^\circ)),$$
where $\mathbf{D}(L(Q^\circ))$ denotes the derived category of left dg $L(Q^\circ)$-modules.

 We observe that usually the Leavitt path algebra $L(Q^\circ)$ is infinite dimensional in each degree, and that correspondingly the singularity category of $\Lambda$ is usually Hom-infinite \cite{Chen11}. Leavitt path algebras are related to noncommutative geometry \cite{Smi}, symbolic dynamic systems \cite{ALPS, Haz}, graph $C^*$-algebras \cite{CO, AAM} and algebraic bivariant K-theory \cite{Cor,CorMon18, CorMon20}.

We will extend the above description to the general case. For $\Lambda=\mathbb{K}Q/I$, we have a natural decomposition $\Lambda=\mathbb{K}Q_0\oplus J$ with $J$ its Jacobson radical.  Following \cite{Sch}, we introduce the \emph{radical quiver} $\widetilde Q$ of $\Lambda$: it has the same vertex set as $Q$, that is, $\widetilde Q_0=Q_0$; for any vertices $i$ and $j$, the arrows from $i$ to $j$ correspond to elements in a basis of $e_jJe_i$. Here, $e_i$ denotes the corresponding primitive idempotent of the vertex $i$. In other words, we identify $J$ with $\mathbb{K}\widetilde Q_1$, the vector space spanned by the arrow set $\widetilde Q_1$ of $\widetilde Q$. The multiplication on $J$ is transferred to an associative product
$$\mu\colon \mathbb{K}\widetilde Q_1\otimes_{\mathbb{K}\widetilde Q_0} \mathbb{K}\widetilde Q_1\longrightarrow \mathbb{K}\widetilde Q_1.$$
In this way, $\Lambda$ is viewed as a \emph{deformation} of the radical-square-zero algebra $\widetilde{\Lambda}:=\mathbb{K}\widetilde Q_0\oplus \mathbb{K}\widetilde Q_1$; see \cite{Sch,BW}. We mention that the algebra $\Lambda$ may be recovered from $\widetilde{\Lambda}$ using the product $\mu$.

It is well known that such an associative product $\mu$ gives rise to a differential on the path algebra of the opposite quiver of $\widetilde Q$; see \cite{BSZ}.  In the same manner, it gives rise to a differential $\partial$ on the Leavitt path algebra $L(\widetilde Q^\circ)$. Here, $\widetilde Q^\circ$ is the quiver without sinks that is obtained from $\widetilde Q$ by removing sinks repeatedly. The resulting  dg Leavitt path algebra is denoted by  $L(\widetilde Q^\circ)_\partial$ temporarily.

The following result describes the (dg) singularity categories of $\Lambda$ in terms of dg Leavitt path algebras; see Theorem~\ref{thm:quiveralgebra}. It indicates that dg Leavitt path algebras are ubiquitous in the study of singularity categories.

\begin{thmx}\label{thmA} Let $\Lambda=\mathbb{K}Q/I$ be a finite dimensional algebra, and  $\widetilde Q$  be its radical quiver. Then there is a quasi-equivalence
\[
\mathbf{S}_{\rm dg}(\Lambda) \simeq \mathbf{per}_{\rm dg}(L(\widetilde Q^\circ)_\partial).
\]
Consequently, there are triangle equivalences
$$\mathbf{D}_{\rm sg}(\Lambda) \simeq \mathbf{per}(L(\widetilde Q^\circ)_\partial) \quad \mbox{ and } \quad \mathbf{K}_{\rm ac}(\Lambda\mbox{-}{\rm Inj})\simeq \mathbf{D}(L(\widetilde Q^\circ)_\partial).$$
\end{thmx}

If the algebra $\Lambda$ is radical square zero, then $\widetilde Q=Q$ and the differential $\partial$ vanishes. Applying Theorem~\ref{thmA} to this situation, we recover the mentioned results in \cite{Smi} and \cite{CY,CLW}.

 The idea behind Theorem~\ref{thmA} is illustrated by the following diagram.
\begin{align}\label{align:deformationdiagram}
\xymatrix@C=4pc{
\widetilde{\Lambda} \ar@{~>}[d]_-{\text{deform}}\ar@{<->}[rr]^-{\text{quasi-equivalence}} && L(\widetilde Q^{\circ}) \ar@{~>}[d]^-{\text{deform}}\\
\Lambda   \ar@{<->}[rr]^-{\text{quasi-equivalence}} && L(\widetilde Q^\circ)_\partial
}
\end{align}
The horizontal arrows indicate the quasi-equivalences between the relevant dg singularity categories and perfect dg derived categories. For the vertical arrow on the right, we mention that it is customary to deform a dg algebra by only changing its differential, which is a particular $A_\infty$-deformation \cite{TT}; compare \cite{Kel11}. However, we do not know how to deduce the quasi-equivalence at the bottom from the one at the top via the deformation theory \cite{KeLo, LoVa} of dg categories; see Remark~\ref{rem:deformation}(2).

  The proof of Theorem~\ref{thmA} is divided into two steps. We first introduce the \emph{singular Yoneda dg category} $\mathcal{SY}$ of $\Lambda$, which turns out to be quasi-equivalent to $\mathbf{S}_{\rm dg}(\Lambda)$.
  Secondly,   using the explicit description of $\mathcal{SY}$,  we show that  the endomorphism algebra of $\mathbb{K}Q_0$ in  $\mathcal{SY}$ is isomorphic to the dg Leavitt path algebra $L(\widetilde Q^\circ)_\partial$.

  The construction of $\mathcal{SY}$ follows from a general operation for dg categories described as follows.   Let $\mathcal{C}$ be a dg category, $\Omega$ a dg endofunctor on $\mathcal{C}$, and $\theta\colon {\rm Id}_\mathcal{C}\rightarrow \Omega$ a closed natural transformation of degree zero satisfying $\theta\Omega=\Omega\theta$. By inverting  $\theta_X$ for all objects $X$, we construct a new and explicit dg category $\mathcal{SC}$ with a dg functor
   $$\iota\colon \mathcal{C}\longrightarrow \mathcal{SC},$$
   called the (strict) \emph{dg localization} along $\theta$. The objects of $\mathcal{SC}$ are the same as $\mathcal C$, and the Hom complexes are defined via a colimit construction which is similar to the one in defining the singular Hochschild cochain complex \cite{Wan1}.

    To obtain $\mathcal{SY}$ from the above  general operation, we consider the \emph{Yoneda dg category} $\mathcal Y$, a natural dg enhancement of the derived category of $\Lambda$ via the bar resolution \cite{Kel94}. The relevant dg endofunctor on $\mathcal Y$ is induced by noncommutative differential forms \cite{CQ,ZF}.

   The quasi-equivalence between $\mathcal{SY}$ and $\mathbf{S}_{\rm dg}(\Lambda)$ is a special case of the following general result; see Theorem~\ref{thm:dgl}. We mention that a similar idea of this result is implicitly contained in \cite[Section~7]{Kel05}.

\begin{thmx} \label{thmB} Assume that $\mathcal{C}$ is pretriangulated. Then $\mathcal{SC}$ is pretriangulated and $\iota$ induces a quasi-equivalence
$$\mathcal{C}/\mathcal{N}\stackrel{\sim}\longrightarrow \mathcal{SC},$$
where $\mathcal{N}$ is the full dg subcategory formed by the cones of $\theta_X$, and $\mathcal{C}/\mathcal{N}$ denotes the dg quotient category.
\end{thmx}

For dg quotient categories, we refer to \cite{Kel99, Dri}. In general, the structure of a dg quotient category is rather complicated and mysterious. Theorem~\ref{thmB} describes certain dg quotient categories explicitly.

\subsection{Conventions and structure}  We fix a commutative ring $\mathbb{K}$ and work over $\mathbb{K}$. This means that we require that all the  categories and functors are  $\mathbb{K}$-linear. In Section~\ref{section10} we will further assume that $\mathbb K$ is a field.

We fix two $\mathbb{K}$-algebras $E$ and $\Lambda$ together with a fixed homomorphism $E\rightarrow \Lambda$ of $\mathbb{K}$-algebras. Denote by $\overline{\Lambda}$  its cokernel, which is naturally an $E$-$E$-bimodule. In most cases,  we will further assume that $\Lambda$ is augmented over $E$, that is, there is an algebra homomorphism $\pi\colon \Lambda\rightarrow E$ such that the composition $E\rightarrow \Lambda\stackrel{\pi}\rightarrow E$ is the identity. Then the $E$-$E$-bimodule $\overline{\Lambda}$ has an induced associative product $\mu\colon \overline{\Lambda}\otimes_E\overline{\Lambda}\rightarrow \overline{\Lambda}$ by identifying $\overline \Lambda$ with the kernel of $\pi$. We will view $E$ as a $\Lambda$-module via the homomorphism $\pi$.

In the sequel, we will work in the relative setup. For example, we will study various $E$-relative derived categories and $E$-relative singularity categories of $\Lambda$.

By default, a module means a left module. For example, ${\rm Hom}_E(-,-)$ means the Hom bifunctor on the category of left $E$-modules. A left $E$-module $M$ is usually denoted by $_EM$, which emphasizes the $E$-action from the left side.

Throughout, we use cohomological notation. In the dg setup, we always consider homogeneous elements or morphisms. The translation functor on any triangulated category is denoted by $\Sigma$.

\bigskip

The paper is organized as follows. In Section~\ref{section2}, we study the Cohn algebra $C_E(M)$ and Leavitt algebra $L_E(M)$ associated to an $E$-$E$-bimodule $M$. We prove that the Leavitt algebra is isomorphic to the colimit of an explicit sequence; see Theorem~\ref{thm:Leavitt-col}.

In Section~\ref{section3}, we assume that the bimodule $M$ is equipped with an associative product $\mu\colon M\otimes_E M\rightarrow M$. We show that $\mu$  induces a differential on the Cohn algebra, which descends to a differential on the Leavitt algebra. Consequently, we obtain the dg Cohn algebra and dg Leavitt algebra associated to $(M, \mu)$. In Section~\ref{section4}, we work in the quiver case. More precisely, for a finite quiver $Q$, we set $E=\mathbb{K}Q_0$ and $M=\mathbb{K}Q_1$. Applying the results in Section~\ref{section3} to this situation, we obtain the dg Cohn path algebra and dg Leavitt path algebra; compare Proposition~\ref{prop:iso-path}.

We recall basic facts on pretriangulated dg categories in Section~\ref{section5}. We introduce an explicit dg localization in Section~\ref{section6}. The universal property in Proposition~\ref{prop:dgl} justifies this terminology.  Theorem~\ref{thm:dgl} shows that the dg localization is quasi-equivalent to a dg quotient category.

Inspired by \cite{Kel94}, we introduce the Yoneda dg category $\mathcal{Y}$ of $\Lambda$ in Section~\ref{section7}. It is quasi-equivalent to the dg derived category of $\Lambda$; see Proposition~\ref{prop:Theta} and Corollary~\ref{cor:finite}.  We prove that the endomorphism algebra of $E$ in $\mathcal{Y}$ is isomorphic to a dg tensor algebra associated to $(\overline{\Lambda}, \mu)$; see Proposition~\ref{prop:iso-Y-tensor}.

In Section~\ref{section8}, we introduce noncommutative differential forms \cite{CQ,ZF} with values in complexes of $\Lambda$-modules. This gives rise to a dg endofunctor  $\Omega_{\rm nc}$ on $\mathcal{Y}$, together with a natural transformation $\theta \colon \mathrm{Id}_{\mathcal Y} \to \Omega_{\rm nc}$. We actually show that the assumptions for the dg localization in Section~\ref{section6} are satisfied on $(\mathcal{Y}, \Omega_{\rm nc}, \theta)$.

In Section~\ref{section9}, we take the dg localization of $\mathcal{Y}$ along $\theta$, and obtain the singular Yoneda dg category $\mathcal{SY}$ of $\Lambda$. This terminology is justified by Proposition~\ref{prop:V-SY} and Corollary~\ref{cor:finite-SY}, that is, the singular Yoneda dg category $\mathcal{SY}$ is quasi-equivalent to the dg singularity category.   In Theorem~\ref{thm:SY-Leavitt}, we prove that the endomorphism algebra of $E$ in $\mathcal{SY}$ is exactly isomorphic to the dg Leavitt algebra $L_E(\overline{\Lambda})$ associated to $(\overline{\Lambda}, \mu)$, studied in Section~\ref{section3}.

In Section \ref{section10}, we apply Theorem~\ref{thm:SY-Leavitt} to any finite dimensional algebra $\Lambda$. Theorem~\ref{thm:quiveralgebra} relates the dg singularity category of $\Lambda$ to the dg Leavitt path algebra, which is associated to the radical quiver $\widetilde{Q}$ of $\Lambda$ and a transferred associative product $\mu$ on $\mathbb{K}\widetilde{Q}_1$. In the end, we give an explicit example of a dg Leavitt path algebra, whose minimal $A_\infty$-model is  explicitly described.

In the appendix, Bernhard Keller and Yu Wang give an alternative proof of Theorem ~\ref{thm:quiveralgebra}  using Koszul-Moore duality in \cite{Kel03a} and derived localizations in \cite{BCL18}.

\section{The Cohn and Leavitt algebras}\label{section2}

Throughout this section, we assume that $E$ is a $\mathbb{K}$-algebra and that $M$ is an $E$-$E$-bimodule on which $\mathbb{K}$ acts centrally. We study the Cohn algebra and Leavitt algebra associated to $M$.

 Denote by $M^*={\rm Hom}_E(M, E)$ the left dual $E$-$E$-bimodule whose bimodule structure is given by
\begin{align}
\label{equ:dual-action}
(afb)(m)=f(ma)b \quad \quad \text{for $a, b\in E$, $m\in M$ and $f\in M^*$.}
\end{align}
Denote by
$$T_E(M^*)=E\oplus M^*\oplus (M^*)^{\otimes_E 2}\oplus \cdots$$ the tensor algebra. Its typical element $f_1\otimes_E f_2\otimes_E \cdots \otimes_E f_q$, with $f_i\in M^*$ for  each  $1\leq i\leq q$, will be abbreviated as $f_{1, q}$. If $q=0$, the notation $f_{1,q}$ usually means the unit element $1_E$.

Inspired by \cite[\S~8]{Co} and \cite[Definition~1.5.1]{AAM}, we will define the \emph{Cohn algebra} $C_E(M)$ associated to $M$ as follows. As a $\mathbb{K}$-module, we have
$$C_E(M):=T_E(M^*) \otimes_E T_E(M) =  \bigoplus_{p\geq 0} T_E(M^*)\otimes_E {M^{\otimes_E p}}.$$
Its typical element
\begin{align}\label{align:typicalelement}
f_1\otimes_E f_2\otimes_E \cdots \otimes_E f_q\otimes_E x_1\otimes_E x_2\otimes_E \cdots \otimes_E x_p
\end{align}
will be abbreviated as $f_{1,q}\otimes_E x_{1,p}$ for $f_i\in M^*$, $x_j\in M$ and $p, q\geq 0$. Take another typical element $g_{1, s}\otimes_E y_{1, t}$. The multiplication of $C_E(M)$ is determined by  the following rule:
\begin{align}\label{equ:mult}
(f_{1,q}\otimes_E x_{1,p}) \bullet (g_{1, s}\otimes_E y_{1, t})= f_{1, q}\otimes_E Z\otimes_E y_{1,t}
\end{align}
where the middle tensor $Z$ is equal to
\begin{align*}
\begin{cases}
g_p(x_1g_{p-1}(x_2g_{p-2}(\dotsb (x_{p-1}g_1(x_p))\dotsb)))\,  g_{p+1,s}\in (M^*)^{\otimes_E s-p},                    & \text{ if  $p<s$}; \\
g_p(x_1g_{p-1}(x_2g_{p-2}(\dotsb (x_{p-1}g_1(x_p))\dotsb))) \quad \in E,     & \text{ if  $p=s$};\\
x_{1, p-s} \, g_s(x_{p-s+1}g_{s-1}(x_{p-s+2}g_{s-2}(\dotsb (x_{p-1}g_1(x_p))\dotsb))) \in M^{\otimes_E p-s},                                 & \text{ if $p>s$}.
\end{cases}
\end{align*}
It is routine to verify that the above multiplication makes $C_E(M)$ into an associative $\mathbb{K}$-algebra and that its unit is given by $1_E \in E$. We observe that $T_E(M^*)$ and $T_E(M)$ are naturally subalgebras of $C_E(M)$.

The following example illustrates the multiplication $\bullet$ of $C_E(M)$ in more detail.

\begin{exm} We have
$$(f_{1,q}\otimes_E x_{1,3}) \bullet (g_{1,4}\otimes_E y_{1,t})=f_{1, q}\otimes_E g_3(x_1 g_2(x_2 g_1(x_3)))g_4\otimes_E y_{1,t},$$
which lies in $(M^*)^{\otimes_E (q+1)}\otimes_E M^{\otimes_E t}$. The element $g_3(x_1 g_2(x_2 g_1(x_3)))$ lies in $E$, and the expression $g_3(x_1 g_2(x_2 g_1(x_3)))g_4$ means the left $E$-action of $g_3(x_1 g_2(x_2 g_1(x_3)))$ on the element $g_4 \in M^*$; see (\ref{equ:dual-action}). Similarly, we have
$$(f_{1,q}\otimes_E x_{1,4}) \bullet (g_{1,3}\otimes_E y_{1,t})=f_{1, q}\otimes_E x_1g_3(x_2 g_2(x_3 g_1(x_4)))\otimes_E y_{1,t},$$
which lies in $(M^*)^{\otimes_E q}\otimes_E M^{\otimes_E (t+1)}$. The expression $x_1g_3(x_2 g_2(x_3 g_1(x_4)))$ means the right $E$-action of $g_3(x_2 g_2(x_3 g_1(x_4))) \in E$ on the element $x_1 \in M$.
\end{exm}

We observe that $x\bullet g=g(x)\in E$ for $x\in M$ and $g\in M^*$. Therefore, the inclusion $M^*\oplus M\subseteq C_E(M)$ induces a well-defined $\mathbb{K}$-algebra homomorphism
$$\Phi \colon T_E(M^*\oplus M)/{ ( x\otimes_E g-g(x)\; |\; x\in M, g\in M^*)}\longrightarrow C_E(M).$$

\begin{prop}\label{prop:isoC}
The above algebra homomorphism $\Phi$ is an isomorphism.
\end{prop}

\begin{proof}
Denote the domain of $\Phi$ by $R$. In $C_E(M)$, we have
$$f_{1,q}\otimes_E x_{1,p}=f_1\bullet \cdots \bullet f_q\bullet x_1\bullet \cdots \bullet x_p.$$
It follows that $E\oplus (M^*\oplus M)$ generates $C_E(M)$ and thus $\Phi$ is surjective.

We define a $\mathbb{K}$-linear map
$$\Phi' \colon C_E(M)\longrightarrow R,$$
 which sends a typical element $f_{1,q}\otimes_E x_{1,p}\in C_E(M)$ to the image of the corresponding tensor $f_{1,q}\otimes_E x_{1,p}\in T_E(M^*\oplus M)$ in $R$. Using (\ref{equ:mult}), we verify that $\Phi'$ is an algebra homomorphism. We deduce $\Phi' \circ \Phi={\rm Id}_R$ by evaluating the both sides on $E\oplus (M^*\oplus M)$. Then $\Phi$ is injective, proving the required statement.
\end{proof}

\begin{rem}\label{rem:CO}
The following evaluation map
$${\rm ev}\colon M\otimes_E M^*\longrightarrow E, \quad x\otimes_E g\mapsto g(x)$$
is an $E$-$E$-bimodule map. Then $(M, M^*, {\rm ev})$ is an $R$-system in the sense of \cite[Definition~1.1]{CO}. By the above isomorphism, we observe that the Cohn algebra $C_E(M)$ is isomorphic to the Toeplitz ring of $(M, M^*, {\rm ev})$; see \cite[Theorem~1.7]{CO}.
\end{rem}

Assume that the underlying left $E$-module of  $M$ is finitely generated projective. We have the canonical isomorphism of $E$-$E$-bimodules
$$M^*\otimes_E M\stackrel{\sim}\longrightarrow {\rm Hom}_E(M, M), \quad\quad  f\otimes_E x\mapsto (m\mapsto f(m)x).$$
We denote by $c\in M^*\otimes_E M$ the preimage of ${\rm Id}_M$, which is called the \emph{Casimir element} of $M$. We observe that $ac=ca$ for any $a\in E$.

 Write $c=\sum_{i\in S}\alpha_i^*\otimes_E \alpha_i\in M^*\otimes_E M$. Then $\{\alpha_i\}_{i\in S}$ and $\{\alpha_i^*\}_{i\in S}$ form a dual basis of $M$, i.e.\ we have
\begin{align}\label{equ:dualbasis}
x=\sum_{i\in S}\alpha_i^*(x)\alpha_i \quad \quad \mbox{and} \quad \quad f=\sum_{i\in S}\alpha_i^* f(\alpha_i)
\end{align}
for any $x\in M$ and $f\in M^*$. We will view $c$ as an element in $C_E(M)$.

The following definition is inspired by \cite[\S~3]{Lea} and \cite{AA, AGGP, AMP}.

\begin{defn}
Let $M$ be an $E$-$E$-bimodule with $_EM$ finitely generated projective. The \emph{Leavitt algebra} $L_E(M)$ associated to $M$ is defined to be the  quotient algebra
$$L_E(M)=C_E(M)/{(1_E-c)}.$$
\end{defn}

\vskip 3pt

By the above isomorphism $\Phi$, we infer an isomorphism of algebras
$$L_E(M)\simeq T_{E}(M^*\oplus M)/{( m\otimes_E g-g(m), \ 1_E-c \mid m\in M, g\in M^*)}.$$
Similar to Remark~\ref{rem:CO}, the Leavitt algebra $L_E(M)$ is isomorphic to the Cuntz-Pimsner ring of $(M, M^*, {\rm ev})$ relative to the whole algebra $E$; see \cite[Definition~3.16]{CO} and compare \cite[Example~5.8]{CO}.

\begin{lem}\label{lem:princ}
The principal ideal $(1_E-c)$ of $C_E(M)$ is spanned, as an $E$-$E$-bimodule, by  elements of the form $f_{1,q}\otimes_E x_{1,p}-f_{1,q}\otimes_E c\otimes_E x_{1,p}$ for $p, q \geq 0$.
\end{lem}

\begin{proof}
Denote by $I$  the $E$-$E$-subbimodule of $C_E(M)$ spanned by elements of the form $f_{1,q}\otimes_E x_{1,p}-f_{1,q}\otimes_E c\otimes_E x_{1,p}.$  Since
$$f_{1,q}\otimes_E x_{1,p}-f_{1,q}\otimes_E c\otimes_E x_{1,p}=f_{1, q}\bullet (1_E-c)\bullet x_{1, p},$$
it follows that $I\subseteq (1_E-c)$.  We claim that $I$ is a two-sided ideal of $C_E(M)$. Then  the required statement follows.

We only prove that $I$ is a left ideal, since similarly one proves that it is also a right ideal. It is clear that $I$ is a left $T_E(M^*)$-submodule of $C_E(M)$. Hence, it suffices to prove that for any $x \in M$, the element $w:=x\bullet (f_{1,q}\otimes_E x_{1,p}-f_{1,q}\otimes_E c\otimes_E x_{1,p})$ still lies in  $I$.

There are two cases.  If $q\geq 1$, then $w=f_1(x)(f_{2,q}\otimes_E x_{1,p}-f_{2,q}\otimes_E c\otimes_E x_{1,p})$, which clearly lies in $I$. If $q=0$, we have
$$w=x\otimes_E x_{1, p}-\sum_{i\in S}\alpha_i^*(x)\alpha_i\otimes_E x_{1, p}=0,$$
where the right equality follows from (\ref{equ:dualbasis}). Then $w$ trivially lies in $I$.
\end{proof}

For each $p\geq 0$, we have a natural morphism of $T_E(M^*)$-$E$-bimodules
\begin{align}\label{equ:nat-mor-c}
T_E(M^*)\otimes_E M^{\otimes_E p}&\longrightarrow T_E(M^*)\otimes_E M^{\otimes_E(p+1)}\\
 f_{1, q}\otimes_E x_{1, p}&\longmapsto f_{1, q}\otimes_E c\otimes_E x_{1, p}. \nonumber
\end{align}
Letting $p$ vary, we obtain a sequence of morphisms.

We have the following structure theorem on Leavitt algebras; compare \cite[Subsections~1.2 and 5.5]{Smi}.

\begin{thm}\label{thm:Leavitt-col}
Let $M$ be an $E$-$E$-bimodule with $_EM$ finitely generated projective. Then as a $T_E(M^*)$-$E$-bimodule, the Leavitt algebra $L_E(M)$ is isomorphic to the colimit of the above sequence.
\end{thm}

\begin{proof}
By the construction of colimits, the mentioned colimit is isomorphic to the following quotient bimodule
$$(\bigoplus_{p\geq 0} T_E(M^*)\otimes_E M^{\otimes_E p})/I=C_E(M)/I,$$
where $I$ is the $E$-$E$-subbimodule spanned by elements of the form $f_{1, q}\otimes_E x_{1, p}- f_{1, q}\otimes_E c\otimes_E x_{1, p}$. By Lemma~\ref{lem:princ}, $I$ coincides with the principal ideal $(1_E-c)$ of $C_E(M)$. Then we are done.
\end{proof}

\section{The dg Cohn and Leavitt algebras}\label{section3}

As in the previous section, let $E$ be a $\mathbb{K}$-algebra and $M$ be an $E$-$E$-bimodule on which $\mathbb{K}$ acts centrally. Throughout this section, we further assume that $_EM$ is finitely generated projective.  We will introduce the dg Cohn algebra and dg Leavitt algebra associated to a pair $(M, \mu)$, where $\mu$ is an associative bilinear map on $M$.

Recall that $V^*={\rm Hom}_E(V, E)$ for any $E$-$E$-bimodule $V$. We observe that the following canonical map of $E$-$E$-bimodules
\[
\begin{array}{rccc}
{\rm can}\colon&  M^*\otimes_E M^* &\longrightarrow & (M\otimes_E M)^*\\
 & f_1\otimes_E f_2 &\longmapsto& (x_1\otimes_E x_2\mapsto f_2(x_1f_1(x_2))\in E)
 \end{array}
\]
is an isomorphism.

We fix an $E$-$E$-bimodule homomorphism
$$\mu\colon M\otimes_E M\longrightarrow M$$
 which is associative, that is,
$$\mu\circ (\mu\otimes_E {\rm Id}_M)=\mu\circ ({\rm Id}_M\otimes_E \mu).$$
Then we have two induced maps of $E$-$E$-bimodules:
\begin{align}\label{equ:par+}
\partial_{+}\colon M^*\stackrel{\mu^*}\longrightarrow (M\otimes_E M)^* \xrightarrow{{\rm can}^{-1}} M^*\otimes_E M^*
\end{align}
and
\begin{align}\label{equ:par-}
\partial_{-}\colon M\longrightarrow M^*\otimes_E M\otimes_E M \xrightarrow{{\rm Id}_{M^*}\otimes_E \mu} M^*\otimes_E M.
\end{align}
Here, the unnamed arrow sends $x$ to $c\otimes_E x$ with $c$ the Casimir element of $M$.

The following elementary facts are well known; compare \cite[3.7]{Swe}. We mention that the first statement is somehow dual to \cite[Remark~4.17]{Kuel} in the bocs theory.

\begin{lem}\label{lem:coass}
Keep the notation as above. Then the following statements hold.
\begin{enumerate}
\item[(1)] The map $\partial_+$ is coassociative, that is, $(\partial_+\otimes_E {\rm Id}_{M^*})\circ \partial_+=( {\rm Id}_{M^*}\otimes_E \partial_+)\circ \partial_+$.
    \item[(2)] The map $\partial_{-}$ makes $M$ into a left $(M^*, \partial_+)$-comodule, that is, $(\partial_{+}\otimes_E {\rm Id}_M)\circ \partial_{-}=({\rm Id}_{M^*}\otimes_E \partial_{-})\circ \partial_{-}$.
\end{enumerate}
\end{lem}

\begin{proof}
(1)  follows from the associativity of $\mu$ by duality. We observe that
$$(\partial_+\otimes_E {\rm Id}_{M^*})\circ \partial_+=({\rm can}_1)^{-1}\circ ( \mu \otimes_E {\rm Id}_M)^*\circ \mu^* =({\rm can}_1)^{-1}\circ (\mu\circ( \mu \otimes_E {\rm Id}_M))^*.$$
Here,  ${\rm can}_1\colon M^*\otimes_E M^*\otimes_E M^*\rightarrow (M\otimes_E M\otimes_E M)^*$ is the canonical isomorphism sending $f_{1,3}$ to the map $(x_{1, 3}\mapsto f_3(x_1f_2(x_2f_1(x_3)))\in E)$. Similarly, we have
\[
({\rm Id}_{M^*} \otimes_E \partial_+ )\circ \partial_+ = ({\rm can}_1)^{-1} \circ (\mu \circ ({\rm Id}_{M}  \otimes_E \mu))^*.
\]

(2) Recall that $c=\sum_{i \in S}\alpha_i^*\otimes_E \alpha_i\in M^*\otimes_E M$. For any $x\in M$,  we have
$$(\partial_{+}\otimes_E {\rm Id}_M)(\partial_{-}(x))=\sum_{i\in S}\partial_+(\alpha_i^*)\otimes_E \mu(\alpha_i\otimes_E x)$$
and
\begin{align*}
({\rm Id}_{M^*}\otimes_E \partial_{-})(\partial_{-}(x)) &=\sum_{i,j\in S}\alpha_i^*\otimes_E \alpha_j^*\otimes_E \mu(\alpha_j\otimes_E\mu(\alpha_i\otimes_E x))\\
&=\sum_{i,j\in S}\alpha_i^*\otimes_E \alpha_j^*\otimes_E \mu(\mu(\alpha_j\otimes_E\alpha_i)\otimes_E x)).
\end{align*}
The associativity of $\mu$ is used in the last equality. Therefore, it suffices to verify the following identity in
$M^*\otimes_E M^*\otimes_E M$.
\begin{align}\label{equ:comod}
\sum_{i\in S}\partial_+(\alpha_i^*)\otimes_E \alpha_i=\sum_{i,j\in S}\alpha_i^*\otimes_E \alpha_j^*\otimes_E \mu(\alpha_j\otimes_E \alpha_i)
\end{align}

  There is a canonical isomorphism ${\rm can}_2\colon M^*\otimes_E M^*\otimes_E M \rightarrow {\rm Hom}_E(M\otimes_E M, M)$ sending $f_{1,2}\otimes_E y$ to the map $(x_{1,2}\mapsto f_2(x_1f_1(x_2))y\in M)$. On one hand, we have
  \begin{align*}
  {\rm can}_2\Big(\sum_{i\in S}\partial_+(\alpha_i^*)\otimes_E \alpha_i\Big)(x_1\otimes_E x_2)=\sum_{i\in S} \alpha_i^*(\mu(x_1\otimes_E x_2))\alpha_i=\mu(x_1\otimes_E x_2).
  \end{align*}
  Here, the left equality uses the definition of $\partial_+$ and the right one uses (\ref{equ:dualbasis}). On the other hand, we have
  \begin{align*}
  {\rm can}_2\Big(\sum_{i,j\in S}\alpha_i^*\otimes_E \alpha_j^*\otimes_E \mu(\alpha_j\otimes_E \alpha_i)\Big)(x_1\otimes_E x_2)&=\sum_{i, j\in S} \alpha_j^*(x_1\alpha_i^*(x_2))\mu(\alpha_j\otimes_E \alpha_i)\\
    & =\sum_{i\in S}  \mu\Big(\sum_{j\in S}\alpha_j^*(x_1\alpha_i^*(x_2))\alpha_j\otimes_E \alpha_i\Big)\\
    &=\sum_{i \in S} \mu(x_1\alpha_i^*(x_2)\otimes_E \alpha_i)\\
    &=\mu\Big(x_1\otimes_E \sum_{i\in S}\alpha_i^*(x_2)\alpha_i\Big)\\
    &=\mu(x_1\otimes_E x_2).
  \end{align*}
  Here, both the third and fifth equalities use (\ref{equ:dualbasis}). Then we are done with (\ref{equ:comod}).
\end{proof}

We consider the tensor algebra $T_E(M^*\oplus M)$. It is $\mathbb{Z}$-graded by means of ${\rm deg}\;E=0$, ${\rm deg}\; M^*=1$ and ${\rm deg}\;  M=-1$.

 We apply  \cite[Lemma~1.8]{BSZ} with $\delta_0\colon E \rightarrow M^*\oplus M$ being the zero map and $\delta_1$ given by the following map:
 $$M^*\oplus M \xrightarrow{\partial_+\oplus \partial_{-}} (M^*\otimes_E M^*)\oplus (M^*\otimes_E M)\subseteq (M^*\oplus M)\otimes_E (M^*\oplus M).$$
 Then there is a  unique $E$-derivation
$$\partial\colon T_E(M^*\oplus M)\longrightarrow T_E(M^*\oplus M)$$
 of degree one, such that $\partial (x) = \partial_-(x)$ and $\partial(f) = \partial_+(f)$ for any $x \in M$ and $f \in M^*$. This means that $\partial$ satisfies the following graded Leibniz rule
\begin{align}\label{equ:Leibniz}
\partial(u \otimes_E v)=\partial(u)\otimes_E v +(-1)^{|u|} u\otimes_E \partial(v)
\end{align}
for any homogenous elements $u, v\in T_E(M^*\oplus M)$, and $\partial(a)=0$ for any $a\in E$.

 We observe that $\partial^2|_{M^*}=0$ by Lemma~\ref{lem:coass}(1) and that $\partial^2|_{M}=0$ by Lemma~\ref{lem:coass}(2); here, we use the minus sign in the graded Leibniz rule. By \cite[Remark~1.7(3)]{BSZ}, we infer that $\partial^2=0$. In other word, $(T_E(M^*\oplus M), \partial)$ is a dg tensor algebra; compare \cite[Section~1]{BSZ}.

\begin{rem}\label{rem:dgten}
We mention the following \emph{asymmetry} in the dg tensor algebra $(T_E(M^*\oplus M), \partial)$: the subalgebra $T_E(M)$ is not closed under $\partial$, while the subalgebra $T_E(M^*)$ is closed under $\partial$ and thus becomes a dg subalgebra.
\end{rem}

\begin{lem}\label{lem:c-closed}
The Casimir element $c$, viewed as an element in $T_E(M\oplus M^*)$, is closed, that is, $\partial(c)=0$.
\end{lem}

\begin{proof}
By the graded Leibniz rule, we have
\begin{align*}
\partial(c)=&\sum_{i\in S} \partial_{+}(\alpha_i^*)\otimes_E \alpha_i-\sum_{i\in S} \alpha_i^*\otimes_E \partial_{-}(\alpha_i)\\
&=\sum_{i\in S} \partial_{+}(\alpha_i^*)\otimes_E \alpha_i-\sum_{i,j\in S} \alpha_i^*\otimes_E \alpha_j^*\otimes_E\mu(\alpha_j\otimes_E\alpha_i)=0.
\end{align*}
Here, the second equality uses the construction \eqref{equ:par-} of $\partial_{-}$ and the last equality is precisely (\ref{equ:comod}).
\end{proof}

\begin{lem}\label{lem:dgidealclosed}
The two-sided ideal  $( x\otimes_E g-g(x)\; |\; x\in M, g\in M^*)$ of  $T_E(M^*\oplus M)$  is a dg ideal, that is, it is closed under $\partial$.
\end{lem}

\begin{proof}
Denote the above ideal by $J$. By the graded Leibniz rule, it suffices to prove that $\partial(x\otimes_E g-g(x))$ still lies in $J$. Since $g(x) \in E$ and thus $\partial (g(x)) = 0$, we have
$$\partial(x\otimes_E g-g(x))=\partial_{-}(x)\otimes_E g - x\otimes_E \partial_+(g).$$

Define an element  $\phi\in M^*$  by
$$\phi(y)=g(\mu(y\otimes_E x))\; \mbox{ for any } y\in M.$$
 By the definition of $\partial_{-}$, we have
 $$\partial_{-}(x)\otimes_E g=\sum_{i\in S}\alpha_i^*\otimes_E \mu(\alpha_i\otimes_E x)\otimes_E g.$$
  Therefore, we infer that the element $w_1:=\partial_{-}(x)\otimes_E g-\sum_{i \in S}\alpha_i^*g(\mu(\alpha_i\otimes_E x))$ lies in the two-sided ideal $J$. Moreover, we have
$$w_1=\partial_{-}(x)\otimes_E g-\sum_{i\in S}\alpha_i^*\phi(\alpha_i)=\partial_{-}(x)\otimes_E g-\phi, $$
where  the right equality uses (\ref{equ:dualbasis}).

Write $\partial_+(g)=\sum_{j\in T} h_j\otimes_E f_j$. Then the following element
$$w_2:= x\otimes_E \partial_+(g)-\sum_{j\in T} h_j(x)f_j=\Big(\sum_{j\in T} (x\otimes_E h_j-h_j(x))\Big)\otimes_E f_j$$
 lies in $J$. By the definition \eqref{equ:par+} of $\partial_+$, we have $\mathrm{can} \circ \partial_+(g)(y \otimes_E x) = \mu^*(g)(y \otimes_E x)$ for any $y \in M$. Namely,  we have
 \begin{align}\label{align:verydefintionpartial}
 \sum_{j\in T}f_j(yh_j(x))=g(\mu(y\otimes_E x))=\phi(y).
 \end{align}
 We infer that $w_2=x\otimes_E \partial_{+}(g)-\phi$. Since $w_1, w_2$ lie in $J$, so does their difference $w_1-w_2 = \partial_{-}(x)\otimes_E g  - x\otimes_E \partial_+(g)$. We deduce the required statement.
\end{proof}

Combining Proposition~\ref{prop:isoC} with Lemma~\ref{lem:dgidealclosed}, we infer that the differential $\partial$ on $T_{E}(M^*\oplus M)$ descends to the Cohn algebra $C_E(M)$. By Lemma~\ref{lem:c-closed}, the Casimir element $c$ is closed in $C_E(M)$. Therefore, the differential $\partial$ descends further to the Leavitt algebra $L_E(M)$.

By abuse of notation, we will use $\partial$ to denote the induced differentials on both $C_E(M)$ and $L_E(M)$. We emphasize that $\partial$ is uniquely determined by $\partial_{+}$ in (\ref{equ:par+}) and $\partial_{-}$  in (\ref{equ:par-}) via the graded Leibniz rule (\ref{equ:Leibniz}).

\begin{defn}\label{defn:dgleavittalgebra}
Let $M$ be an $E$-$E$-bimodule with $_EM$ finitely generated projective. Assume that $\mu\colon M\otimes_E M\rightarrow M$ is an associative morphism of $E$-$E$-bimodules.  The resulting dg algebras $(C_E(M), \partial)$ and $(L_E(M), \partial)$ are called the \emph{dg Cohn algebra} and \emph{dg Leavitt algebra} associated to $(M, \mu)$, respectively.
\end{defn}

\begin{rem}\label{rem:partial}
We claim that the differential $\partial$ on $L_E(M)$ is completely determined by $\partial_{+}$. To be more precise, we assume that $\partial'$ is any $E$-derivation on $L_E(M)$ whose restriction to $M^*$ is $\partial_{+}$. We will show that the restriction of $\partial'$ to $M$ necessarily coincides with $\partial_{-}$, which particularly yields $\partial' = \partial$.

Recall the Casimir element $c = \sum_{i \in S} \alpha_i^* \otimes_E \alpha$. For any $x\in M$, we have $x\otimes_E\alpha_i^*=\alpha_i^*(x)\in E$. Therefore, we have
$$0=\partial'(x\otimes_E\alpha_i^*)=\partial'(x)\otimes_E\alpha_i^*-x\otimes_E\partial_+(\alpha_i^*).$$
By the relation $1_E=c$, we have
\begin{align}\label{equ:partial-uniq}
\partial'(x)=\sum_{i\in S}\partial'(x)\otimes_E\alpha_i^*\otimes_E\alpha_i=\sum_{i \in S}x\otimes_E\partial_{+}(\alpha_i^*)\otimes_E\alpha_i.
\end{align}
The above identity together with the graded Leibniz rule \eqref{equ:Leibniz}  already confirms the claim. Moreover, we have
\begin{align*}
\sum_{i \in S}x\otimes_E \partial_{+}(\alpha_i^*) \otimes_E\alpha_i &=\sum_{i, j\in S}\alpha_i^*(x) \alpha_j^* \otimes_E \mu(\alpha_j\otimes_E \alpha_i)\\
&= \sum_{i, j \in S}\alpha_j^* \otimes_E \mu(\alpha_j\alpha_i^*(x)\otimes_E \alpha_i)\\
&= \sum_{i, j \in S}\alpha_j^* \otimes_E \mu(\alpha_j\otimes_E \alpha_i^*(x)\alpha_i)\\
&= \sum_{j \in S}\alpha_j^* \otimes_E \mu(\alpha_j\otimes_E x)=\partial_{-}(x).
\end{align*}
Here, the first equality uses (\ref{equ:comod}) and the relation $x\otimes_E \alpha_i^*=\alpha_i^*(x)$, the second one uses the fact $\alpha_i^*(x)c=c\alpha_i^*(x)\in M^*\otimes_E M$, and the fourth one uses (\ref{equ:dualbasis}). This proves $\partial'(x)=\partial_{-}(x)$. Since any $E$-derivation on $L_E(M)$ is uniquely determined by its values at the generating $E$-$E$-bimodule $M^* \oplus M$, it follows that $\partial'= \partial.$
\end{rem}

\section{Quivers and path algebras}\label{section4}

In this section, we study the dg Cohn algebra and dg Leavitt algebra in the quiver situation, namely the dg Cohn path algebra and dg Leavitt path algebra, respectively. The differentials are described explicitly.

A quiver is a directed graph. Formally, it is a quadruple  $Q=(Q_{0}, Q_{1}; s, t)$ consisting of a set $Q_{0}$ of vertices, a set $Q_{1}$ of arrows and two maps $s, t\colon  Q_{1}\xrightarrow []{}Q_{0}$, which associate to each arrow $\alpha$ its starting vertex $s(\alpha)$ and its terminating vertex $t(\alpha)$, respectively. We visualize an arrow $\alpha$ as $\alpha\colon s(\alpha)\rightarrow t(\alpha)$. A vertex is called a {\it sink} if no arrow starts in this vertex.  The quiver $Q$ is \emph{finite} provided that both $Q_0$ and $Q_1$ are finite sets.

We fix a finite quiver $Q$.  A path of length $n$ is a sequence $p=\alpha_{n}\dotsb\alpha_{2}\alpha_{1}$ of arrows with $t(\alpha_{j})=s(\alpha_{j+1})$ for $1\leq j\leq n-1$. Denote by $l(p)=n$.  The starting vertex of $p$, denoted by $s(p)$, is defined to be $s(\alpha_{1})$. The terminating vertex of $p$, denoted by $t(p)$, is defined to be  $t(\alpha_{n})$. We identify an arrow with a path of length one. We associate to each vertex $i\in Q_{0}$ a \emph{trivial} path $e_{i}$ of length zero, and set $s(e_i)=i=t(e_i)$. Denote by $Q_n$ the set of paths of length $n$.

The \emph{path algebra} $\mathbb{K} Q=\bigoplus_{n\geq 0}\mathbb{K} Q_{n}$ is a free $\mathbb{K}$-module with  a basis given by all the paths in $Q$, whose multiplication is given as follows: for two paths $p$ and $q$ satisfying $s(p)=t(q)$, the product $pq$ is their concatenation; otherwise,  the product $pq$ is defined to be zero. Here, we write the concatenation of paths from right to left. For example, we have  $e_{t(p)}p=p=pe_{s(p)}$ for each path $p$.

  We denote by $\overline{Q}$ the \emph{double quiver} of $Q$, which is obtained from $Q$ by adding for each arrow $\alpha\in Q_1$ a new arrow $\alpha^*$ in the opposite direction, that is, we have $s(\alpha^*)=t(\alpha)$ and $t(\alpha^*)=s(\alpha)$. The added arrows $\alpha^*$ are called the \emph{ghost} arrows. Denote by $Q_1^*$ the set formed by the ghost arrows. More generally, we denote by $Q_n^*$ the set formed by  all paths of length $n$ which consist entirely of ghost arrows. For a path $p=\alpha_n\cdots \alpha_2\alpha_1\in Q_n$, we set $p^*=\alpha_1^*\alpha_2^*\cdots \alpha_n^*\in Q^*_n$.

Set $E=\mathbb{K}Q_0=\bigoplus_{i\in Q_0}\mathbb{K}e_i$ and $M=\mathbb{K}Q_1$. Then $E$ is a subalgebra of $\mathbb{K}Q$ and $M$ is naturally an $E$-$E$-bimodule. Recall that $M^*={\rm Hom}_E(M, E)$. Each ghost arrow $\alpha^*$ gives rise to an element $(\beta\mapsto \delta_{\alpha, \beta}e_{t(\alpha)})$ in  $M^*$. Here, $\delta$ is the Kronecker symbol. In this way, we have an identification of $E$-$E$-bimodules.
$$M^*=\mathbb{K}Q_1^*$$

It is well known that the inclusions $E\subseteq \mathbb{K}Q$ and $M\subseteq \mathbb{K}Q$ induce a canonical isomorphism $$T_E(M)\stackrel{\sim}\longrightarrow \mathbb{K}Q$$
of algebras. In more detail, a tensor $\alpha_n\otimes_E\cdots \otimes_E \alpha_2 \otimes_E\alpha_1$  of arrows is sent to the corresponding path $p=\alpha_n\cdots \alpha_2\alpha_1$. Similarly, we use the above identification $M^*=\mathbb{K}Q_1^*$ and embed $M^*$ into $\mathbb{K}\overline{Q}$. Then we obtain a canonical isomorphism
\begin{align}\label{iso:tensor}
T_E(M^*\oplus M)\stackrel{\sim}\longrightarrow \mathbb{K}\overline{Q}
\end{align}
of algebras.

As in \cite[Definition~1.5.1]{AAM}, the \emph{Cohn path algebra} $C(Q)$ is defined as the following quotient algebra.
$$C(Q)=\mathbb{K}\overline{Q}/{(\alpha\beta^*-\delta_{\alpha, \beta}e_{t(\alpha)}\; |\; \alpha, \beta\in Q_1)}$$
Following \cite{AA, AGGP, AMP}, the \emph{Leavitt path algebra} $L(Q)$ is defined as the further quotient algebra.
$$L(Q)=C(Q)/{(e_i-\sum_{\{\alpha\in Q_1\; |\; s(\alpha)=i\}} \alpha^*\alpha \; |\; i\in Q_0 \mbox{ are non-sinks in $Q$})}$$
The relations $\alpha\beta^*-\delta_{\alpha, \beta}e_{t(\alpha)}$ and $e_i-\sum_{\{\alpha\in Q_1\; |\; s(\alpha)=i\}} \alpha^*\alpha$ are known as the first and second \emph{Cuntz-Krieger relations}, respectively.  For relations between Cohn path algebras and Leavitt path algebras, we refer to \cite[Theorem~1.5.18]{AAM}.

Denote by $Q^\circ$ the quiver without sinks, that is obtained from $Q$ by removing sinks repeatedly.

\begin{prop}\label{prop:iso-path}
Keep the notation as above. Then the following statements hold.
\begin{enumerate}
\item[(1)] There is a canonical isomorphism $C_E(M)\simeq C(Q)$ of algebras.
\item[(2)] There is a canonical isomorphism $L_E(M)\simeq L(Q^\circ)$ of algebras.
\end{enumerate}
\end{prop}

\begin{proof}
(1) follows from Proposition~\ref{prop:isoC} and the isomorphism (\ref{iso:tensor}). Here, we observe that the element $\alpha\otimes_E \beta^*-\beta^*(\alpha)\in T_E(M^*\oplus M)$ corresponds to $\alpha\beta^*-\delta_{\alpha, \beta}e_{t(\alpha)}\in \mathbb{K}\overline{Q}$.

(2) Recall that $1_E=\sum_{i\in Q_0}e_i$. Using the above identification $M^*=\mathbb{K}Q_1^*$, we infer that the Casimir element $c\in M^*\otimes_E M$ corresponds to $\sum_{\alpha\in Q_1}\alpha^*\alpha\in C(Q)$. Therefore, using (1), we infer that $L_E(M)$ is isomorphic to
\begin{align*}
&C(Q)\big /{\big(\sum_{i\in Q_0}e_i-\sum_{\alpha\in Q_1}\alpha^*\alpha\big)}\\
={} &C(Q)\Big/ {\Big(e_i-\sum_{\{\alpha\in Q_1\; |\; s(\alpha)=i\}} \alpha^*\alpha,\  e_j\; |\; i\in Q_0 \mbox{ non-sinks, }  j\in Q_0 \mbox{ sinks}\Big)}.
\end{align*}
It follows that $L_E(M)$ is isomorphic to $L(Q)/{(e_j\; |\; j\in Q_0 \mbox{ sinks})}$.

We claim that the following equality holds in $L(Q)$:
$${(e_j\; |\; j\in Q_0 \mbox{ sinks})}=(e_j\; |\; j\in Q_0\backslash Q^\circ_0).$$
Denote the  ideal on the left hand side by $I$. Clearly, $I$ lies in the ideal on the right hand side.
 Then it suffices to show that for each $j\in Q_0\backslash Q^\circ_0$, $e_j$ belongs to $I$. By the construction of $Q^\circ$, we have a filtration of full subquivers
$$Q^\circ=F_NQ \subseteq F_{N-1}Q\subseteq \cdots \subseteq F_1Q\subseteq F_0Q=Q,$$
such that each $F_nQ$ is obtained from $F_{n-1}Q$ by removing all the sinks. For each $j\in Q_0\backslash Q^\circ_0$, we define its \emph{height}, denoted by $h(j)$, to be the maximal $h$ satisfying $j\in (F_hQ)_0$.

We use induction on $h(j)$ to prove that $e_j$ belongs to $I$ for any $j\in Q_0\backslash Q^\circ_0$. We observe that $h(j)=0$ if and only if $j$ is a sink in $Q$. Then the case $j=0$ is trivial. Now assume that $h(j)=h>0$; in particular, $j$ is not a sink in $Q$. Any arrow $\alpha$ starting at $j$ necessarily satisfies $h(t(\alpha))<h$.  By the induction hypothesis, we have $e_{t(\alpha)}\in I$. The second Cuntz-Krieger relation yields
$$e_j=\sum_{\{\alpha\in Q_1\; |\; s(\alpha)=j\}}\alpha^*\alpha=\sum_{\{\alpha\in Q_1\; |\; s(\alpha)=j\}}\alpha^*e_{t(\alpha)}\alpha.$$
Therefore, $e_j$ belongs to the two-sided ideal $I$, proving the claim.

By the claim above, the quotient algebra $L(Q)/I$ equals $L(Q)/{(e_j\; |\; j\in Q_0\backslash Q^\circ_0)}$, where the latter  is clearly isomorphic to $L(Q^\circ)$. Then the required isomorphism follows readily.
\end{proof}

By Proposition~\ref{prop:iso-path}(1) and the explicit construction of $C_E(M)$,  we infer that $C(Q)$ is a free $\mathbb{K}$-module with  a basis given by the following set
$$\{p^*q\; |\; p \mbox{ and } q \mbox{ are paths in } Q \mbox{ satisfying } t(p)=t(q)\}.$$
We observe that both $C(Q)$ and $L(Q)$ is naturally $\mathbb{Z}$-graded  such that $|e_i|=0$, $|\alpha|=-1$ and $|\alpha^*|=1$.

\begin{rem}
The chosen grading of $L(Q)$ here is different from the one in \cite{Smi, CY, CLW}, where the degrees of $e_i$, $\alpha$ and $\alpha^*$ equal $0$, $1$ and $-1$, respectively.  There is an involution $(-)^* \colon L(Q) \to L(Q)$ of algebras given by $(e_i)^* = e_i, (\alpha)^*=\alpha^*$ and $(\alpha^*)^* =\alpha$. We observe that the involution identifies the two gradings on $L(Q)$. Therefore, there is no essential difference between these two gradings. In the following consideration of dg Leavitt path algebras, one sees that the grading here is more reasonable.
\end{rem}

Recall that $E=\mathbb{K}Q_0$ and $M=\mathbb{K}Q_1$. We fix an associative morphism $\mu\colon M\otimes_E M \rightarrow M$ of $E$-$E$-bimodules. Identifying $M\otimes_E M$ with $\mathbb{K}Q_2$, we might view $\mu$ as an $E$-$E$-bimodule map
$$\mu\colon \mathbb{K}Q_2\longrightarrow \mathbb{K}Q_1.$$
Consequently, it is uniquely determined by the following formula: for each path $p$ of length two in $Q$, we have
\begin{align}\label{equ:struc-co}
\mu(p)=\sum_{\{\alpha\in Q_1\; |\; \alpha// p\}} \lambda_{p, \alpha}\alpha
\end{align}
for some structure coefficients $\lambda_{p, \alpha}\in \mathbb{K}$. Here $\alpha // p$ indicates that $\alpha$ is \emph{parallel} to $p$, i.e.\ $s(\alpha) = s(p)$ and $t(\alpha) = t(p)$. The associativity of $\mu$ is equivalent to the following condition: for each path $q=\alpha_3\alpha_2\alpha_1$ of length three and each arrow $\alpha$ parallel to $q$, we have
$$\sum_{\{\beta\in Q_1\; |\; \beta // \alpha_3\alpha_2\}} \lambda_{\alpha_3\alpha_2, \beta} \lambda_{\beta\alpha_1, \alpha}=\sum_{\{\beta'\in Q_1\; |\; \beta'//  \alpha_2\alpha_1\}} \lambda_{\alpha_2\alpha_1, \beta'}\lambda_{\alpha_3\beta', \alpha}.$$

Recall from the previous section that $C_E(M)$ is naturally a dg algebra. By the canonical isomorphism in Proposition~\ref{prop:iso-path}(1) and transferring the structures, we infer that $C(Q)$ is a dg algebra with differential $\partial$, called the \emph{dg Cohn path algebra} associated to $(Q, \mu)$. As the differential descends to $L(Q^\circ)$, we obtain the dg algebra $(L(Q^\circ), \partial)$, called the \emph{dg Leavitt path algebra} associated to $(Q, \mu)$.

\begin{rem}
By Lemma~\ref{lem:c-closed} and its proof, we observe that the differential $\partial$ on $C(Q)$ descends to $L(Q)$. Then the differential of $L(Q^\circ)$ is inherited from the one of $L(Q)$ via the isomorphism
$$L(Q)/{(e_j\; |\; j\in Q_0 \mbox{ sinks})}\simeq L(Q^\circ).$$
By Proposition~\ref{prop:iso-path}(2), the dg Leavitt path algebra $(L(Q^\circ), \partial)$, rather than $(L(Q), \partial)$, is more relevant to us.
\end{rem}

Recall that the differential $\partial$ of $C(Q)$ is completely determined by $\partial_+$ and $\partial_{-}$; see (\ref{equ:par+}) and (\ref{equ:par-}). Both maps are  uniquely determined by the structure coefficients $\lambda_{p, \alpha}$ in (\ref{equ:struc-co}).

To make them explicit, we use the identification $M^*=\mathbb{K}Q^*_1$. Moreover, we identify $M^*\otimes_E M^*$ with $\mathbb{K}Q_2^*$, sending a typical tensor $\alpha^*\otimes_E \beta^*$ of ghost arrows to $(\beta\alpha)^*=\alpha^*\beta^*\in Q_2^*$. Then we have
\begin{align}\label{equ:+}
\begin{array}{cccc}
\partial_{+}\colon & \mathbb{K}Q_1^* & \longrightarrow  & \mathbb{K}Q^*_1\otimes_E \mathbb{K}Q^*_1=\mathbb{K}Q_2^*,\\
 & \alpha^*   & \mapsto & \displaystyle \sum_{\{p\in Q_2\; |\; \alpha// p \}}\lambda_{p, \alpha}p^*
 \end{array}
 \end{align}
and
\begin{align}\label{equ:-}
\begin{array}{cccc}
\partial_{-}\colon &  \mathbb{K}Q_1 & \longrightarrow &  \mathbb{K}Q_1^*\otimes_E \mathbb{K}Q_1\subseteq C(Q), \\
&  \alpha  & \mapsto & \displaystyle \sum_{\{\beta\in Q_1\; |\; s(\beta)=t(\alpha)\}}  \beta^* \mu(\beta\alpha) =  \sum \lambda_{\beta \alpha, \beta'}\; \beta^* \beta'.
\end{array}
\end{align}
where the last sum without subscript runs over ${\{\beta, \beta'\in Q_1\; |\; s(\beta)=t(\alpha),\;  \beta'//\beta\alpha \}}$.

We now give a concrete example.

\begin{exm}\label{example:quiverone}
Let $n\geq 1$ and $R_n$ be the rose quiver with one vertex and $n$ loops.

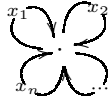
\[\xymatrix{
\cdot \ar@(l,u)@<+2pt>[]|{x_1} \ar@(u,r)@<+2pt>[]|{x_2}
\ar@(d,l)@<+2pt>[]|{x_n} \ar@(r,d)@<-2pt>[]|{\cdots}
}\]
Then $E=\mathbb{K}$ and $M=\bigoplus_{i=1}^n\mathbb{K}x_i$. Define a $\mathbb{K}$-linear product $\mu\colon M\otimes_\mathbb{K}M\rightarrow M$ according to the following rule:
\begin{align*}
\mu(x_i\otimes x_j) =
\begin{cases}
x_{i+j}, & \text{ if $i+j\leq n$}; \\
0, & \text{otherwise.}
\end{cases}
\end{align*}
We observe that $\mu$ is associative.

The dg Cohn path algebra and dg Leavitt path algebra associated to $(R_n, \mu)$ are described as follows:
$$C(R_n)=\mathbb{K}\langle x_1, \dotsc, x_n, y_1, \dotsc, y_n\rangle/{(x_iy_j-\delta_{i, j}\; |\; 1\leq i, j \leq n)}$$
and
$$L(R_n)=\mathbb{K}\langle x_1, \dotsc, x_n, y_1, \dotsc, y_n\rangle/{(x_iy_j-\delta_{i, j}, \ 1-\sum_{k=1}^n y_kx_k \; |\; 1\leq i, j \leq n)}.$$
Both algebras are graded such that $|x_i|=-1$ and $|y_i|=1$. Here, we write $y_i$ for the ghost arrow $x_i^*$. The differential $\partial$ on both algebras is uniquely determined by the following formula: for each $1\leq i\leq n$, we have
$$\partial(y_i)=\sum_{1\leq j\leq i-1}y_jy_{i-j} \quad \mbox{ and } \quad \partial(x_i)=\sum_{i<j\leq n} y_{j-i}x_j.$$
In particular, we have $\partial(y_1)=0=\partial(x_n)$.

The algebras $C(R_n)$ and $L(R_n)$ are known as the (classical) Cohn algebra and Leavitt algebra, respectively. We mention that $C(R_1)$ is also called the Toeplitz-Jacobson algebra; see \cite[Proposition~1.3.7]{AAM} and \cite{IS}. Moreover, the dg algebra $L(R_1)$ has the trivial differential, and  is isomorphic to the graded Laurent polynomial algebra $\mathbb K[y, y^{-1}]$ in  one variable, where $y$ has degree $1$ and $y^{-1}$ has degree $-1$.
\end{exm}

\section{Pretriangulated dg categories}\label{section5}

In this section, we recall some basic facts on dg categories. We are mainly concerned with pretriangulated dg categories, dg quotient categories and the perfect dg derived categories of dg algebras. The main references are \cite{Dri, Kel06}.

\subsection{DG categories and functors}  Let $\mathcal{C}$ be a dg category. For two objects $X$ and $Y$, its Hom complex  is usually denoted by $\mathcal{C}(X, Y)=(\bigoplus_{p\in \mathbb{Z}}\mathcal{C}(X, Y)^p, d_\mathcal{C})$. Morphisms in $\mathcal{C}(X, Y)^p$ are said to be homogeneous of degree $p$. A morphism $f\colon X\rightarrow Y$ is said to be closed, if $d_\mathcal{C}(f)=0$.

We denote by $Z^0(\mathcal{C})$ the ordinary category of $\mathcal{C}$, which has the same objects as $\mathcal{C}$ and whose morphisms are precisely closed morphisms in $\mathcal{C}$ of degree zero, that is, its Hom modules are  the zeroth cocycles of the corresponding Hom complexes. Similarly, the \emph{homotopy category} $H^0(\mathcal{C})$ has the same objects and its Hom modules are given by the zeroth cohomology of the corresponding Hom complexes.  An object $X$ is \emph{contractible} in $\mathcal{C}$ if ${\rm Id}_X$ is a coboundary, or equivalently, $X$ is isomorphic to the zero object in $H^0(\mathcal{C})$.

We denote by $\mathcal{C}^{\rm op}$ the \emph{opposite} dg category of $\mathcal{C}$, whose composition $\circ^{\rm op}$ is given by $g\circ^{\rm op} f =(-1)^{|g||f|}f\circ g$.

A closed morphism $f\colon X\rightarrow Y$ of degree zero is called a \emph{dg-isomorphism}, if it is an isomorphism in $Z^0(\mathcal{C})$, or equivalently in $\mathcal{C}$; it is called a \emph{homotopy equivalence}, if its image in $H^0(\mathcal{C})$ is an isomorphism.

In the following examples, we fix the notation which will be used later. For a $\mathbb{K}$-algebra $\Lambda$, we denote by $\Lambda\mbox{-Mod}$ the abelian category of left $\Lambda$-modules.

\begin{exm}\label{exm:Cdg}
Let $\Lambda$ be a $\mathbb{K}$-algebra. For two complexes $X$ and $Y$ of $\Lambda$-modules, we denote by ${\rm Hom}_\Lambda(X, Y)$ the Hom complex  given as follows: its $p$-th homogeneous component is given by an infinite product
   $${\rm Hom}_\Lambda(X, Y)^p=\prod_{n \in \mathbb{Z}} {\rm Hom}_{\Lambda\mbox{-}{\rm Mod}}(X^n, Y^{n+p}),$$
 whose elements  will be denoted by $f=\{f^n\}_{n\in \mathbb{Z}}$ with $f^n \in {\rm Hom}_{\Lambda\mbox{-}{\rm Mod}}(X^n, Y^{n+p})$; the differential $d$ acts on $f$ via
 $$d(f)^{n}=d_Y^{n+p}\circ f^n-(-1)^{|f|}f^{n+1}\circ d_X^n, \quad \quad \text{for each $n\in \mathbb{Z}.$}$$

 The collection of all complexes of $\Lambda$-modules with these Hom complexes yields a dg category, denoted by $C_{\rm dg}(\Lambda\mbox{-}{\rm Mod})$. We observe that $Z^0C_{\rm dg}(\Lambda\mbox{-}{\rm Mod})=C(\Lambda\mbox{-{\rm Mod}})$ is the category of complexes of $\Lambda$-modules and that $H^0C_{\rm dg}(\Lambda\mbox{-}{\rm Mod})=\mathbf{K}(\Lambda\mbox{-{\rm Mod}})$ is the classical homotopy category of complexes of $\Lambda$-modules.

 Let $\mathfrak{a}$ be an additive category. Slightly generalizing the above construction, we obtain the dg category $C_{\rm dg}(\mathfrak{a})$ of complexes in $\mathfrak{a}$. The homotopy category $H^0C_{\rm dg}(\mathfrak{a})$ equals $\mathbf{K}(\mathfrak{a})$, the classical homotopy category of complexes in $\mathfrak a$.
     \end{exm}

The dg category $C_{\rm dg}(\mathbb{K}\mbox{-}{\rm Mod})$ is usually denoted by $C_{\rm dg}(\mathbb{K})$.

\begin{exm}
Let $\mathcal{C}$ and $\mathcal{D}$ be two dg categories. Assume that $\mathcal{C}$ is small. For two dg functors $F, G\colon \mathcal{C}\rightarrow \mathcal{D}$, a natural transformation $\eta=(\eta_X)_{X\in {\rm Obj}(\mathcal{C})} \colon F\rightarrow G$ of degree $p$ consists of morphisms $\eta_X\colon F(X)\rightarrow G(X)$ of degree $p$ in $\mathcal{D}$ satisfying  the following graded naturality property: for any morphism $a\colon X\rightarrow X'$ in $\mathcal{C}$, we have
$$G(a)\circ \eta_X=(-1)^{p|a|}\eta_{X'}\circ F(a) \colon F(X) \to G(X').$$
We now define the Hom complex ${\rm Hom}(F, G)$ such that its $p$-th component is formed by natural transformations of degree $p$ from $F$ to $G$ and that its differential is given by $d(\eta)_{X}=d_\mathcal{D}(\eta_X)$ for any object $X\in \mathcal{C}$.

The collection of all dg functors from $\mathcal{C}$ to $\mathcal{D}$ together with the Hom complexes yields a dg category, denoted by ${\rm Fun}_{\rm dg}(\mathcal{C}, \mathcal{D}).$ In particular, a natural transformation $\eta = (\eta_X)_{X\in {\rm Obj(\mathcal{C})}}$ is {\rm closed} if $d_{\mathcal D}(\eta_X) = 0$ for any object $X \in \mathcal C$.

 By a left dg $\mathcal{C}$-module, we mean a dg functor $M\colon \mathcal{C}\rightarrow C_{\rm dg}(\mathbb{K})$. Write
$$\mathcal{C}\mbox{-{\rm DGMod}} = {\rm Fun}_{\rm dg}(\mathcal{C}, C_{\rm dg}(\mathbb{K}))$$
for the dg category formed by left dg $\mathcal{C}$-modules. Write $\mathbf{K}(\mathcal{C})=H^0(\mathcal{C}\mbox{-{\rm DGMod}})$ for the homotopy category of dg $\mathcal{C}$-modules. Denote by $\mathbf{D}(\mathcal{C})=\mathbf{K}(\mathcal{C})/{\mathbf{K}^{\rm ac}(\mathcal{C})}$ the derived category of dg $\mathcal{C}$-modules, where $\mathbf{K}^{\rm ac}(\mathcal{C})$ is the triangulated subcategory of $\mathbf{K}(\mathcal{C})$ formed by acyclic modules, and $\mathbf{K}(\mathcal{C})/{\mathbf{K}^{\rm ac}(\mathcal{C})}$ means the Verdier quotient.
\end{exm}

 For a left dg $\mathcal{C}$-module $M$ and each $i\in \mathbb{Z}$, the shifted $\mathcal C$-module $\Sigma^i(M)$ is defined as follows: as a complex
 $$\Sigma^i(M)(X)=\Sigma^i(MX), \quad \quad \text{for each object $X$ in $\mathcal C$}$$
 and for any morphism $f\colon X\rightarrow X'$ in $\mathcal{C}$, the induced map $\Sigma^i(M)(f)\colon \Sigma^i(MX)\rightarrow \Sigma^i(MX')$ sends $m$ to $(-1)^{i|f|}M(f)(m)$.
 Here, for simplicity,  we write $MX = M(X)$ and $MX' = M(X')$.
 For a closed morphism $\eta\colon M\rightarrow N$ of degree zero between dg modules, its \emph{cone} ${\rm Cone}(\eta)$ is a dg module defined as follows:
 $${\rm Cone}(\eta)(X):={\rm Cone}(\eta_X)=NX\oplus \Sigma(MX)$$
  is the mapping cone of $\eta_X\colon MX\rightarrow NX$, and ${\rm Cone}(\eta)(f)$ is given by $\left(\begin{smallmatrix} N(f) & 0\\ 0 & \Sigma(M)(f)\end{smallmatrix}\right)$.

A dg functor $F\colon \mathcal{C}\rightarrow \mathcal{D}$ is said to be quasi-fully faithful if for any objects $X$ and $Y$, the induced cochain map
$$\mathcal{C}(X, Y)\longrightarrow \mathcal{D}(FX, FY)$$
is a quasi-isomorphism. Consequently, $H^0(F)\colon H^0(\mathcal{C})\rightarrow H^0(\mathcal{D})$ is fully faithful. The dg functor $F$ is called a \emph{quasi-equivalence}, provided that it is quasi-fully faithful and $H^0(F)$ is essentially surjective.

We denote by $\mathbf{dgcat}$ the category of small dg categories, whose morphisms are dg functors.  The \emph{homotopy category} $\mathbf{Hodgcat}$ is the localization of $\mathbf{dgcat}$ with respect to all the quasi-equivalences. In other words, $\mathbf{Hodgcat}$ is obtained from $\mathbf{dgcat}$ by formally inverting quasi-equivalences. By the model structure \cite{Tab} on $\mathbf{dgcat}$, the morphisms between two objects in $\mathbf{Hodgcat}$ form a set; compare \cite[Appendix B.4-6]{Dri}.

For dg categories $\mathcal{C}$ and $\mathcal{D}$, a morphism in $\mathbf{Hodgcat}$ between them is sometimes called a \emph{dg quasi-functor}. It can be realized as a roof
$$\mathcal{C} \stackrel{F}\longleftarrow \mathcal{C}' \stackrel{F'}\longrightarrow \mathcal{D}$$
 of dg functors, where $F$ is a quasi-equivalence; moreover, $F$ can be taken as a semi-free resolution of $\mathcal{C}$; moreover, the dg quasi-functor is an isomorphism if and only if the dg functor $F'$ is a quasi-equivalence. For details, we refer to \cite{Tab} and \cite[Appendix~B.5]{Dri}.

 In the sequel, we abuse a dg quasi-functor with a genuine dg functor, and abuse isomorphisms in $\mathbf{Hodgcat}$ with quasi-equivalences. In practice, we will relax the smallness assumption by  the following remark; compare \cite[Remark~1.22 and Appendix~A]{LO}.

\begin{rem}
When we consider complexes or dg modules possibly without finite generation conditions, we usually encounter non-small dg categories.  Then we have to choose a universe $\mathbb{U}$ and restrict ourselves to $\mathbb{U}$-small complexes or dg modules; compare \cite[Section~2]{Toe} and \cite[Subsection~4.4, p.172]{Kel06}. This allows us to treat them equally in the framework of the homotopy category $\mathbf{Hodgcat}$.
 \end{rem}

\subsection{Exact and pretriangulated dg categories}

Let $\mathcal{C}$ be  a small dg category. Consider the Yoneda embedding
$$\mathbf{Y}_\mathcal{C}\colon \mathcal{C}\longrightarrow \mathcal{C}^{\rm op}\mbox{-DGMod}, \; X\mapsto \mathcal{C}(-, X).$$
It is a fully faithful dg functor, which induces a fully faithful functor
$$H^0(\mathbf{Y}_\mathcal{C})\colon H^0(\mathcal{C})\longrightarrow \mathbf{K}(\mathcal{C}^{\rm op}).$$

Recall from \cite[Section~2]{Kel99} that the dg category $\mathcal{C}$ is \emph{exact} (= strongly pretriangulated in the sense of \cite{BK}) if the essential image of $\mathbf{Y}_\mathcal{C}$ is closed under shifts and cones. In other words, all the shifted modules $\Sigma^i\mathcal{C}(-, X)$ and cones ${\rm Cone}(\mathcal{C}(-, f))$ are dg-representable for any object $X$ and any closed morphism $f$ of degree zero.

The following internal characterization of exact dg categories is well known; see \cite[Section~3]{BLL}.

\begin{lem}\label{lem:exactdg}
Let $\mathcal{C}$ be a small dg category. Then $\mathcal{C}$ is exact if and only if the following two conditions are satisfied:
\begin{enumerate}
\item[(1)] the internal shifts of objects exist, that is, for each object $X$, there exist two objects $X_1$ and $X_2$ with  two closed isomorphisms $ X\rightarrow X_1$ and $X_2\rightarrow X$ in $\mathcal{C}$ of degree one;
\item[(2)] the internal cones of morphisms exist, that is,  for each closed morphism $f\colon X\rightarrow X'$ of degree zero, there is a diagram in $\mathcal{C}$
    \[
\xymatrix{ X' \ar[rr]^-{j} && Z \ar@/^1pc/@{.>}[ll]^-{t} \ar[rr]^-{p} &&  X \ar@/^1pc/@{.>}[ll]^-{s} }\]
with $|j|=0=|t|$, $|p|=1$ and $|s|=-1$ subject to the following identities:
$$p\circ j=0=t\circ s,\ \ {\rm Id}_Z=s\circ p+j\circ t,\ \  {\rm Id}_{X'}=t\circ j, \ \ {\rm Id}_{X}=p\circ s$$
and
$$ d_\mathcal{C}(j)=0=d_{\mathcal{C}}(p),\ \ f=t\circ d_{\mathcal{C}}(s).$$
\end{enumerate}
\end{lem}

The dg category $\mathcal{C}$ is \emph{pretriangulated} \cite{BK} if all the shifted modules $\Sigma^i\mathcal{C}(-, X)$ and cones ${\rm Cone}(\mathcal{C}(-, f))$ are homotopy equivalent to representable functors, or equivalently, the essential image of $H^0(\mathbf{Y}_\mathcal{C})$ is a triangulated subcategory of $\mathbf{K}(\mathcal{C}^{\rm op})$. Consequently, for a pretriangulated dg category $\mathcal{C}$, its homotopy category $H^0(\mathcal{C})$ has a canonical triangulated structure in the following sense: for any dg functor $F\colon \mathcal{C}\rightarrow \mathcal{D}$ between pretriangulated dg categories, the induced functor $H^0(F)\colon H^0(\mathcal{C})\rightarrow H^0(\mathcal{D})$ is naturally a triangle functor. We mention that an exact dg category is clearly pretriangulated.

For any dg category $\mathcal{C}$, its \emph{exact hull} means an exact dg category $\mathcal{C}^{\rm ex}$ with a fully faithful dg functor
$${\rm can}_\mathcal{C}\colon \mathcal{C}\longrightarrow \mathcal{C}^{\rm ex},$$
which induces a dg-equivalence
\begin{align}\label{align:exacthull}
{\rm Fun}_{\rm dg}(\mathcal{C}^{\rm ex}, \mathcal{D})\longrightarrow {\rm Fun}_{\rm dg}(\mathcal{C}, \mathcal{D}), \; F\mapsto F\circ {\rm can}_\mathcal{C}
\end{align}
for any exact dg category $\mathcal{D}$. Indeed, one might take $\mathcal{C}^{\rm ex}$ to be the smallest full dg subcategory of $\mathcal{C}^{\rm op}\mbox{-DGMod}$ containing the representable functors and closed under shifts and cones; then ${\rm can}_\mathcal{C}$ is given by the Yoneda embedding. For an explicit construction of the exact hull, we refer to \cite{BK} and \cite[Subsection~2.4]{Dri}.

 The following facts are standard. The first statement implies that pretriangulated dg categories are invariant under quasi-equivalences. In contrast, exact dg categories usually are not invariant under quasi-equivalences.

 \begin{lem}\label{lem:pretri}
 Let $\mathcal{C}$ and $\mathcal{D}$ be two small dg categories. Then the following hold.
 \begin{enumerate}
 \item Let $F\colon \mathcal{C}\rightarrow \mathcal{D}$ be a quasi-equivalence. Then $\mathcal{C}$ is pretriangulated if and only if so is $\mathcal{D}$.
 \item Assume that $\mathcal{D}$ is exact and that  $F\colon \mathcal{C}\rightarrow \mathcal{D}$ is a fully-faithful dg functor with  $H^0(F)$ essentially surjective. Then $\mathcal{C}$ is pretriangulated.
 \end{enumerate}
 \end{lem}

\begin{proof}
For (1), we consider the following diagram, which is commutative up to isomorphism.
\[\xymatrix{
H^0(\mathcal{C}) \ar[d]_-{H^0(\mathbf{Y}_\mathcal{C})}\ar[rr]^{H^0(F)} && H^0(\mathcal{D}) \ar[d]^-{H^0(\mathbf{Y}_\mathcal{D})}\\
\mathbf{K}(\mathcal{C}^{\rm op}) \ar[rr]^{-\otimes_\mathcal{C}\mathcal{D}} && \mathbf{K}(\mathcal{D}^{\rm op})
}\]
Here, $\mathcal{D}$ is viewed a dg $\mathcal{C}$-$\mathcal{D}$-bimodule. The vertical arrows are fully faithful and the bottom arrow is a triangle functor. As $H^0(F)$ is an equivalence, it follows immediately that $H^0(\mathcal{C})$ is a triangulated subcategory of  $\mathbf{K}(\mathcal{C}^{\rm op})$ if and only if the same holds for $\mathcal{D}$. Then (1) follows immediately.

(2) is a very special case of (1), once we observe that $\mathcal{D}$ is pretriangulated and that $F$ is a quasi-equivalence.
\end{proof}

The following elementary fact will be used often; see \cite[Lemma~2.5]{LS16} and \cite[Lemma~3.1]{CC}.

\begin{lem}\label{lem:tri-quasi-equiv}
Let $F\colon \mathcal{C}\rightarrow \mathcal{D}$ be a dg functor between two pretriangulated dg categories. Assume that $H^0(F)$ is a triangle equivalence. Then $F$ is a quasi-equivalence.
\end{lem}

Let $\mathcal{C}$ be a small dg category. For a full dg subcategory $\mathcal{D}$, we denote by $\mathcal{C}/\mathcal{D}$ the corresponding dg quotient \cite{Kel99,Dri}. Denote by $q\colon \mathcal{C}\rightarrow \mathcal{C}/\mathcal{D}$ the quotient functor, which acts on objects by the identity.

 When all the Hom complexes in $\mathcal{C}$ are homotopically flat over $\mathbb{K}$, an explicit construction of $\mathcal{C}/\mathcal{D}$ by freely adjoining contracting homotopies is given in \cite[Subsection~3.1]{Dri}; compare \cite[Section~4]{Kel99}. In general, we refer to \cite[Subsection~3.5]{Dri} or \cite[Subsection~3.1]{Tab10}. More precisely, we replace $\mathcal{C}$ by its semi-free resolution $\widetilde{C}$, and $\mathcal{D}$ by the corresponding full dg subcategory $\widetilde{\mathcal{D}}$ of $\widetilde{C}$. The Hom complexes in $\widetilde{C}$ are semi-free over $\mathbb{K}$, and thus homotopically flat. Then $\mathcal{C}/\mathcal{D}$ is defined to $\widetilde{\mathcal{C}}/{\widetilde{\mathcal{D}}}$, where the latter is constructed explicitly in \cite[Subsection~3.1]{Dri}. Therefore, strictly speaking, $q\colon \mathcal{C}\rightarrow \mathcal{C}/\mathcal{D}$ is a dg quasi-functor, which is not necessarily a genuine dg functor. In other words, to study dg quotient categories, one has to work in the homotopy category $\mathbf{Hodgcat}$; see \cite{Tab10}.

The following universal property of $q$ is due to \cite[Theorem~1.6.2(ii)]{Dri}; compare \cite[Theorem~4.8]{Kel06}. For a cleaner version, we refer to \cite[Theorem~4.0.1]{Tab10}.

\begin{lem}\label{lem:q-universal}
Assume that $F\colon \mathcal{C}\rightarrow \mathcal{C}'$ is a dg functor such that $F(D)$ is contractible for any object $D$ in $\mathcal{D}$. Then there is a unique morphism $\overline{F}\colon \mathcal{C}/\mathcal{D}\rightarrow \mathcal{C}'$ in $\mathbf{Hodgcat}$ satisfying $F=\overline{F}\circ q$.
\end{lem}

The following fundamental fact will be used frequently; see \cite[Theorem~3.4]{Dri} and \cite[Theorem~1.3(i) and Lemma~1.5]{LO}. The homotopically flatness conditions required in \cite[Theorem~3.4]{Dri} are not essential, because we might replace $\mathcal{C}$ by  its semi-free resolution $\widetilde{\mathcal{C}}$,  on which the homotopically flatness conditions hold automatically; compare \cite[Subsection~3.5]{Dri}.

\begin{lem}\label{lem:quo}
Assume that both $\mathcal{C}$ and $\mathcal{D}$ are pretriangulated. Then $\mathcal{C}/\mathcal{D}$ is pretriangulated. Moreover, the quotient functor $q$ induces  an isomorphism of triangulated categories
        $$H^0(\mathcal{C})/H^0(\mathcal{D})\stackrel{\sim}\longrightarrow H^0(\mathcal{C}/\mathcal{D}),$$
        where $H^0(\mathcal{C})/H^0(\mathcal{D})$ denotes the corresponding Verdier quotient.
\end{lem}

\subsection{The perfect dg derived categories}

Recall that an additive category $\mathfrak{a}$ is \emph{idempotent-split} provided that each idempotent morphism $e\colon X\rightarrow X$ admits a factorization $X\stackrel{u}\rightarrow Y \stackrel{v}\rightarrow X$ satisfying $u\circ v={\rm Id}_Y$.

Let $\mathcal{T}$ be a triangulated category. For any set $S$ of objects, we denote by ${\rm thick}\langle S\rangle$ its \emph{thick hull}, that is, the smallest triangulated subcategory of $\mathcal{T}$ containing $S$ and closed under direct summands. An object $X$ is said to be a \emph{generator} of $\mathcal{T}$, provided that $\mathcal{T}={\rm thick}\langle X\rangle$.

Let $A$ be a dg algebra. We will always view $A$ as a dg category with a single object. Then we have the dg category $A\mbox{-DGMod}$ of left dg $A$-modules and the homotopy category $\mathbf{K}(A)=H^0(A\mbox{-DGMod})$. The thick hull ${\rm thick}\langle A\rangle$ of $A$ is usually denoted by $\mathbf{per}(A)$, whose objects are called \emph{perfect modules}.  Denote by $\mathbf{per}_{\rm dg}(A)$ the full dg subcategory of $A\mbox{-DGMod}$ formed by perfect $A$-modules, called the \emph{perfect dg derived category} of $A$.

The following result is implicitly contained in \cite[Subsection~4.2]{Kel94}. Recall that $A^{\rm op}$ denotes the opposite dg algebra of $A$.

\begin{prop}\label{prop:quasi-equiv}
Let $\mathcal{C}$ be a pretriangulated dg category. Assume that $H^0(\mathcal{C})$ is idempotent-split and that $X$ is a generator of $H^0(\mathcal{C})$. Then there is a quasi-equivalence
$$\mathcal{C}(X, -)\colon \mathcal{C}\stackrel{\sim}\longrightarrow \mathbf{per}_{\rm dg}(\mathcal{C}(X, X)^{\rm op}).$$
\end{prop}

\begin{proof}
Write $A=\mathcal{C}(X, X)^{\rm op}$ and $F=\mathcal{C}(X, -)\colon \mathcal{C}\rightarrow A\mbox{-DGMod}$. We observe that $F(X)=A$.

Consider the triangle functor $H^0(F)\colon H^0(\mathcal{C})\rightarrow \mathbf{K}(A)$.  Since $X$ generates $H^0(\mathcal{C})$, we infer that the essential image of $H^0(F)$ lies in $\mathbf{per}(A)$. By the Yoneda embedding, $F$ induces a quasi-isomorphism
$$\mathcal{C}(X, X) \longrightarrow A\mbox{-DGMod}(F(X), F(X)).$$
The above complexes compute $H^0(\mathcal{C})(X, \Sigma^n(X))$ and ${\rm Hom}_{\mathbf{K}(A)}(F(X), \Sigma^nF(X))$, respectively. We conclude that $H^0(F)$ induces an isomorphism
 $$H^0(\mathcal{C})(X, \Sigma^n(X)) \simeq {\rm Hom}_{\mathbf{K}(A)}(F(X), \Sigma^nF(X)), \quad \text{for each $n \in \mathbb Z$}.$$
 Since $X$ generates $H^0(\mathcal{C})$, we infer from \cite[Lemma~1]{Bei} that $H^0(F)$ is fully faithful. Since $H^0(\mathcal{C})$ is idempotent-split, we infer that
 $$H^0(F)\colon H^0(\mathcal{C})\longrightarrow \mathbf{per}(A)$$
 is a triangle equivalence. Then the required quasi-equivalence follows immediately from Lemma~\ref{lem:tri-quasi-equiv}.
\end{proof}

Let $\mathcal{T}$ be a triangulated category with arbitrary coproducts. An object $X$ is \emph{compact} if ${\rm Hom}_\mathcal{T}(X, -)$ commutes with arbitrary coproducts. Denote by $\mathcal{T}^c$ the full subcategory formed by compact objects; it is a thick triangulated subcategory. In particular, $\mathcal{T}^c$ is always idempotent-split.  The triangulated category $\mathcal{T}$ is \emph{compactly generated}, provided that  there is a set $\mathcal{S}$ of compact objects such that each nonzero object $X$ satisfies ${\rm Hom}_\mathcal{T}(\Sigma^i(S), X)\neq 0$ for some $S\in \mathcal{S}$ and $i\in \mathbb{Z}$. As a typical example, the derived category $\mathbf{D}(\mathcal{C})$ of dg modules over a small dg category $\mathcal{C}$ is compactly generated; moreover, we have
\begin{align}\label{equ:der-compact}
\mathbf{D}(\mathcal{C})^c={\rm thick}\langle \mathcal{C}(X, -) \mid X\in {\rm Obj}(\mathcal{C})\rangle.
\end{align}
For details, we refer to \cite[Subsection~5.3]{Kel94}.

For each dg algebra $A$, we have an inclusion $A\hookrightarrow \mathbf{per}_{\rm dg}(A)^{\rm op}$ of dg categories, which sends the unique object to the dg $A$-module $A$ itself. We have the restriction
$${\rm res}\colon \mathbf{D}(\mathbf{per}_{\rm dg}(A)^{\rm op})\longrightarrow \mathbf{D}(A), \quad M\mapsto M(A)$$
along the above inclusion.

The following result is a special case of \cite[Theorem~8.1]{Kel94}. We sketch a proof for the convenience of the reader.

\begin{lem}\label{lem:der-equi}
The above restriction functor is a triangle equivalence.
\end{lem}

\begin{proof}
Write $\mathcal{C}=\mathbf{per}_{\rm dg}(A)$. As $\mathcal{C}$ is exact and $H^0(\mathcal{C})$ is idempotent-split, the Yoneda embedding $\mathbf{Y}_\mathcal{C}$ allows us to identify $H^0(\mathcal{C})$ with $\mathbf{D}(\mathcal{C}^{\rm op})^c$; see (\ref{equ:der-compact}).

The restriction functor `$\rm res$' preserves infinite coproducts. Therefore, it suffices to prove that it preserves compact objects and restricts to an equivalence between the full subcategories formed by compact objects.

Since  `${\rm res}$' sends a representable functor $\mathcal{C}(-, P)$ to $\mathcal{C}(A, P)=P$, it follows that  it preserves compact objects. Moreover, the following composition is the identity functor.
\[
H^0(\mathcal{C}) \xrightarrow{\sim}  \mathbf{D}(\mathcal{C}^{\rm op})^c \xrightarrow{\rm res} \mathbf{D}(A)^c \xrightarrow{\sim } \mathbf{per}(A) = H^0(\mathcal{C})
\]
We infer that `${\rm res}$' restricts to an equivalence between  $\mathbf{D}(\mathcal{C}^{\rm op})^c$ and $ \mathbf{D}(A)^c$.
 \end{proof}

\section{An explicit dg localization}\label{section6}

We introduce an explicit dg localization.  Throughout this section, we fix a triple  $(\mathcal{C}, \Omega, \theta)$. Here, $\mathcal C$ is a dg category, $\Omega\colon \mathcal{C}\rightarrow \mathcal{C}$ is a dg endofunctor, and $\theta\colon {\rm Id}_\mathcal{C}\rightarrow \Omega$ is a closed natural transformation of degree zero satisfying $\theta \Omega = \Omega \theta$.

In the setup, for any object $X$, we have $d_\mathcal{C}(\theta_X)=0$,   $|\theta_X|=0$ and \ $\theta_{\Omega X} = \Omega(\theta_X) \in {\mathcal C}(\Omega X, \Omega^2(X))$. Moreover, $\theta_X$ is natural in $X$.  Here, for each $p \geq 1$, we denote by $\Omega^p$ the $p$-th iterated composition of $\Omega$. Set $\Omega^0={\rm Id}_\mathcal{C}$.

We will define a new dg category $\mathcal{SC}$ as follows: its objects are the same as $\mathcal{C}$ and the Hom complexes are given by
\begin{align}\label{align:colimit}
\mathcal{SC}(X, Y)=\varinjlim_{p\geq 0} \; \mathcal{C}(X, \Omega^p(Y)),
\end{align}
where $\mathcal{C}(X, \Omega^p(Y))\rightarrow \mathcal{C}(X, \Omega^{p+1}(Y))$ sends $f$ to $\theta_{\Omega^p(Y)}\circ f$.  For each $f\in \mathcal{C}(X, \Omega^p(Y))$, its image in $\mathcal{SC}(X, Y)$ is denoted by $[f;p]$. Therefore, we have
\begin{align}\label{equ:colim}
[f;p]=[\theta_{\Omega^{p}(Y)}\circ f;p+1].
\end{align}
The degree of $[f;p]$ equals the one of $f$. The differential of $\mathcal{SC}(X, Y)$ is given  such that $d([f;p])=[d_\mathcal{C}(f);p]$. For two morphisms $[f;p]\colon X\rightarrow Y$ and $[g;q]\colon Y\rightarrow Z$, their composition is given by
$$[g; q]\circ [f;p]:=[\Omega^p(g)\circ f; p+q].$$
One verifies that the composition is well defined, and that $\mathcal{SC}$  is a dg category. In particular,  $[{\rm Id}_X; 0]$ is the identity of $X$ in $\mathcal{SC}$.  We mention that the above construction resembles the one in \cite[Subsection~5.1]{Kel05}.

\begin{lem}\label{lem:SC}
Keep the notation as above. Then the following  statements hold.
\begin{enumerate}
\item[(1)] For each object $X$, the morphism $[\theta_X; 0]\colon X\rightarrow \Omega(X)$ is a dg-isomorphism in $\mathcal{SC}$ with $[\theta_X; 0]^{-1}=[{\rm Id}_{\Omega(X)}; 1]$.
\item[(2)] For any $p\geq 0$ and morphism $f\colon X\rightarrow \Omega^p(Y)$ in $\mathcal{C}$, we have
$$[f; p]=[\theta_Y;0]^{-1}\circ [\theta_{\Omega(Y)}; 0]^{-1}\circ  \cdots \circ [\theta_{\Omega^{p-1}(Y)};0]^{-1}\circ [f;0].$$
\item[(3)] An object $X$ is contractible in $\mathcal{SC}$ if and only if the morphism $\theta_{\Omega^{n}(X)}\circ \cdots \circ \theta_{\Omega(X)}\circ \theta_X\in \mathcal{C}(X, \Omega^{n+1}(X))$ is a coboundary for some $n$.
\end{enumerate}
\end{lem}

\begin{proof}
(1) follows from the following direct computations:
$$[{\rm Id}_{\Omega(X)}; 1]\circ [\theta_X;0]=[\theta_X;1]=[{\rm Id}_X; 0],$$
and
$$[\theta_X;0]\circ [{\rm Id}_{\Omega(X)}; 1]=[\Omega(\theta_X);1]=[\theta_{\Omega(X)};1]=[{\rm Id}_{\Omega(X)};0]. $$
Here, in both identities we use (\ref{equ:colim}); moreover, in the second identity, we use the assumption $\Omega\theta=\theta\Omega$.

(2) By (1), the right hand side of the required identity equals
$$[{\rm Id}_{\Omega(Y)};1]\circ [{\rm Id}_{\Omega^2(Y)}; 1]\circ  \cdots \circ [{\rm Id}_{\Omega^{p}(Y)};1]\circ [f;0].$$
This composition equals $[f;p]$.

(3) Assume that $X$ is contractible in $\mathcal{SC}$, that is, there is a morphism $[f;p]$ of degree $-1$ such that
$$[{\rm Id}_X;0]=d([f;p])=[d_\mathcal{C}(f);p],$$
where $f\colon X\rightarrow \Omega^p(X)$ is of degree $-1$. From the colimit construction \eqref{align:colimit}, the identity $[{\rm Id}_X;0]=[d_\mathcal{C}(f);p]$ means that there is a sufficiently large $n$ such that the following identity holds in $\mathcal{C}$
 \begin{align*}
 \theta_{\Omega^{n}(X)}\circ \cdots \circ \theta_{\Omega(X)}\circ \theta_X \circ {\rm Id}_X & =\theta_{\Omega^{n}(X)}\circ \cdots \circ \theta_{\Omega^p(X)} \circ d_\mathcal{C}(f)\\
 &=d_\mathcal{C}(\theta_{\Omega^{n}(X)}\circ \cdots \circ \theta_{\Omega^p(X)} \circ f).
 \end{align*}
 Here,  the second equality uses the assumption that $\theta$ is closed and of degree zero.
 This implies the ``only if'' part. Similarly, we may prove the ``if'' part. \end{proof}

There is a canonical dg functor
$$\iota\colon \mathcal{C}\longrightarrow \mathcal{SC}$$
 given by $\iota(X)=X$ and $\iota(f)=[f;0]$. By Lemma~\ref{lem:SC}(1), each morphism $\iota(\theta_X)$ is a dg-isomorphism. By the following universal property, we might call $\iota$ a (strict)  \emph{dg localization} of $\mathcal{C}$ along $\theta$; compare \cite[Subsection~8.2]{Toe} and \cite[Subsection~3.9]{Kel11}.

\begin{prop}\label{prop:dgl}
Let $F\colon \mathcal{C}\rightarrow \mathcal{D}$ be a dg functor such that for each object $X$ in $\mathcal{C}$, $F(\theta_X)$ is a dg-isomorphism in $\mathcal{D}$. Then there is a unique dg functor $F'\colon \mathcal{SC}\rightarrow \mathcal{D}$ satisfying $F= F' \circ \iota$.
\end{prop}

\begin{proof}
Let us first prove the uniqueness. By $F=F' \circ \iota $, we have that $F'(X)=F(X)$ and $F'([f;0])=F(f)$ for any object $X$ and morphism $f$ in $\mathcal{C}$.  For a  general morphism $[f;p]$ in $\mathcal{SC}$, we apply  Lemma~\ref{lem:SC}(2) to deduce
\begin{align}\label{equ:F'}
F'([f;p])=F(\theta_Y)^{-1}\circ F(\theta_{\Omega(Y)})^{-1} \circ \cdots\circ F(\theta_{\Omega^{p-1}(Y)})^{-1}\circ F(f).
\end{align}
This identity implies the uniqueness of $F'$.

To construct such a dg functor, we set $F'(X)=F(X)$ and use  (\ref{equ:F'}) to define the action of  $F'$ on morphisms. It is routine to verify that $F'$ is a well-defined dg functor and is the required one.
\end{proof}

\begin{lem}\label{lem:SC-pre}
Let $\iota\colon \mathcal{C}\rightarrow \mathcal{SC}$ be as above. If $\mathcal{C}$ is exact (resp. pretriangulated), then so is $\mathcal{SC}$.
\end{lem}

\begin{proof}
(1) Assume first that $\mathcal{C}$ is exact. To show that $\mathcal{SC}$ is exact, it suffices to verify the two conditions in Lemma~\ref{lem:exactdg}. The first condition  is clear, as $\mathcal{C}$ satisfies the same condition.

Let us verify the second condition.  We first observe that $\iota(f)=[f;0]$ has an internal cone, which is given by the internal cone of $f$ in $\mathcal{C}$. For a general morphism $[f;p]$, we just combine Lemma~\ref{lem:SC}(2)  with the following general fact: given a dg-isomorphism $h\colon X\rightarrow Y$ in a dg category $\mathcal{D}$, a closed morphism  $g\colon X'\rightarrow X$ of degree zero has an internal cone if and only if so does $h\circ g$.

(2) Assume that $\mathcal{C}$ is pretriangulated. Consider its exact hull ${\rm can}_\mathcal{C}\colon\mathcal{C}\rightarrow \mathcal{C}^{\rm ex}$. By the universal property \eqref{align:exacthull} of the exact hull, the dg endofunctor $\Omega$ extends uniquely to a dg endofunctor $\Omega^{\rm ex}$ on $\mathcal{C}^{\rm ex}$; moreover, $\theta\colon {\rm Id}_\mathcal{C}\rightarrow \Omega$ extends to $\theta^{\rm ex}\colon {\rm Id}_{\mathcal{C}^{\rm ex}}\rightarrow \Omega^{\rm ex}$, which is also closed of degree zero. Therefore, we can form the dg localization $\iota^{\rm ex}\colon \mathcal{C}^{\rm ex}\rightarrow \mathcal{S}(\mathcal{C}^{\rm ex})$. We have the following commutative diagram.
 \[
 \xymatrix{
 \mathcal{C}\ar[d]_-{\iota} \ar[rr]^-{{\rm can}_\mathcal{C}} && \mathcal{C}^{\rm ex}\ar[d]^-{\iota^{\rm ex}}\\
  \mathcal{SC} \ar[rr]^-{\mathcal{S}{\rm can}_\mathcal{C}} && \mathcal{S}(\mathcal{C}^{\rm ex})
 }\]
 The bottom dg functor is induced from ${\rm can}_\mathcal{C}$ and thus is also fully faithful.  Since $\mathcal{C}$ is pretriangulated, $H^0({\rm can}_\mathcal{C})$ is an equivalence. Applying $H^0$ to the commutative diagram, we infer  that $H^0(\mathcal{S}{\rm can}_\mathcal{C})$ is essentially surjective. By (1), we know that $\mathcal{S}(\mathcal{C}^{\rm ex})$ is exact. Applying Lemma~\ref{lem:pretri}(2) to $\mathcal{S}{\rm can}_\mathcal{C}$, we infer that $\mathcal{SC}$ is pretriangulated.
\end{proof}

We assume now that the dg category $\mathcal{C}$ is pretriangulated.  For each object $X$, we denote by  ${\rm Cone}(\theta_X)$ the cone of the image of $\theta_X$ in $H^0(\mathcal{C})$. In other words, ${\rm Cone}(\theta_X)$ is determined by the following exact triangle in $H^0(\mathcal{C})$.
 \begin{align}\label{tri:cone}
 X\stackrel{\theta_X}\longrightarrow \Omega(X) \longrightarrow {\rm Cone}(\theta_X)\longrightarrow \Sigma(X)
 \end{align}
 Here, we confuse $\theta_X$ with its image in $H^0(\mathcal{C})$.  Denote by
$${\rm thick}\langle {\rm Cone}(\theta_X)\; |\; X\in {\rm Obj}(\mathcal{C})\rangle $$
 the thick hull of these cones in $H^0(\mathcal{C})$. The full dg subcategory of $\mathcal{C}$ formed by the objects in ${\rm thick}\langle {\rm Cone}(\theta_X)\; |\; X\in {\rm Obj}(\mathcal{C})\rangle $ is denoted by $\mathcal{N}$.

The following general result shows that the dg localization is quasi-equivalent to a dg quotient. A similar idea appears implicitly in \cite[Section~7, the proof of Theorem~2]{Kel05}, which relates the dg orbit category in \cite[Subsection~5.1]{Kel05} to a dg quotient. The precise relationship between the following result and the mentioned one in \cite{Kel05} will be  explored elsewhere.

 \begin{thm}\label{thm:dgl}
 Let $\mathcal{C}$ be a pretriangulated dg category, and let $\iota\colon \mathcal{C}\rightarrow \mathcal{SC}$ be the dg localization along $\theta$. Then $\iota$ induces an isomorphism in $\mathbf{Hodgcat}$
 $$\mathcal{C}/\mathcal{N}\stackrel{\sim}\longrightarrow \mathcal{SC}$$
 which yields an isomorphism of  triangulated categories
 $$H^0(\mathcal{C})/{{\rm thick}\langle {\rm Cone}(\theta_X)\; |\; X\in {\rm Obj}(\mathcal{C})\rangle } \stackrel{\sim}\longrightarrow H^0(\mathcal{SC}).$$
 \end{thm}

\begin{proof}
Recall that $\iota(\theta_X)$ is a dg-isomorphism in $\mathcal{SC}$. Therefore, its image in $H^0(\mathcal{SC})$ is an isomorphism. Applying $H^0(\iota)$ to (\ref{tri:cone}),  we infer that ${\rm Cone}(\theta_X)$ is annihilated by  $H^0(\iota)$, that is, $\iota({\rm Cone}(\theta_X))$ is contractible in $\mathcal{SC}$.  It follows that $\iota$ sends any object in $\mathcal{N}$ to a contractible object in $\mathcal{SC}$. Therefore, by the universal property in Lemma~\ref{lem:q-universal}, $\iota$ induces a morphism
$$\overline{\iota}\colon \mathcal{C}/\mathcal{N} \longrightarrow \mathcal{SC}$$
in $\mathbf{Hodgcat}$. Moreover, it induces a triangle functor
$$\Phi:=H^0(\overline{\iota})\colon H^0(\mathcal{C})/{H^0(\mathcal{N})} \longrightarrow H^0(\mathcal{SC}).$$

We claim that the induced triangle functor $\Phi$ is fully faithful.  As $\Phi$ acts on objects by the identity and thus it is essentially surjective, this claim implies that  $\Phi$ is an isomorphism of triangulated categories. By Lemma~\ref{lem:tri-quasi-equiv}, $\overline{\iota}$ is an isomorphism in $\mathbf{Hodgcat}$.

Let us prove the above claim. We observe that $\theta_X$ becomes an isomorphism in $H^0(\mathcal{C})/{H^0(\mathcal{N})}$. Therefore, for any object $X$,  $\iota(\theta_X)^{-1}\in H^0(\mathcal{SC})$ belongs to the image of $\Phi$. By Lemma~\ref{lem:SC}(2), any general morphism $[f;p]\colon X\rightarrow Y$ in $H^0(\mathcal{SC})$ is of the following form
$$[f;p]=\iota(\theta_Y)^{-1}\circ \iota(\theta_{\Omega(Y)})^{-1}\circ  \cdots \circ \iota(\theta_{\Omega^{p-1}(Y)})^{-1}\circ \iota(f).$$
It follows that $[f;p]$ necessarily lies in the image of $\Phi$, that is, $\Phi$ is full.

We assume that $\Phi(X)\simeq 0$, that is, $X$ is contractible in $\mathcal{SC}$. By Lemma~\ref{lem:SC}(3), the morphism $\theta_{\Omega^{n}(X)}\circ \cdots \circ \theta_{\Omega(X)}\circ \theta_X$ is a coboundary. Consider the following exact triangle in $H^0(\mathcal{C})$.
$$X\xrightarrow{\theta_{\Omega^{n}(X)}\circ \cdots \circ \theta_{\Omega(X)}\circ \theta_X} \Omega^{n+1}(X)\longrightarrow C\longrightarrow \Sigma(X)$$
It follows that $\Sigma(X)$ is isomorphic to a direct summand of $C$ in $H^0(\mathcal{C})$. On the other hand, as the cone of a composite morphism,  $C$ is an iterated extension of ${\rm Cone}(\theta_{\Omega^i(X)})$ for $0\leq i\leq n$. We conclude that $C$ and thus $X$ lie in $H^0(\mathcal{N})$ and are isomorphic to zero in $H^0(\mathcal{C})/H^0(\mathcal{N})$. This proves that $\Phi$ is faithful on objects. The claim follows from the following general fact in \cite[the proof of Theorem~3.5, p.446]{Ric89}: a full triangle functor which is faithful on objects is necessarily faithful.
\end{proof}

\section{The Yoneda dg category}\label{section7}

We introduce, using the bar resolution,  the Yoneda dg category that is a natural dg enhancement of the derived category. We prove that the dg tensor algebra studied in Section~\ref{section3} is isomorphic to the endomorphism algebra of a specific  object in the Yoneda dg category; see Proposition~\ref{prop:iso-Y-tensor}.  Throughout, we work in the relative situation.

\subsection{The bar and Yoneda dg categories}
\label{subsection7.1}

Let $E\rightarrow \Lambda$ be an algebra homomorphism. Its cokernel is denoted by  $\overline{\Lambda}$, which has a natural $E$-$E$-bimodule structure. We denote by $s\overline{\Lambda}$ the graded $E$-$E$-bimodule concentrated in degree $-1$. Its element is usually written as $s\overline{a}$.

The normalized $E$-relative bar resolution $\mathbb{B}$ of $\Lambda$ is a complex of $\Lambda$-$\Lambda$-bimodules given as follows. As a graded $\Lambda$-$\Lambda$-bimodule, we have
$$\mathbb{B}=\Lambda\otimes_E T_E(s\overline{\Lambda})\otimes_E \Lambda,$$
where ${\rm deg}(a_0\otimes_E s\overline{a}_{1, n}\otimes_E a_{n+1})=-n$. Here, for simplicity, we write
$$s\overline{a}_{1, n} := s\overline{a}_1\otimes_E s\overline{a}_2 \otimes_E\cdots \otimes_E s\overline{a}_n.$$
The differential $d$ is given such that $d(a_0\otimes_E a_1)=0$ and that
\begin{align*}
d(a_0\otimes_E s\overline{a}_{1, n}\otimes_E a_{n+1})={}  & a_0a_1\otimes_E s\overline{a}_{2,n}\otimes_E a_{n+1}+ (-1)^n a_0\otimes_E s\overline{a}_{1,n-1}\otimes_E a_{n}a_{n+1}\\
&+\sum_{i=1}^{n-1} (-1)^i a_0\otimes_E s\overline{a}_{1, i-1}\otimes_E s\overline{a_ia_{i+1}}\otimes_E s\overline{a}_{i+2, n}\otimes_E a_{n+1}.
\end{align*}
Here and later, as usual, $s\overline{a}_{1, 0}$ and $s\overline{a}_{n+1, n}$ are understood to be the empty word and should be ignored.

It is well known that $\mathbb{B}$ is a coalgebra in the monoidal category of  complexes of $\Lambda$-$\Lambda$-bimodules. To be more precise, we have a cochain map between complexes of $\Lambda$-$\Lambda$-bimodules
$$\Delta\colon \mathbb{B}\longrightarrow \mathbb{B}\otimes_\Lambda\mathbb{B}$$
given by
$$\Delta(a_0\otimes_E s\overline{a}_{1, n}\otimes_E a_{n+1})=\sum_{i=0}^n (a_0\otimes_E s\overline{a}_{1,i}\otimes_E 1_\Lambda)\otimes_\Lambda (1_\Lambda\otimes_E s\overline{a}_{i+1, n}\otimes_E a_{n+1}).$$
The natural cochain map $\varepsilon\colon  \mathbb{B}\rightarrow \Lambda$ is induced by the multiplication of $\Lambda$. We have the following coassociative property
$$(\Delta\otimes_\Lambda {\rm Id}_\mathbb{B})\circ \Delta = ({\rm Id}_\mathbb{B} \otimes_\Lambda \Delta)\circ \Delta$$
and the counital property
$$(\varepsilon\otimes_\Lambda {\rm Id}_\mathbb{B}) \circ \Delta={\rm Id}_\mathbb{B}=({\rm Id}_\mathbb{B}\otimes_\Lambda \varepsilon) \circ \Delta.$$

Following the treatment in \cite[Subsection~6.6]{Kel94}, we define the \emph{$E$-relative bar dg category} $\mathcal{B}=\mathcal{B}_{\Lambda/E}$ as follows. The objects are precisely all the complexes of $\Lambda$-modules, and the Hom complex between two objects $X$ and $Y$ is given by
$$\mathcal{B}(X, Y)={\rm Hom}_\Lambda(\mathbb{B}\otimes_\Lambda X, Y).$$
The composition of two morphisms $f\in \mathcal{B}(X, Y)$ and $g\in \mathcal{B}(Y, Z)$ is defined to be
$$g\ast f := (\mathbb{B}\otimes_\Lambda X\xrightarrow{\Delta\otimes_\Lambda {\rm Id}_X} \mathbb{B}\otimes_\Lambda \mathbb{B}\otimes_\Lambda X\xrightarrow{{\rm Id}_\mathbb{B}\otimes_\Lambda f} \mathbb{B}\otimes_\Lambda Y\xrightarrow{g} Z).$$
Moreover, the identity endomorphism in $\mathcal{B}(X, X)$ is given by
$$\mathbb{B}\otimes_\Lambda X \xrightarrow{\varepsilon\otimes_\Lambda {\rm Id}_X} \Lambda\otimes_\Lambda X=X.$$
We mention that the bar dg category might be viewed as the coKleisli category of the comonad $\mathbb{B}\otimes_\Lambda-$ on  the dg category $C_{\rm dg}(\Lambda\mbox{-Mod})$ of complexes of $\Lambda$-modules; see \cite[Definition~4.8(ii)]{Kuel}.

We will unpack the above definition of $\mathcal{B}$ and obtain its alternative form. The \emph{$E$-relative Yoneda dg category} $\mathcal{Y}=\mathcal{Y}_{\Lambda/E}$ has the same objects as $\mathcal{B}$. For two complexes $X$ and $Y$ of $\Lambda$-modules, the underlying graded $\mathbb{K}$-module of the Hom complex $\mathcal{Y}(X, Y)$ is given by an infinite product
$$\mathcal{Y}(X, Y)=\prod_{n\geq 0}{\rm Hom}_E((s\overline{\Lambda})^{\otimes_E n}\otimes_EX, Y).$$
We denote
$$\mathcal{Y}_n(X, Y): ={\rm Hom}_E((s\overline{\Lambda})^{\otimes_E n}\otimes_E X, Y),$$
and say that elements in $\mathcal{Y}_n(X, Y)$ are of \emph{ filtration-degree} $n$. Observe that $\mathcal{Y}_0(X, Y)={\rm Hom}_E(X, Y)$.  The differential $\delta$ of $\mathcal{Y}(X, Y)$ is determined by
$$\begin{pmatrix}\delta_{\rm in}\\
\delta_{\rm ex}\end{pmatrix}\colon \mathcal{Y}_n(X, Y)\longrightarrow \mathcal{Y}_n(X, Y)\oplus \mathcal{Y}_{n+1}(X, Y),$$
where
$$\delta_{\rm in}(f)(s\overline{a}_{1,n}\otimes_E x)=d_Y (f(s\overline{a}_{1,n}\otimes_E x))-(-1)^{|f|+n}f(s\overline{a}_{1,n}\otimes_E d_X(x))$$
and
\begin{align*}
\delta_{\rm ex}(f)(s\overline{a}_{1,n+1}\otimes_E x) = & (-1)^{|f|+1}a_1 f(s\overline{a}_{2,n+1}\otimes_E x)+(-1)^{|f|+n}f(s\overline{a}_{1,n}\otimes_E a_{n+1}x)\\
&+\sum_{i=1}^{n}(-1)^{|f|+i+1} f(s\overline{a}_{1, i-1}\otimes_E s\overline{a_ia_{i+1}}\otimes_E s\overline{a}_{i+2, n+1}\otimes_E x).
\end{align*}
The composition $\odot$ of morphisms is defined as follows: for $f\in \mathcal{Y}_n(X, Y)$ and $g\in \mathcal{Y}_m(Y, Z)$, their composition $g\odot f\in \mathcal{Y}_{n+m}(X, Z)$ is given such that
\begin{equation}\label{equation:yonedaproduct}
(g\odot f)(s\overline{a}_{1, m+n}\otimes_E x)=(-1)^{m|f|} g(s\overline{a}_{1, m}\otimes_E f(s\overline{a}_{m+1, m+n}\otimes_E x)).
\end{equation}
The identity endomorphism in $\mathcal{Y}(X, X)$ is given by the genuine identity map ${\rm Id}_X\in \mathcal{Y}_0(X, X)$.

\begin{lem}
There is an isomorphism $\mathcal{B}\simeq \mathcal{Y}$ of dg categories.
\end{lem}

\begin{proof}
We observe that $\mathbb{B}\otimes_\Lambda X$ is canonically isomorphic to
$$\bigoplus_{n\geq 0}\Lambda\otimes_E (s\overline{\Lambda})^{\otimes_E n}\otimes_E X.$$
 Therefore, we have
{\footnotesize \begin{align}\label{equ:BY}
\mathcal{B}(X, Y) \simeq \prod_{n\geq 0} {\rm Hom}_\Lambda (\Lambda\otimes_E (s\overline{\Lambda})^{\otimes_E n}\otimes_E X, Y)\simeq \prod_{n\geq 0} {\rm Hom}_E((s\overline{\Lambda})^{\otimes_E n}\otimes_E X, Y)=\mathcal{Y}(X, Y).
\end{align}}
The above isomorphism identifies $f\in \mathcal{Y}_n(X, Y)$ with $\tilde{f}\colon \Lambda\otimes_E (s\overline{\Lambda})^{\otimes_E n}\otimes_E X\rightarrow Y $ given by
$$\tilde{f}(a_0\otimes_E s\overline{a}_{1, n}\otimes_E x)=a_0 f(s\overline{a}_{1, n}\otimes_E x).$$
The two dg categories $\mathcal{B}$ and $\mathcal{Y}$ have the same objects. It is routine to verify that the above isomorphism of the Hom complexes induces the required isomorphism of dg categories.
\end{proof}

\subsection{The dg derived categories}
Recall from Example~\ref{exm:Cdg} that $C_{\rm dg}(\Lambda\mbox{-Mod})$ denotes the dg category of complexes of $\Lambda$-modules. A complex $X$ of $\Lambda$-modules is called \emph{$E$-relatively acyclic} if it is contractible as a complex of $E$-modules, or equivalently,  $X  \simeq 0$ in $\mathbf{K}(E\mbox{-Mod})$; see \cite[Subsection~7.4]{Kel98}. In particular, an $E$-relatively acyclic complex is acyclic. Those complexes form a full dg subcategory $C_{\rm dg}^{\rm rel\mbox{-}ac}(\Lambda\mbox{-Mod})$. The corresponding dg quotient
$$\mathbf{D}_{\rm dg}(\Lambda/E)=C_{\rm dg}(\Lambda\mbox{-Mod})/{C_{\rm dg}^{\rm rel\mbox{-}ac}(\Lambda\mbox{-Mod})}$$
is called the \emph{$E$-relative dg derived category} of $\Lambda$. By Lemma~\ref{lem:quo}, its homotopy category $H^0(\mathbf{D}_{\rm dg}(\Lambda/E))$ is isomorphic to the $E$-relative derived category
$$\mathbf{D}(\Lambda/E)=\mathbf{K}(\Lambda\mbox{-Mod})/{\mathbf{K}^{\rm rel\mbox{-}ac}(\Lambda\mbox{-Mod})}.$$

A cochain map $f\colon X\rightarrow Y$ between complexes of $\Lambda$-modules is said to be an \emph{$E$-relative quasi-isomorphism} if its mapping cone is $E$-relatively acyclic, i.e.\ $\mathrm{Cone}(f) \simeq 0$ in $\mathbf{K}(E\mbox{-Mod})$.

Recall that a $\Lambda$-module $N$ is {\it $E$-relatively projective} if it is a direct summand of $\Lambda \otimes_E V$ for some $E$-module $V$. A complex $P$ of $\Lambda$-modules is called \emph{$E$-relatively dg-projective} provided that each component $P^i$ is $E$-relatively projective and the Hom complex ${\rm Hom}_\Lambda(P, X)$ is acyclic for any $E$-relatively acyclic complex $X$.

The following facts are standard.

\begin{lem}\label{lem:dgproj} For any complex $X$ of $\Lambda$-modules,  the following statements hold.
\begin{enumerate}
\item The complex $\mathbb{B}\otimes_\Lambda X$ of $\Lambda$-modules is $E$-relatively dg-projective.

\item The natural surjection $\varepsilon\otimes_\Lambda {\rm Id}_X\colon \mathbb{B}\otimes_\Lambda X \twoheadrightarrow \Lambda\otimes_\Lambda X=X$ is an $E$-relative quasi-isomorphism.

\item If $X$ is $E$-relatively acyclic, then $\mathbb{B}\otimes_\Lambda X$ is contractible.
\item  If $X$ is $E$-relatively dg-projective, then $\varepsilon\otimes_\Lambda {\rm Id}_X$ is a homotopy equivalence.
\end{enumerate}
\end{lem}

\begin{proof}
For (1),  we observe that $\mathbb{B}\otimes_\Lambda X$ has an ascending filtration with factors
 $$\Lambda\otimes_E (s\overline{\Lambda})^{\otimes_E n}\otimes_E X,$$
  each of which is $E$-relatively dg-projective. By \cite[Proposition~7.5]{Kel98}, we infer that $\mathbb{B}\otimes_E X$ is $E$-relatively dg-projective.

For (2),  we recall the standard fact that $\varepsilon\colon \mathbb{B}\rightarrow \Lambda$ is a homotopy equivalence as a cochain map between complexes of $E$-$\Lambda$-bimodules; see for example \cite[Chapter~IX, Theorem~8.1]{MacL}. That is,
${\rm Cone}(\varepsilon) \simeq 0$ in $\mathbf{K}(\text{$E \otimes \Lambda^{\rm op}$-Mod})$. This yields $${\rm Cone}(\varepsilon\otimes_\Lambda {\rm Id}_X) \simeq {\rm Cone}(\varepsilon) \otimes_\Lambda X \simeq 0$$
 in $\mathbf{K}(\text{$E$-Mod})$. We infer that $\varepsilon\otimes_\Lambda {\rm Id}_X$ is an $E$-relative quasi-isomorphism.

  For (3), it follows from (2)  that $\mathbb{B}\otimes_\Lambda X \simeq X \simeq 0$ in $\mathbf{K}(\text{$E$-Mod})$. Then we infer (3) by the following easy fact: any  complex which is both  $E$-relatively dg-projective and $E$-relatively acyclic is necessarily contractible as a complex of $\Lambda$-modules.

  For (4), we consider the following exact triangle in $\mathbf{K}(\Lambda\mbox{-Mod})$.
$$\mathbb{B}\otimes_\Lambda X \xrightarrow{\varepsilon\otimes_\Lambda {\rm Id}_X} X \longrightarrow {\rm Cone}(\varepsilon\otimes_\Lambda {\rm Id}_X)\longrightarrow \Sigma(\mathbb{B}\otimes_\Lambda X)$$
By (1), $\mathbb{B}\otimes_\Lambda X$ is $E$-relatively dg-projective. Since $X$ is $E$-relatively dg-projective, it follows that  so is ${\rm Cone}(\varepsilon\otimes_\Lambda {\rm Id}_X)$. By the proof of (2), we know that ${\rm Cone}(\varepsilon\otimes_\Lambda {\rm Id}_X)$ is $E$-relatively acyclic. Then by the above easy fact, we infer that ${\rm Cone}(\varepsilon\otimes_\Lambda {\rm Id}_X)$ is contractible as a complex of $\Lambda$-modules. This implies (4).
\end{proof}

Consider the natural dg functor
$${\rm \Theta}\colon C_{\rm dg}(\Lambda\mbox{-Mod})\longrightarrow \mathcal{Y}_{\Lambda/E}=\mathcal{Y}$$
which acts on objects by the identity, and identifies $f\in {\rm Hom}_\Lambda(X, Y)$ with $f\in {\rm Hom}_E(X, Y) = \mathcal{Y}_0(X, Y)\subseteq \mathcal{Y}(X, Y)$. Indeed, $C_{\rm dg}(\Lambda\mbox{-Mod})$ might be viewed as a non-full dg subcategory of $\mathcal{Y}$.

The following result justifies our terminology for $\mathcal{Y}$, since for each $\Lambda$-module $M$, the cohomology of $\mathcal{Y}(M, M)$ is isomorphic to the \emph{$E$-relative Yoneda algebra} of $M$
$${\rm Ext}^*_{\Lambda/E}(M, M)=\bigoplus_{i\geq 0} {\rm Hom}_{\mathbf{D}(\Lambda/E)} (M, \Sigma^i(M)).$$

\begin{prop}\label{prop:Theta}
The above dg functor $\Theta$ induces an isomorphism in $\mathbf{Hodgcat}$
 $$\overline{\Theta}\colon \mathbf{D}_{\rm dg}(\Lambda/E)\simeq \mathcal{Y}_{\Lambda/E}.$$
 Consequently, $\mathcal{Y}_{\Lambda/E}$ is pretriangulated and we have an isomorphism
 $$\mathbf{D}(\Lambda/E)\simeq H^0(\mathcal{Y}_{\Lambda/E})$$
 of triangulated categories.
\end{prop}

\begin{proof}
For any $E$-relatively acyclic complex $X$, by Lemma~\ref{lem:dgproj}(3), $\mathbb{B}\otimes_\Lambda X$ is contractible as a complex of $\Lambda$-modules.
Recall the isomorphism in \eqref{equ:BY}
$$\mathcal{Y}(X, X)\simeq \mathcal{B}(X, X)={\rm Hom}_\Lambda(\mathbb{B}\otimes_\Lambda X, X).$$
 It follows that $\mathcal{B}(X, X)$ and thus $\mathcal{Y}(X, X)$ are acyclic. We infer that $X$ is contractible in $\mathcal{Y}$. This shows that $\Theta(X)=X$ is contractible for any $X$ in ${C_{\rm dg}^{\rm rel\mbox{-}ac}(\Lambda\mbox{-Mod})}$.

By the universal property in Lemma~\ref{lem:q-universal}, $\Theta$ induces a morphism in $\mathbf{Hodgcat}$
$$\overline{\Theta}\colon \mathbf{D}_{\rm dg}(\Lambda/E)\longrightarrow  \mathcal{Y}.$$
 As $\overline{\Theta}$ acts on objects by the identity, it suffices to prove that it is quasi-fully faithful.

We claim that for any $E$-relatively dg-projective complex $X$,  the inclusion
$${\rm Hom}_\Lambda(X, Y)\longrightarrow \mathcal{Y}(X, Y)$$
is a quasi-isomorphism. Indeed, note that the inclusion  equals the composition of
$${\rm Hom}_\Lambda(\varepsilon\otimes_\Lambda {\rm Id}_X, Y)\colon {\rm Hom}_\Lambda(X, Y)\longrightarrow \mathcal{B}(X, Y)$$
with the isomorphism (\ref{equ:BY}). By Lemma~\ref{lem:dgproj}(4), $\varepsilon\otimes_\Lambda {\rm Id}_X$ is a homotopy equivalence. Therefore, ${\rm Hom}_\Lambda(\varepsilon\otimes_\Lambda {\rm Id}_X, Y)$ is a quasi-isomorphism, proving the claim.

Denote by $\mathcal{P}$ the full dg subcategory of $C_{\rm dg}(\Lambda\mbox{-Mod})$ formed by $E$-relatively dg-projective complexes. The above claim implies that the following composition
$$\mathcal{P}\stackrel{\rm inc}\longrightarrow C_{\rm dg}(\Lambda\mbox{-Mod})\stackrel{q}\longrightarrow \mathbf{D}_{\rm dg}(\Lambda/E)\stackrel{\overline{\Theta}}\longrightarrow \mathcal{Y}$$
is quasi-fully faithful, where `${\rm inc}$' and $q$ denote the inclusion and quotient functors, respectively.  By \cite[Proposition~7.4]{Kel98}, the composite dg functor  $q \circ {\rm inc}$
is a quasi-equivalence.  This implies that $\overline{\Theta}$ is quasi-fully faithful, as required.

For the second statement, we recall that  the dg derived category $\mathbf{D}_{\rm dg}(\Lambda/E)$ is pretriangulated; see Lemma~\ref{lem:quo}. It follows from Lemma~\ref{lem:pretri}(1) and the isomorphism $\overline{\Theta}$ that $\mathcal{Y}$ is also pretriangulated.
\end{proof}

\begin{rem}\label{rem:bounded-ss-Y}
(1)  Denote by $\mathcal{Y}_{\Lambda/E}^b$ the full dg subcategory of $\mathcal{Y}_{\Lambda/E}$ consisting of bounded complexes. We have the \emph{$E$-relative bounded dg derived category}
$$\mathbf{D}_{\rm dg}^b(\Lambda/E)=C^b_{\rm dg}(\Lambda\mbox{-Mod})/{C_{\rm dg}^{{\rm rel\mbox{-}ac}, b}(\Lambda\mbox{-Mod})},$$
where $C^b_{\rm dg}(\Lambda\mbox{-Mod})$ is the full dg subcategory of $C_{\rm dg}(\Lambda\mbox{-Mod})$ consisting of bounded complexes. Its homotopy category $H^0(\mathbf{D}_{\rm dg}^b(\Lambda/E))$ coincides with the $E$-relative bounded derived category
$$\mathbf{D}^b(\Lambda/E)=\mathbf{K}^b(\Lambda\mbox{-Mod})/{\mathbf{K}^{{\rm rel\mbox{-}ac}, b}(\Lambda\mbox{-Mod})}.$$

 The bounded version of  Proposition~\ref{prop:Theta} claims an isomorphism in $\mathbf{Hodgcat}$
 $$\mathbf{D}^b_{\rm dg}(\Lambda/E)\simeq \mathcal{Y}^b_{\Lambda/E}$$
 and an isomorphism of triangulated categories
 $$\mathbf{D}^b(\Lambda/E)\simeq H^0(\mathcal{Y}^b_{\Lambda/E}).$$

 (2) Assume that $E$ is a semisimple ring. Then the $E$-relative derived dg category $\mathbf{D}_{\rm dg}(\Lambda/E)$ coincides with the (absolute) dg derived category
     $$\mathbf{D}_{\rm dg}(\Lambda\mbox{-Mod})=C_{\rm dg}(\Lambda\mbox{-Mod})/{C_{\rm dg}^{\rm ac}(\Lambda\mbox{-Mod})}.$$
    Similarly, $\mathbf{D}(\Lambda/E)$ coincides with $\mathbf{D}(\Lambda\mbox{-Mod})$, the derived category of $\Lambda\mbox{-Mod}$.  Therefore, we have isomorphisms in $\mathbf{Hodgcat}$
     $$ \mathbf{D}_{\rm dg}(\Lambda\mbox{-Mod})\simeq \mathcal{Y}_{\Lambda/E} \quad \mbox{ and } \quad \mathbf{D}^b_{\rm dg}(\Lambda\mbox{-Mod})\simeq \mathcal{Y}^b_{\Lambda/E}.$$
     They induce isomorphisms of triangulated categories:
     $$ \mathbf{D}(\Lambda\mbox{-Mod})\simeq H^0(\mathcal{Y}_{\Lambda/E}) \quad \mbox{ and } \quad \mathbf{D}^b(\Lambda\mbox{-Mod})\simeq H^0(\mathcal{Y}^b_{\Lambda/E}).$$
\end{rem}

We are mostly interested in the finitely generated modules. In the given algebra extension $E\rightarrow \Lambda$, we assume that $E$ is a semisimple ring and that $\Lambda$ is a left noetherian ring. Denote by $\Lambda\mbox{-mod}$ the full subcategory of $\Lambda\mbox{-Mod}$ formed by finitely generated $\Lambda$-modules.  The \emph{bounded dg derived category}  is defined as
$$\mathbf{D}_{\rm dg}^b(\Lambda\mbox{-mod})=C^b_{\rm dg}(\Lambda\mbox{-mod})/{C_{\rm dg}^{{\rm ac}, b}(\Lambda\mbox{-mod})},$$
whose homotopy category coincides with $\mathbf{D}^b(\Lambda\mbox{-mod})$, the bounded derived category of $\Lambda\mbox{-mod}$. Denote by $\mathcal{Y}_{\Lambda/E}^f$ the full dg subcategory of $\mathcal{Y}_{\Lambda/E}^b$ consisting of bounded complexes in $\Lambda\mbox{-mod}$.

We have the following finite version of Proposition~\ref{prop:Theta}; compare Remark~\ref{rem:bounded-ss-Y}(2).

\begin{cor}\label{cor:finite}
Assume that $E$ is a semisimple ring and that $\Lambda$ is a left noetherian ring. Then there is an isomorphism in $\mathbf{Hodgcat}$
 $$ \mathbf{D}^b_{\rm dg}(\Lambda\mbox{-}{\rm mod})\simeq \mathcal{Y}^f_{\Lambda/E},$$
 which induces an isomorphism of triangulated categories
 $$\mathbf{D}^b(\Lambda\mbox{-}{\rm mod})\simeq H^0(\mathcal{Y}^f_{\Lambda/E}).$$
\end{cor}

\subsection{The dg tensor algebra as an endomorphism algebra}\label{subsec:YE}

Throughout this subsection, we will impose the following finiteness conditions on the algebra extension $E\rightarrow \Lambda$.
\begin{enumerate}
\item[(Fin1)] The left $E$-module $_E\Lambda$ is finitely generated projective.
\item[(Fin2)] There is an algebra homomorphism $\pi\colon \Lambda\rightarrow E$ such that the composition $E\rightarrow \Lambda \stackrel{\pi}\rightarrow E$ is the identity map.
\end{enumerate}

Recall that $\overline \Lambda = \Lambda / (E\cdot 1_\Lambda)$ is the cokernel of $E\to \Lambda$. The natural map
$${\rm Ker}(\pi)\stackrel{\sim}\longrightarrow \overline{\Lambda}, \quad a\mapsto \overline{a}$$
is an isomorphism of $E$-$E$-bimodules. The multiplication on the ideal ${\rm Ker}(\pi)$ induces an associative $E$-$E$-bimodule morphism
\begin{align}\label{equ:mu}
\mu\colon \overline{\Lambda}\otimes_E \overline{\Lambda}\longrightarrow \overline{\Lambda}, \quad \overline{a}\otimes_E\overline{b}\mapsto \overline{ab}.
\end{align}
Here, we emphasize that the elements $a$ and $b$ are chosen to lie in ${\rm Ker}(\pi)$. Associated to $(\overline{\Lambda}, \mu)$, we have the dg tensor algebra $(T_E(\overline{\Lambda}^*), \partial)$; see Remark~\ref{rem:dgten}. We mention that such dg tensor algebras are related to normal bocses; see \cite[Remark~3.4]{BSZ}.

We will view $E$ as a $\Lambda$-module via the algebra homomorphism $\pi$.

\begin{prop}\label{prop:iso-Y-tensor}
Assume that (Fin1) and (Fin2) hold. Then there is an isomorphism $\mathcal{Y}_{\Lambda/E}(E, E)\simeq (T_E(\overline{\Lambda}^*), \partial)^{\rm op}$ of dg algebras.
\end{prop}

\begin{proof}
In this proof, we write $M=\overline{\Lambda}$. Recall that
$$\mathcal{Y}(E, E)=\prod_{n\geq 0}{\rm Hom}_E((sM)^{\otimes_E n}, E)=\bigoplus_{n\geq 0}{\rm Hom}_E((sM)^{\otimes_E n}, E),$$
whose multiplication is induced by the composition $\odot$ and the differential is given by $\delta_{\rm ex}$; see \eqref{equation:yonedaproduct}. Here, we note that $\delta_{\rm in}$ vanishes.

We note that $T_E(M^*)$ is graded with $(M^*)^{\otimes_E n}$ in  degree $n$. To emphasize the degrees, we replace $M^*$ by $s^{-1}M^*$. Its typical element is denoted by $s^{-1}f$ with $f\in M^*$. So, as a graded algebra, we have
$$T_E(s^{-1}M^*)=\bigoplus_{n\geq 0} (s^{-1}M^*)^{\otimes_E n}.$$
For each degree $n>0$, we have a natural isomorphism of $E$-$E$-bimodules.
\[
\begin{array}{cccc}
\phi^n\colon &  (s^{-1}M^*)^{\otimes_E n} & \rightarrow  & {\rm Hom}_E((sM)^{\otimes_E n}, E)\\
& s^{-1}f_{1, n} &\mapsto & (s\overline{a}_{1,n}\mapsto (-1)^nf_n(\overline{a}_1f_{n-1}(\dotsb (\overline{a}_{n-1}f_1(\overline{a}_n))\cdots))\in E)
\end{array}
\]
Define $\phi^0\colon E\rightarrow {\rm Hom}_E(E, E)$ by $\phi^0(x)(y)=\phi(yx)$.

These $\phi^n$'s yield an isomorphism
$$\phi\colon T_E(s^{-1}M^*)^{\rm op}\stackrel{\sim}\longrightarrow \mathcal{Y}(E, E)$$
of graded $\mathbb{K}$-modules. It is direct to verify that $\phi$ is compatible with the multiplications.

To show that $\phi$ preserves the differentials, it suffices to verify
$$\phi \circ \partial=\delta_{\rm ex}\circ \phi$$
 on $E\oplus s^{-1}M^*$. The verification on $E$ is trivial, since $\partial|_E=0$ and $\delta_{\rm ex}|_{\mathcal{Y}_0(E,E)}=0$.  Recall that the restriction of $\partial$ on $s^{-1}M^*$ is given by $\partial_{+}$ in (\ref{equ:par+}). Then we are done by verifying the commutativity of the following square.
\begin{align}\label{diag:partial+}
\xymatrix{
s^{-1}M^*\ar[d]_-{\partial_{+}}\ar[rr]^-{\phi^1}&& {\rm Hom}_E(sM, E)\ar[d]^-{\delta_{\rm ex}}\\
s^{-1}M^*\otimes_E s^{-1}M^* \ar[rr]^-{\phi^2} && {\rm Hom}_E((sM)^{\otimes_E 2},E)
}\end{align}

For the verification, take any $f\in M^*$. We observe that $\delta_{\rm ex}\circ \phi^1(s^{-1}f)$ is the element in ${\rm Hom}_E((sM)^{\otimes_E 2},E)$, which sends $s\overline{a}_{1, 2}$ to $f(\overline{a_1a_2})$. Here, we use the fact that there are minus signs appearing in both  $\phi^1$ and $\delta_{\rm ex}$.

On the other hand, we assume that
$$\partial_{+}(s^{-1}f)=\sum_{i=1}^m s^{-1}g_i\otimes_E s^{-1}h_i\in s^{-1}M^*\otimes_E s^{-1}M^*$$
 for some $g_i, h_i \in M^*$. By the very definition of $\partial_+$ in (\ref{equ:par+}), we have
\begin{align}\label{equ:ver}
\sum_{i=1}^m h_i(\overline{a}_1 g_i(\overline{a}_2)) = f(\overline{a_1a_2});
\end{align}
see also \eqref{align:verydefintionpartial}. By the definition of $\phi^2$ above, we observe that $$\phi^2(\sum_{i=1}^m s^{-1}g_i\otimes_E s^{-1}h_i)$$
sends $s\overline{a}_{1,2}$ to $\sum_{i=1}^m h_i(\overline{a}_1g_i(\overline{a}_2))$. Combining this observation with (\ref{equ:ver}), we infer that $\phi^2\circ \partial_+(s^{-1}f)$ also sends  $s\overline{a}_{1,2}$ to $f(\overline{a_1a_2})$. This completes the verification of the commutativity above.
\end{proof}

\section{Noncommutative differential forms}\label{section8}

In this section, we study noncommutative differential forms with values in a complex.  This gives rise to a dg endofunctor $\Omega_{\rm nc}$ on the Yoneda dg category $\mathcal Y$.  We also obtain a closed natural transformation $\theta \colon {\rm Id}_{\mathcal Y} \to \Omega_{\rm nc}$ of degree zero satisfying $\theta \Omega_{\rm nc} = \Omega_{\rm nc} \theta$; see Section \ref{section6}. As before, we fix an algebra extension $E\rightarrow \Lambda$ and work in the relative situation.

Let $X$ be a complex of $\Lambda$-modules. The complex of \emph{$X$-valued $E$-relative noncommutative differential $1$-forms} is defined by
$$\Omega_{{\rm nc}, \Lambda/E}(X)=s\overline{\Lambda}\otimes_E X,$$
which is graded by means of ${\rm deg}(s\overline{a}_1\otimes_E x)=|x|-1$ and whose differential is given by $d(s\overline{a}_1\otimes_E x)=-s\overline{a}_1\otimes_E d_X(x)$. The $\Lambda$-action is given by the following nontrivial rule:
\begin{align}\label{equ:nontri-rule}
a\blacktriangleright (s\overline{a}_1\otimes_E x)=s\overline{aa_1}\otimes_E x-s\overline{a}\otimes_E a_1x.
\end{align}

To justify the above terminology, we observe that
$$\Omega_{{\rm nc}, \Lambda/E}(\Lambda)=s\overline{\Lambda}\otimes_E \Lambda$$
 is a stalk complex of $\Lambda$-$\Lambda$-bimodules concentrated in degree $-1$, where the right $\Lambda$-action is given by the multiplication of $\Lambda$. This stalk complex is called the \emph{graded bimodule of $E$-relative right noncommutative differential $1$-forms} \cite{Wan1}. Moreover, we have a canonical isomorphism
 $$\Omega_{{\rm nc}, \Lambda/E}(\Lambda)\otimes_\Lambda X\simeq \Omega_{{\rm nc}, \Lambda/E}(X), $$
 which sends $(s\overline{a}_0\otimes_E a_1)\otimes_\Lambda x$ to $s\overline{a}_0\otimes_E a_1x$. We mention that the study of noncommutative differential forms goes back to \cite[Sections~1~and~2]{CQ}.

 To avoid notational overload, we write $\mathcal{Y}=\mathcal{Y}_{\Lambda/E}$ and $\Omega_{\rm nc}(X)=\Omega_{{\rm nc}, \Lambda/E}(X)$. We have a dg functor $$\Omega_{\rm nc}\colon \mathcal{Y}\longrightarrow \mathcal{Y}, \quad X\mapsto \Omega_{\rm nc}(X),$$
which sends a morphism $f\in \mathcal{Y}_n(X, Y)$ to the morphism in $\mathcal{Y}_{n}(\Omega_{\rm nc}(X), \Omega_{\rm nc}(Y))$:
 $$  \xymatrix@C=3pc{
 (s\overline{\Lambda})^{\otimes_E n}\otimes_E \Omega_{\rm nc}(X) \ar[r]^-{=}  &  (s\overline{\Lambda})^{\otimes_E (n+1)}\otimes_E X \ar[r]^-{\mathrm{Id}_{s\overline{\Lambda}}\otimes_E f} & s\overline{\Lambda}\otimes_E Y=\Omega_{\rm nc}(Y).
 }$$

 We have a closed natural transformation of degree zero
 $$\theta\colon {\rm Id}_\mathcal{Y}\longrightarrow \Omega_{\rm nc}$$
defined as follows. For any $X \in \mathcal{Y}$, $\theta_X$ lies in $\mathcal{Y}_1(X, \Omega_{\rm nc}(X))\subseteq \mathcal{Y}(X, \Omega_{\rm nc}(X))$ and is given by
$$\theta_X(s\overline{a}\otimes_E x)=s\overline{a}\otimes_E x\in \Omega_{\rm nc}(X).$$
Note that $\theta_X$ is of degree zero and $\delta(\theta_X)=0$ using the nontrivial rule (\ref{equ:nontri-rule}). For the naturalness of $\theta$, we observe that for each $f\in \mathcal{Y}_n(X, Y)$, we have
$$\theta_Y\odot f=\Omega_{\rm nc}(f)\odot \theta_X.$$
Indeed, both sides send $s\overline{a}_{1, n+1}\otimes_E x$ to $(-1)^{|f|}s\overline{a}_1\otimes_E f(s\overline{a}_{2, n}\otimes_E x)$.

 We observe that
 $$\Omega_{\rm nc}(\theta_X)=\theta_{\Omega_{\rm nc}(X)},$$
 since both sides lie in $\mathcal{Y}_1(\Omega_{\rm nc}(X), \Omega_{\rm nc}^2(X)) = \mathrm{Hom}_E(\Omega_{\rm nc}^2(X), \Omega_{\rm nc}^2(X))$ and correspond to the identity map of $\Omega_{\rm nc}^2(X)$. In summary, we conclude that the triple $(\mathcal{Y}_{\Lambda/E}, \Omega_{\rm nc}, \theta)$ satisfies the assumptions made in the first paragraph of Section~\ref{section6}.

\vskip 5pt

 In the remaining of this section, we will analyze the cone of $\theta_X$ in $H^0(\mathcal{Y})$.

Let $N$ be the complex $\overline{\Lambda}\otimes_E X$ of $\Lambda$-modules, with the $\Lambda$-action given by
$$b(\overline{a}\otimes_E x)=\overline{ba}\otimes_E x-\overline{b}\otimes_E ax.$$
In view of (\ref{equ:nontri-rule}), we have $\Sigma(N)=\Omega_{\rm nc}(X)$. Consider the following  sequence of complexes of $\Lambda$-modules.
 $$\xi_X\colon 0\longrightarrow N \stackrel{i_X}\longrightarrow \Lambda\otimes_E X\stackrel{m_X}\longrightarrow X\longrightarrow 0$$
Here, $i_X(\overline{a}\otimes_E x)=a\otimes_E x-1_\Lambda \otimes_E ax$ and $m_X(a\otimes_E x)=ax$. We claim that it is a split short exact sequence between the underlying complexes of $E$-modules.

For the claim, we define $s_X\colon X\rightarrow \Lambda\otimes_E X$ by $s_X(x)=1_\Lambda \otimes_E x$, and $t_X\colon \Lambda\otimes_E X\rightarrow N$ by $t_X(a\otimes_E x)=\overline{a}\otimes_E x$. Both $s_X$ and $t_X$ are chain maps between complexes of $E$-modules. We infer the claim from the following easy identities:
$$m_X\circ s_X={\rm Id}_X, \quad  t_X\circ i_X={\rm Id}_N, \quad \mbox{and } \quad  i_X\circ t_X+s_X\circ m_X={\rm Id}_{\Lambda\otimes_E X}.$$

  Since $\Sigma(N)=\Omega_{\rm nc}(X)$, the mapping cone of $i_X$ is described as follows:
  $$\mathrm{Cone}(i_X)=(\Lambda\otimes_E X) \oplus \Omega_{\rm nc}(X), $$
 whose differential is given  by $d_{\mathrm{Cone}(i_X)}(a\otimes_E x, 0)=(a\otimes_E d_X(x), 0)$ and
  $$d_{\mathrm{Cone}(i_X)}(0, s\overline{a}\otimes_E x)= (a\otimes_E x-1_\Lambda \otimes_E ax, \ -s\overline{a}\otimes_E d_X(x)).$$
  Denote the projection by
  \[
  {\rm pr}\colon \mathrm{Cone}(i_X)\rightarrow \Omega_{\rm nc}(X).
  \]
  Since  the underlying short exact sequence of  $\xi_X$ splits over $E$,  the induced cochain map
  $$(m_X,0)\colon \mathrm{Cone}(i_X)\longrightarrow X$$
  is an $E$-relative quasi-isomorphism.

We view both ${\rm pr}$ and $(m_X,0)$ as morphisms in $\mathcal{Y}$, which have filtration-degree zero. Namely, ${\rm pr} \in \mathcal Y_0(\mathrm{Cone}(i_X), \Omega_{\rm nc}(X))$ and $(m_X,0) \in \mathcal Y_0(\mathrm{Cone}(i_X), X).$

\begin{lem}\label{lem:theta}
Keep the notation as above. Then we have $\theta_X\odot (m_X,0)={\rm pr}$ in $H^0(\mathcal{Y}_{\Lambda/E})$.
\end{lem}

\begin{proof}
Define a  map $h \colon \mathrm{Cone}(i_X)\rightarrow \Omega_{\rm nc}(X)$ of degree $-1$ by $h(a\otimes_E x, 0)=s\overline{a}\otimes_E x$ and $h(0, s\overline{b}\otimes_E y)=0$. We view $h$ as an element in $\mathcal{Y}_0(\mathrm{Cone}(i_X), \Omega_{\rm nc}(X))$.

We observe that the differential $\delta_{\rm in}(h)\colon \mathrm{Cone}(i_X)\rightarrow \Omega_{\rm nc}(X)$ is determined by
$$\delta_{\rm in}(h)(a\otimes_E x, 0)=0 \quad \mbox{ and } \quad \delta_{\rm in}(h)(0, s\overline{b}\otimes_E y)=s\overline{b}\otimes_E y.$$
The differential $\delta_{\rm ex}(h)$ lies in
$$\mathcal{Y}_1(\mathrm{Cone}(i_X), \Omega_{\rm nc}(X)) = \mathrm{Hom}_E( s\overline{\Lambda}\otimes_E \mathrm{Cone}(i_X),  \Omega_{\rm nc}(X)),$$
which is determined by
$$\delta_{\rm ex}(h)(s\overline{a}_1\otimes_E (a\otimes_E x, 0))=-s\overline{a}_1\otimes_E ax\quad \mbox{ and }\quad \delta_{\rm ex}(h)(s\overline{a}_1\otimes_E(0, s\overline{b}\otimes_E y))=0.$$
Note that $ \delta_{\rm in}(h)= \mathrm{pr}$ and $\delta_{\rm ex}(h) = -\theta_X\odot (m_X,0)$. We conclude that
$$\delta(h)=\delta_{\rm in}(h)+\delta_{\rm ex}(h)={\rm pr}-\theta_X\odot (m_X,0).$$
This proves the required equality in $H^0(\mathcal{Y})$.
\end{proof}

Recall that the cone of $\theta_X$ is determined by the following exact triangle in $H^0(\mathcal{Y})$.
$$X\stackrel{\theta_X}\longrightarrow \Omega_{\rm nc}(X)\longrightarrow {\rm Cone}(\theta_X)\longrightarrow \Sigma(X)$$

\begin{prop}\label{prop:cone}
There is an isomorphism ${\rm Cone}(\theta_X)\simeq \Sigma(\Lambda\otimes_E X)$ in $H^0(\mathcal{Y}_{\Lambda/E})$.
\end{prop}

\begin{proof}
The short exact sequence $\xi_X$ induces an exact triangle in $\mathbf{D}(\Lambda/E)$
\begin{align}\label{equ:ext-c}
N\stackrel{i_X}\longrightarrow \Lambda\otimes_E X\stackrel{m_X}\longrightarrow X \stackrel{c}\longrightarrow \Omega_{\rm nc}(X),
\end{align}
where  the connecting morphism $c$ is given by the following roof
$$X  \xleftarrow{(m_X, 0)} \mathrm{Cone}(i_X) \stackrel{\rm pr}\longrightarrow \Omega_{\rm nc}(X).$$ Thus, we have  $c = \mathrm{pr} \circ (m_X, 0)^{-1}$ in $\mathbf{D}(\Lambda/E)$.

Recall from Proposition~\ref{prop:Theta} the triangle isomorphism $H^0(\overline{\Theta})\colon \mathbf{D}(\Lambda/E)\simeq H^0(\mathcal{Y})$, which acts on objects by the identity. As $H^0(\overline{\Theta})$ sends cochain maps identically to the corresponding morphisms in $\mathcal{Y}$ of filtration-degree zero, we obtain
$$H^0(\overline{\Theta})(c)= \mathrm{pr} \circ (m_X, 0)^{-1} = \theta_X,$$
where the right equality follows from Lemma~\ref{lem:theta}.
This yields  $$H^0(\overline{\Theta})({\rm Cone}(c)) = {\rm Cone}(\theta_X).$$ By (\ref{equ:ext-c}), ${\rm Cone}(c)$ is isomorphic to $\Sigma(\Lambda\otimes_E X)$ in $\mathbf{D}(\Lambda/E)$. Then the required statement follows since $H^0(\overline{\Theta})(\Sigma(\Lambda\otimes_E X)) =\Sigma(\Lambda\otimes_E X).$
\end{proof}

\section{The singular Yoneda dg category}\label{section9}
In this section, we introduce the singular Yoneda dg category, which provides dg enhancements for various singularity categories. We fix an algebra extension $E\rightarrow \Lambda $ as before. We prove that the endomorphism algebra of $E$ in the singular Yoneda dg category of $\Lambda$ is isomorphic to a dg Leavitt algebra; see Theorem~\ref{thm:SY-Leavitt}.

As we have seen, the triple $(\mathcal{Y}_{\Lambda/E}, \Omega_{\rm nc}, \theta)$ obtained in Section~\ref{section8} satisfies the assumptions made in Section~\ref{section6}. We thus form the dg localization along $\theta$
$$\iota\colon \mathcal{Y}_{\Lambda/E}\longrightarrow \mathcal{SY}_{\Lambda/E}.$$
The obtained dg category $\mathcal{SY}_{\Lambda/E}$ is called the \emph{$E$-relative singular Yoneda dg category} of $\Lambda$.

Let us describe $\mathcal{SY}=\mathcal{SY}_{\Lambda/E}$ more explicitly. Its objects are just complexes of $\Lambda$-modules. For two objects $X$ and $Y$, the Hom complex is defined to be the colimit of the following sequence of cochain complexes.
$$\mathcal{Y}(X, Y)\longrightarrow \mathcal{Y}(X, \Omega_{\rm nc}(Y))\longrightarrow \cdots \longrightarrow \mathcal{Y}(X, \Omega_{\rm nc}^p(Y))\longrightarrow \mathcal{Y}(X, \Omega_{\rm nc}^{p+1}(Y))\longrightarrow \cdots$$
The structure map sends $f$ to $\theta_{\Omega_{\rm nc}^p(Y)}\odot f$; see \eqref{equation:yonedaproduct}. More precisely, for any $f \in  \mathcal{Y}_n(X, \Omega_{\rm nc}^p(Y))$, the map $ \theta_{\Omega_{\rm nc}^p(Y)}\odot f \in  \mathcal{Y}_{n+1}(X, \Omega_{\rm nc}^{p+1}(Y))$ is given by
\begin{align}\label{align:thetaidentification}
s \overline{a}_{1, n+1} \otimes_E x \longmapsto (-1)^{|f|} s\overline{a}_1 \otimes_E f(s\overline{a}_{2, n+1}\otimes_E x).
\end{align}

The image of $f\in \mathcal{Y}(X, \Omega_{\rm nc}^p(Y))$ in $\mathcal{SY}(X, Y)$ is denoted by $[f;p]$. The composition of $[f;p]$ with $[g;q]\in \mathcal{SY}(Y, Z)$ is defined by
$$[g;q]\odot_{\rm sg} [f;p]=[\Omega_{\rm nc}^p(g)\odot f;p+q].$$
By Proposition~\ref{prop:Theta},  the Yoneda dg category $\mathcal{Y}_{\Lambda/E}$ is pretriangulated. We infer from Lemma~\ref{lem:SC-pre} that $\mathcal{SY}_{\Lambda/E}$ is also pretriangulated.

\subsection{The dg singularity categories}

Recall that $\mathbf{D}(\Lambda/E)$ denotes the $E$-relative derived category of $\Lambda$. We define the \emph{$E$-relative virtual singularity category of $\Lambda$} to be the following quotient triangulated category
$$\mathcal{V}(\Lambda/E)=\mathbf{D}(\Lambda/E)/{{\rm thick}\langle \Lambda\otimes_E V\; |\; V \mbox{ complex of } E\mbox{-modules} \rangle}.$$
Denote by $\mathcal{N}$ the full  dg subcategory of $\mathbf{D}_{\rm dg}(\Lambda/E)$ formed by those objects
in ${\rm thick}\langle \Lambda\otimes_E V\; |\; V \mbox{ complex of }E\mbox{-modules} \rangle$. The following dg quotient category
$$\mathcal{V}_{\rm dg}(\Lambda/E)=\mathbf{D}_{\rm dg}(\Lambda/E)/\mathcal{N}$$
might be called the \emph{$E$-relative virtual  dg singularity category of $\Lambda$}. By Lemma~\ref{lem:quo}, we identify the homotopy category $H^0(\mathcal{V}_{\rm dg}(\Lambda/E))$ with $\mathcal{V}(\Lambda/E)$.

Recall from Proposition~\ref{prop:Theta} the isomorphism $\overline \Theta \colon \mathbf{D}_{\rm dg}(\Lambda/E)\simeq \mathcal{Y}_{\Lambda/E}$.

\begin{prop}\label{prop:V-SY}
Keep the notation as above. Then the composite dg functor $\mathbf{D}_{\rm dg}(\Lambda/E) \stackrel{\overline \Theta} \rightarrow \mathcal{Y}_{\Lambda/E}\stackrel{\iota}\rightarrow \mathcal{SY}_{\Lambda/E}$ induces an isomorphism in $\mathbf{Hodgcat}$
$$\mathcal{V}_{\rm dg}(\Lambda/E)\simeq \mathcal{SY}_{\Lambda/E}.$$
Consequently, we have an isomorphism of triangulated categories
$$\mathcal{V}(\Lambda/E)\simeq H^0(\mathcal{SY}_{\Lambda/E}).$$
\end{prop}

\begin{proof}

Consider the two thick hulls $\mathcal{T}_1={{\rm thick}\langle \Lambda\otimes_E V\; |\; V \mbox{ complex of } E\mbox{-modules} \rangle}$ and $\mathcal{T}_2={\rm thick}\langle {\rm Cone}(\theta_X)\; |\; X \mbox{ complex of } \Lambda\mbox{-modules}\rangle$ in $H^0(\mathcal Y_{\Lambda/E})$. Denote by $\mathcal N_1$ the full dg subcategory of $\mathcal Y_{\Lambda/E}$ formed by those objects in $\mathcal T_1$ and similarly, denote by $\mathcal N_2$ the one formed by those objects in $\mathcal T_2$. Clearly,  we have $\overline \Theta(\mathcal N) = \mathcal N_1$.

 By Proposition~\ref{prop:cone}, we identify ${\rm Cone}(\theta_X)$ with $\Sigma(\Lambda\otimes_E X)=\Lambda\otimes_E \Sigma(X)$. Then we have $\mathcal{T}_2 \subseteq \mathcal{T}_1$. On the other hand, any complex $Y$ of the form $\Lambda\otimes_E V$ is isomorphic to a direct summand of $\Lambda\otimes_EY $ since the natural surjection $\Lambda \otimes_E Y \to Y$ splits in $\mathcal{Y}_{\Lambda/E}$ with a section given by $a \otimes_E v \mapsto a \otimes_E 1 \otimes_E v$.  It follows that $\mathcal{T}_1 \subseteq \mathcal{T}_2$. So we have $\mathcal{T}_1=\mathcal{T}_2$ and thus $\mathcal N_1 = \mathcal N_2$. Now the required isomorphism in $\mathbf{Hodgcat}$ follows by combining the isomorphisms in Theorem~\ref{thm:dgl} and Proposition~\ref{prop:Theta}.
 \end{proof}

The following remark is analogous to Remark~\ref{rem:bounded-ss-Y}.

\begin{rem}\label{rem:bounded-ss-SY}
(1) Denote by $\mathcal{SY}^b_{\Lambda/E}$  the full dg subcategory of  $\mathcal{SY}_{\Lambda/E}$ consisting of bounded complexes. We view the bounded homotopy category $\mathbf{K}^b(\mathcal{P}_{\Lambda/E})$ of $E$-relatively projective $\Lambda$-modules as a triangulated subcategory of $\mathbf{D}^b(\Lambda/E)$. Inspired by \cite{Kel18}, we define the  \emph{$E$-relative completed singularity category} of $\Lambda$ by the following Verdier quotient
    $$\widehat{\mathbf{S}}(\Lambda/E):=\mathbf{D}^b(\Lambda/E)/{\mathbf{K}^b(\mathcal{P}_{\Lambda/E})}.$$
As its dg analogue, the \emph{$E$-relative completed dg singularity category} of $\Lambda$ is defined to be
$$\widehat{\mathbf{S}}_{\rm dg}(\Lambda/E):=\mathbf{D}^b_{\rm dg}(\Lambda/E)/{C^b_{\rm dg}(\mathcal{P}_{\Lambda/E})}.$$
We mention that the relative dg singularity category \cite[Definition~2.23]{BRTV} is quite different from the ones here.

We have a bounded version of Proposition~\ref{prop:V-SY}: there is an isomorphism in $\mathbf{Hodgcat}$
$$\widehat{\mathbf{S}}_{\rm dg}(\Lambda/E)\simeq \mathcal{SY}^b_{\Lambda/E},$$
which induces an isomorphism of triangulated categories
$$\widehat{\mathbf{S}}(\Lambda/E)\simeq H^0(\mathcal{SY}^b_{\Lambda/E}).$$

(2) Assume that $E$ is a semisimple ring. Then  $\mathcal{V}(\Lambda/E)$ coincides with the (absolute) virtual singularity category of $\Lambda$, which is defined as
$$\mathcal{V}(\Lambda):=\mathbf{D}(\Lambda\mbox{-Mod})/{{\rm thick}\langle \oplus_{i\in \mathbb{Z}} \Sigma^{-i}(P^i)\; |\; P^i\in \Lambda\mbox{-Proj} \rangle}.$$
Similarly,  $\widehat{\mathbf{S}}(\Lambda/E)$ coincides with  the completed singularity category \cite{Kel18} of $\Lambda$
    $$ \widehat{\mathbf{S}}(\Lambda):=\mathbf{D}^b(\Lambda\mbox{-Mod})/{\mathbf{K}^b(\Lambda\mbox{-Proj})}.$$
 We have the dg analogue $\mathcal{V}_{\rm dg}(\Lambda)$ of $\mathcal{V}(\Lambda)$, and the dg analogue $\widehat{\mathbf{S}}_{\rm dg}(\Lambda)$ of $\widehat{\mathbf{S}}(\Lambda)$. Then by the same reason, the above coincidences lift to the dg level, namely
 $$\mathcal{V}_{\rm dg}(\Lambda/E)=\mathcal{V}_{\rm dg}(\Lambda) \quad \mbox{ and } \quad \widehat{\mathbf{S}}_{\rm dg}(\Lambda/E)=\widehat{\mathbf{S}}_{\rm dg}(\Lambda).$$

    Consequently, by Proposition~\ref{prop:V-SY} we have isomorphisms in $\mathbf{Hodgcat}$
    $$\mathcal{V}_{\rm dg}(\Lambda)\simeq \mathcal{SY}_{\Lambda/E} \quad \mbox{ and }\quad \widehat{\mathbf{S}}_{\rm dg}(\Lambda)\simeq \mathcal{SY}^b_{\Lambda/E},$$
    which induce isomorphisms of triangulated categories:
    \begin{align}\label{equ:rem-tri}
    \mathcal{V}(\Lambda)\simeq H^0(\mathcal{SY}_{\Lambda/E}) \quad \mbox{ and } \quad \widehat{\mathbf{S}}(\Lambda)\simeq H^0(\mathcal{SY}^b_{\Lambda/E}).
    \end{align}
\end{rem}

 In the remaining of this subsection, we further assume that $E$ is a semisimple ring and that $\Lambda$ is a left noetherian ring. Denote by  $\mathcal{SY}^f_{\Lambda/E}$  the full dg subcategory of  $\mathcal{SY}^b_{\Lambda/E}$ consisting of bounded complexes of finitely generated $\Lambda$-modules.

  We view the bounded homotopy category $\mathbf{K}^b(\Lambda\mbox{-proj})$ of finitely generated projective $\Lambda$-modules as a triangulated subcategory of $\mathbf{D}^b(\Lambda\mbox{-mod})$. Following \cite{Buc, Orl}, the \emph{singularity category} of $\Lambda$ is defined as the following Verdier quotient
 $$\mathbf{D}_{\rm sg}(\Lambda)=\mathbf{D}^b(\Lambda\mbox{-mod})/{\mathbf{K}^b(\Lambda\mbox{-proj})}.$$
Its dg analogue is the \emph{dg singularity category} \cite{Kel18, BRTV, BrDy}, defined as
 $$\mathbf{S}_{\rm dg}(\Lambda)=\mathbf{D}^b_{\rm dg}(\Lambda\mbox{-mod})/{C_{\rm dg}^{b}(\Lambda\mbox{-proj})}.$$
Here, we identify $C_{\rm dg}^{b}(\Lambda\mbox{-proj})$ with the full dg subcategory of $\mathbf{D}^b_{\rm dg}(\Lambda\mbox{-mod})$ formed by bounded complexes of finitely generated projective $\Lambda$-modules.

We have the following finite version of Proposition~\ref{prop:V-SY}; compare Remark~\ref{rem:bounded-ss-SY}(2).

\begin{cor}\label{cor:finite-SY}
Assume that $E$ is a semisimple ring and that $\Lambda$ is a left noetherian ring. Then there is an isomorphism in $\mathbf{Hodgcat}$
 $$ \mathbf{S}_{\rm dg}(\Lambda)\simeq \mathcal{SY}^f_{\Lambda/E},$$
 which induces an isomorphism of triangulated categories
 $$\mathbf{D}_{\rm sg}(\Lambda)\simeq H^0(\mathcal{SY}^f_{\Lambda/E}).$$
\end{cor}

\begin{rem}
(1) Recall that a fully faithful dg functor between dg categories induces a fully faithful functor between their homotopy categories. So the inclusions $\mathcal{SY}^f_{\Lambda/E}\subseteq \mathcal{SY}^b_{\Lambda/E}\subseteq \mathcal{SY}_{\Lambda/E}$ of dg categories induce the following fully faithful triangle functors
$$\mathbf{D}_{\rm sg}(\Lambda)\hookrightarrow \widehat{\mathbf{S}}(\Lambda)\hookrightarrow \mathcal{V}(\Lambda).$$
Here, the fully faithfulness of $\widehat{\mathbf{S}}(\Lambda)\hookrightarrow \mathcal{V}(\Lambda)$ uses (\ref{equ:rem-tri}). We mention that the fully-faithfulness of the functor on the left is known; see \cite[Remark~3.6]{Chen11-MathNach} and compare \cite[Proposition~1.13]{Orl}.

(2) The Hom complexes in $\mathcal{SY}^f_{\Lambda/E}$ have similar flavor with the singular Hochschild  cochain complex \cite{Wan1}. In view of  \cite[Conjecture~1.3]{Kel18}, we expect that the $B_\infty$-structure of the latter complex might be related to the one of the Hochschild cochain complex of $\mathcal{SY}^f_{\Lambda/E}$.
\end{rem}

\subsection{The dg Leavitt algebra as an endomorphism algebra}

We now impose the conditions (Fin1) and (Fin2) in Subsection~\ref{subsec:YE} on the algebra extension $E \to \Lambda$. Associated to $(\overline{\Lambda}, \mu)$ in (\ref{equ:mu}), we have the dg Leavitt  algebra $(L_E(\overline{\Lambda}), \partial)$; see Definition~\ref{defn:dgleavittalgebra}. We view $E$ as a $\Lambda$-module via the algebra homomorphism $\pi$ in (Fin2). Identifying modules with stalk complexes concentrated in degree zero, we view $E$ as an object in $\mathcal{SY}_{\Lambda/E}$.

The isomorphism in the following theorem might be viewed as the core of this work, which establishes a link between the singular Yoneda dg category and the dg Leavitt algebra.

\begin{thm}\label{thm:SY-Leavitt}
Assume that (Fin1) and (Fin2) hold. Then there is an isomorphism  $\mathcal{SY}_{\Lambda/E}(E, E)\simeq (L_E(\overline{\Lambda}), \partial)^{\rm op}$ of dg algebras.
\end{thm}

\begin{proof}
For convenience, we write $M=\overline{\Lambda}$ and omit the subscript $\Lambda/E$ in $\mathcal{Y}_{\Lambda/E}$ and $\mathcal{SY}_{\Lambda/E}$. Observe that $\Omega_{\rm nc}^p(E)=(sM)^{\otimes_E p} \otimes_E E = (sM)^{\otimes_E p} $ and that
$$\mathcal{Y}(E, \Omega_{\rm nc}^{p}(E))=\prod_{n\geq 0}{\rm Hom}_E((sM)^{\otimes_E n}, (sM)^{\otimes_E p})=\bigoplus_{n\geq 0}{\rm Hom}_E((sM)^{\otimes_E n},(sM)^{\otimes_E p}),$$
whose differential is given by $\delta_{\rm ex}$; see Subsection \ref{subsection7.1}. Here, we note that $\delta_{\rm in}$ vanishes.
\newline

\emph{Step 1.} We have a canonical isomorphism
\[
\begin{array}{ccccl}
\psi_p\colon & \mathcal{Y}(E, E)\otimes_E (sM)^{\otimes_E p} &\longrightarrow &  \mathcal{Y}(E, \Omega_{\rm nc}^p(E)) &\\
& g\otimes_E s\overline{a}_{1,p} & \longmapsto &  (s\overline{x}_{1, n}\mapsto (-1)^ {pn} g(s\overline{x}_{1, n}) s\overline{a}_{1,p})&
\end{array}
\]
for any  $g\in \mathcal{Y}_n(E, E)$ and $s\overline{a}_{1, p}\in (sM)^{\otimes_E p}$. Recall from the proof of Proposition~\ref{prop:iso-Y-tensor} the isomorphism $\phi\colon T_E(s^{-1}M^*)\xrightarrow{\simeq} \mathcal{Y}(E, E)$. Then for each $p\geq 0$, we have a composite isomorphism
$$\widetilde \psi_p\colon T_E(s^{-1}M^*)\otimes_E (sM)^{\otimes_E p} \xrightarrow{\phi\otimes_E \mathrm{Id}} \mathcal{Y}(E, E)\otimes_E (sM)^{\otimes_E p}\xrightarrow{\psi_p} \mathcal{Y}(E, \Omega^p_{\rm nc}(E)). $$

We claim that the following diagram
\begin{align}\label{align:commutativediagramLR}
\xymatrix@R=1.5pc{
T_E(s^{-1}M^*)\otimes_E (sM)^{\otimes_E p} \ar[d]_-{L}\ar_-{\simeq}[rr]^-{\widetilde \psi_p} && \mathcal{Y}(E, \Omega^p_{\rm nc}(E)) \ar[d]^-{R} \\
T_E(s^{-1}M^*)\otimes_E (sM)^{\otimes_E p+1} \ar_-{\simeq}[rr]^-{\widetilde \psi_{p+1}} && \mathcal{Y}(E, \Omega^{p+1}_{\rm nc}(E))
}
\end{align}
commutes, where $L$ denotes the map (\ref{equ:nat-mor-c}) and $R$ sends $g$ to $\theta_{\Omega_{\rm nc}^p(Y)}\odot g$. Restricting to homogeneous components, we need  verify the  following commutative diagram.
\[\xymatrix@R=1.5pc{
(s^{-1}M^*)^{\otimes_E n}\otimes_E (sM)^{\otimes_E p} \ar[d]_-{L}\ar[rr]^-{\widetilde\psi_p} && {\rm Hom}_E((sM)^{\otimes_E n}, (sM)^{\otimes_E p}) \ar[d]^-{R} \\
(s^{-1}M^*)^{\otimes_E n+1}\otimes_E (sM)^{\otimes_E p+1} \ar[rr]^-{\widetilde \psi_{p+1}} && {\rm Hom}_E((sM)^{\otimes_E n+1}, (sM)^{\otimes_E p+1}) }\]
More explicitly, we have
$$L(s^{-1}f_{1,n}\otimes_Es\overline{a}_{1, p})=\sum_{i\in S}s^{-1}f_{1,n}\otimes_E s^{-1}  \alpha_i^*\otimes_E s\overline \alpha_i\otimes_E s\overline{a}_{1, p}$$
and $R(g)={\rm Id}_{sM}\otimes_E g$. Furthermore, $\widetilde \psi_p(s^{-1}f_{1,n}\otimes_Es\overline{a}_{1, p})$ is given by
$$s\overline{x}_{1,n}\longmapsto (-1)^{n(p+1)} f_n(\overline{x}_1f_{n-1}(\overline{x}_2 f_{n-2}(   \cdots (\overline{x}_{n-1}f_1(\overline{x}_n))\cdots))) s\overline{a}_{1,p}.$$
Similarly, $\widetilde \psi_{p+1}(s^{-1}f'_{1,n+1}\otimes_Es\overline{b}_{1, p+1})$ is given by
$$s\overline{y}_{1,n+1}\longmapsto (-1)^{(n+1)(p+2)} f'_{n+1}(\overline{y}_1 f'_n(\overline{y}_2 f'_{n-1}( \cdots (\overline{y}_{n}f'_1(\overline{y}_{n+1}))\cdots)))s\overline{b}_{1,p+1}.$$
Therefore, $R \circ \widetilde \psi_p(s^{-1}f_{1,n}\otimes_Es\overline{a}_{1, p})$ equals to the following map
$$s\overline{y}_{1,n+1}\longmapsto (-1)^{(n+1)p} \; s\overline{y}_1\otimes_E f_n(\overline{y}_2f_{n-1}(\overline{y}_3f_{n-2}(\cdots (\overline{y}_{n} f_1(\overline{y}_{n+1}))\cdots)))s\overline{a}_{1,p}.$$
Using (\ref{equ:dualbasis}), we observe that $\widetilde \psi_{p+1}\circ L (s^{-1}f_{1,n}\otimes_Es\overline{a}_{1, p})$ coincides with the above map, proving the claim.
\newline

\emph{Step 2.} Recall from Theorem~\ref{thm:Leavitt-col} that $L_E(M)$ is isomorphic to the colimit of the explicit sequence (\ref{equ:nat-mor-c}). Take the colimits along the maps $L$ and $R$ in \eqref{align:commutativediagramLR}. It follows that the isomorphisms $\widetilde \psi_p$ induces an isomorphism of graded $\mathbb{K}$-modules
$$\Psi\colon L_E(M) \stackrel{\sim} \longrightarrow \mathcal{SY}(E, E).$$

Recall  from \eqref{align:typicalelement} that a typical element $\alpha$ in $L_E(M)$ is represented by a tensor
$$f_{1,n}\otimes_E \overline{a}_{1, p}.$$
To stress the degrees, in what follows,  we will write  $\alpha$ as
$$s^{-1} f_{1, n} \otimes_E s \overline{a}_{1, p}: = s^{-1} f_1 \otimes_E \dotsb \otimes_E s^{-1}f_{n} \otimes_E s\overline{a}_1 \otimes_E \dotsb \otimes_E s \overline{a}_p.$$
 Using the structure map (\ref{equ:nat-mor-c}), we may simultaneously increase the number $n$ and $p$ by one. Consequently, given two typical elements $\alpha$ and $\beta$, we may assume that
$$\alpha=s^{-1}f_{1,n}\otimes_E s\overline{a}_{1, p} \quad \mbox{ and }\quad  \beta=s^{-1}g_{1,p}\otimes_E s\overline{b}_{1, q}$$
for sufficiently large $p$.

We will prove the following identity
\begin{align}\label{equ:Phi-mult}
\Psi(\alpha\bullet \beta)=(-1)^{(n-p)(p-q)}\Psi(\beta)\odot_{\rm sg} \Psi(\alpha).
\end{align}
This implies that $\Psi\colon L_E(M)\longrightarrow \mathcal{SY}(E, E)^{\rm op}$ is an algebra isomorphism. Here, $\bullet$  denotes the product in $L_E(M)$; compare (\ref{equ:mult}).

We observe that $\Psi(\alpha)$ is represented by
$$u=\widetilde \psi_p(s^{-1}f_{1,n}\otimes_E s\overline{a}_{1, p})\in \mathcal{Y}_n(E, \Omega_{\rm nc}^p(E))$$
 and that $\Psi(\beta)$ is represented by
 $$v=\widetilde \psi_q(s^{-1}g_{1,p}\otimes_E s\overline{b}_{1, q})\in \mathcal{Y}_p(E, \Omega_{\rm nc}^{q}(E)).$$ Then $\Psi(\beta) \odot_{\rm sg} \Psi(\alpha)$
is represented by $\Omega_{\rm nc}^p(v) \odot u$ in $\mathcal{Y}_{n+p}(E, \Omega_{\rm nc}^{p+q}(E))$, which is the following composition; compare \eqref{equation:yonedaproduct}
$$(sM)^{\otimes_E(n+p)} \xrightarrow{\mathrm{Id}_{sM}^{\otimes_E p} \otimes_E u}   (sM)^{\otimes_E 2p}\xrightarrow{\mathrm{Id}_{sM}^{\otimes_E p}\otimes_E v} (sM)^{\otimes_E(p+q)}.$$
More precisely, it sends $s\overline{z}_{1, n+p}$ to
$$(-1)^\epsilon s\overline{z}_{1,p}\otimes_E f_n(\overline z_{p+1} f_{n-1} ( {\tiny \dotsb} (\overline z_{n+p-1}f_1(\overline z_{n+p})) {\tiny \dotsb})) \, g_p(\overline{a}_1g_{p-1}(  {\tiny \dotsb}  (\overline{a}_{p-1}g_1(\overline{a}_p)){\tiny \cdots})) s\overline{b}_{1,q}.$$
Here, the sign is given by
$$\epsilon=p(n-p)+(n+1)p+p(p-q)+(p+1)q.$$
We remark that $\epsilon\equiv p+q \; ({\rm mod}\; 2)$.

By \eqref{equ:mult} we have
$$\alpha\bullet \beta= s^{-1}f_{1,n}\otimes_E g_p(\overline{a}_1 g_{p-1}(\cdots (\overline{a}_{p-1}g_1(\overline{a}_p)){\tiny \cdots})) s\overline{b}_{1,q}.$$
Therefore, $\Psi(\alpha\bullet \beta)$ is represented by
$$w=\widetilde \psi_q \left(s^{-1}f_{1,n}\otimes_E g_p(\overline{a}_1 g_{p-1}(\cdots (\overline{a}_{p-1}g_1(\overline{a}_p)){\tiny \cdots})) s\overline{b}_{1,q}\right),$$
 which is an element in $\mathcal{Y}_n(E, \Omega_{\rm nc}^{q}(E))$. However, in $\mathcal{SY}(E, E)$, $w$ is identified with $ \mathrm{Id}_{sM}^{\otimes_E p}\otimes_E w\in \mathcal{Y}_{n+p}(E, \Omega_{\rm nc}^{p+q}(E))$; see \eqref{align:thetaidentification}. The latter element is represented by a map $(sM)^{\otimes_E(n+p)}\rightarrow (sM)^{\otimes_E(p+q)}$, which  sends
  $s\overline{z}_{1, n+p}$ to
$$(-1)^{\epsilon'} s\overline{z}_{1,p}\otimes_E f_n(\overline z_{p+1} f_{n-1} ( {\tiny \dotsb} (\overline z_{n+p-1}f_1(\overline z_{n+p})) {\tiny \dotsb})) \, g_p(\overline{a}_1g_{p-1}(  {\tiny \dotsb}  (\overline{a}_{p-1}g_1(\overline{a}_p)){\tiny \cdots})) s\overline{b}_{1,q}$$
with
$$\epsilon'=p(n-q)+(n+1)q.$$
Then we conclude that
 $$\mathrm{Id}_{sM}^{\otimes_E p}\otimes_E w =(-1)^{(n-p)(p-q)}\; \Omega_{\rm nc}^p(v) \odot u.$$
This verifies (\ref{equ:Phi-mult}).
\newline

\emph{Step 3.} It remains to verify that $\Psi$ preserves the differentials. For this, it suffices to verify the following identity
$$\Psi\circ \partial=\delta_{\rm ex}\circ \Psi$$
on the generating $\mathbb{K}$-submodule $E\oplus(s^{-1}M^*\oplus sM)$. The identity holds trivially on $E$ since both sides vanish.  The verification on $s^{-1}M^*$ is already settled by (\ref{diag:partial+}).

The verification on $sM$ might be deduced from Remark~\ref{rem:partial} and in particular (\ref{equ:partial-uniq}). In what follows, we  give a direct argument. It suffices to verify the following commutative diagram.
\[\xymatrix@R=1.5pc{
sM=sM\otimes_E E\ar[d]_-{\partial_{-}} \ar[rr]^-{\widetilde \psi_0} && {\rm Hom}_E(E, sM)=sM \ar[d]^-{\delta_{\rm ex}}\\
s^{-1}M^*\otimes_E sM\ar[rr]^-{\widetilde \psi_1} && {\rm Hom}_E(sM, sM)
}\]
In this diagram, we observe that $\widetilde \psi_0$ is the identity map. We have
$$\delta_{\rm ex}(s\overline{a})(s\overline{x})=(x\blacktriangleright s\overline{a})=s\mu(\overline{x}\otimes_E \overline{a}),$$
where we use the $\Lambda$-action on $\Omega_{\rm nc}(E)=sM$; see (\ref{equ:nontri-rule}).  On the other hand,
$\partial_{-}(s\overline{a})=\sum_{i\in S}s^{-1}\alpha_i^*\otimes_E s\mu(\overline{\alpha}_i\otimes_E\overline{a})$. Hence, $\widetilde \psi_1\circ \partial_{-1}(s\overline{a})$ sends $s\overline{x}$ to
$$\sum_{i\in S}\alpha_i^*(x)s\mu(\overline{\alpha}_i\otimes_E\overline{a})=s\mu(\overline{x}\otimes_E \overline{a}),$$
where we use (\ref{equ:dualbasis}) for the equality. This proves the required commutativity.
\end{proof}

\section{Applications to finite dimensional algebras}\label{section10}

In this final section, we apply the obtained results to finite dimensional algebras. Proposition~\ref{prop:sing-L} and Theorem~\ref{thm:quiveralgebra} relate  the dg singularity category of a finite dimensional algebra  to a dg Leavitt (path) algebra.   Throughout, $\mathbb{K}$ will be a fixed field.

\subsection{Finite dimensional algebras}

 Let $\Lambda$ be a finite dimensional $\mathbb{K}$-algebra.  Denote by $J={\rm rad}(\Lambda)$ its Jacobson radical and set $E=\Lambda/J$. Denoted  by $\pi\colon \Lambda\rightarrow E$ the natural projection. We assume that there is an algebra embedding $\phi\colon E\rightarrow \Lambda$ satisfying $\pi\circ \phi={\rm Id}_E$. If $E$ is separable over $\mathbb{K}$, such an algebra embedding always exists; see \cite[Theorem~6.2.1]{DK}.

We fix $\phi$ and view $E$ as a $\mathbb{K}$-subalgebra of $\Lambda$. We will use the following isomorphism of $E$-$E$-bimodules
$$J\stackrel{\sim}\longrightarrow \overline{\Lambda}=\Lambda/E, \quad a\mapsto \overline{a}$$
to identify $J$ with $\overline{\Lambda}$. By Definition \ref{defn:dgleavittalgebra}, the multiplication map
$$\mu\colon J\otimes_E J\longrightarrow J, \quad a\otimes_E b\mapsto ab$$
allows us to define the dg tensor algebra $T_E(J^*)=(T_E(J^*), \partial)$ and the dg Leavitt  algebra $L_E(J)=(L_E(J), \partial)$. Here, $J^*={\rm Hom}_E(J, E)$, which is concentrated in degree one. We will suppress the differentials $\partial$ for both $T_E(J^*)$ and $L_E(J)$.

 The following result is expected by experts. It might be viewed a version of Koszul duality; compare \cite[10.5~Lemma, the `exterior' case, and Examples~(c)]{Kel94}.

\begin{prop}\label{prop:der-T}
Keep the notation and assumptions as above. Then there is an isomorphism in $\mathbf{Hodgcat}$
$$\mathbf{D}_{\rm dg}^b(\Lambda\mbox{-{\rm mod}})\simeq  \mathbf{per}_{\rm dg}(T_E(J^*)).$$
Consequently, we have  triangle equivalences
$$\mathbf{D}^b(\Lambda\mbox{-{\rm mod}})\simeq \mathbf{per}(T_E(J^*)) \quad \mbox{ and } \quad \mathbf{K}(\Lambda\mbox{-}{\rm Inj})\simeq \mathbf{D}(T_E(J^*)).$$
\end{prop}

\begin{proof}
 By Corollary~\ref{cor:finite}, we have an isomorphism in $\mathbf{Hodgcat}$
$$\mathbf{D}_{\rm dg}^b(\Lambda\mbox{-{\rm mod}})\simeq \mathcal{Y}^{f}_{\Lambda/E},$$
since $E$ is semisimple.
Recall that $\mathbf{D}^b(\Lambda\mbox{-mod})$ is idempotent-split with $E$ a generator. It follows that $H^0(\mathcal{Y}^{f}_{\Lambda/E})$  is also idempotent-split with $E$  a generator.
Combining Propositions~\ref{prop:quasi-equiv} and ~\ref{prop:iso-Y-tensor},   we obtain an isomorphism in $\mathbf{Hodgcat}$
$$ \mathcal{Y}^{f}_{\Lambda/E}\simeq \mathbf{per}_{\rm dg}(T_E(J^*)).$$
Combining the above two isomorphisms, we obtain the required isomorphism in $\mathbf{Hodgcat}$ and the first consequence.

For the second consequence, we recall from \cite[Proposition~A.1]{Kra} and \cite[Theorem~2.2]{CLiuW} a triangle equivalence
$$ \mathbf{K}(\Lambda\mbox{-}{\rm Inj})\simeq \mathbf{D}(\mathbf{D}^b_{\rm dg}(\Lambda\mbox{-mod})^{\rm op}).$$
As any quasi-equivalence between dg categories induces a derived equivalence, the above two isomorphisms in $\mathbf{Hodgcat}$ induce a derived equivalence
$$ \mathbf{D}(\mathbf{D}^b_{\rm dg}(\Lambda\mbox{-mod})^{\rm op}) \simeq \mathbf{D}(\mathbf{per}_{\rm dg}(T_E(J^*))^{\rm op}).$$
Combining the above two triangle equivalences with the one in Lemma~\ref{lem:der-equi}, we infer the second consequence.
\end{proof}

The following result might be viewed as  a singular analogue of Proposition~\ref{prop:der-T} with a proof of the same style.

\begin{prop}\label{prop:sing-L}
Keep the notation and assumptions as above. Then there is an isomorphism in $\mathbf{Hodgcat}$
$$\mathbf{S}_{\rm dg}(\Lambda)\simeq \mathbf{per}_{\rm dg}(L_E(J)).$$
Consequently, we have  triangle equivalences
$$\mathbf{D}_{\rm sg}(\Lambda)\simeq \mathbf{per}(L_E(J)) \quad \mbox{ and }\quad  \mathbf{K}_{\rm ac}(\Lambda\mbox{-}{\rm Inj})\simeq \mathbf{D}(L_E(J)).$$
\end{prop}

\begin{proof}
By Corollary~\ref{cor:finite-SY}, we have an isomorphism in  $\mathbf{Hodgcat}$
$$\mathbf{S}_{\rm dg}(\Lambda) \simeq \mathcal{SY}^f_{\Lambda/E}.$$
Recall from \cite[Corollary~2.11]{Chen11} that $\mathbf{D}_{\rm sg}(\Lambda)$ is idempotent-split and that $E$ certainly generates it. It follows that the same holds for $H^0(\mathcal{SY}^f_{\Lambda/E})$. By combining Proposition~\ref{prop:quasi-equiv} and Theorem~\ref{thm:SY-Leavitt}, we obtain an isomorphism in  $\mathbf{Hodgcat}$
$$\mathcal{SY}^f_{\Lambda/E}\simeq \mathbf{per}_{\rm dg}(L_E(J)).$$
Then we have the required isomorphism in $\mathbf{Hodgcat}$ and the first consequence.

For the second consequence, we recall from \cite[Theorem~2.2]{CLiuW} the following triangle equivalence
$$ \mathbf{K}_{\rm ac}(\Lambda\mbox{-}{\rm Inj})\simeq \mathbf{D}(\mathbf{S}_{\rm dg}(\Lambda)^{\rm op}).$$
The above two isomorphisms in $\mathbf{Hodgcat}$ yield a derived equivalence
$$\mathbf{D}(\mathbf{S}_{\rm dg}(\Lambda)^{\rm op}) \simeq \mathbf{D}(\mathbf{per}_{\rm dg}(L_E(J))^{\rm op}).$$
Combining the above two triangle equivalences with the one in Lemma~\ref{lem:der-equi}, we infer the second consequence.
\end{proof}

\begin{rem}
Assume that $E$ is separable over $\mathbb{K}$. Then the dg tensor algebra $T_E(J^*)$ is smooth; compare \cite[Subsection~3.6]{Kel11}. Consequently, the quasi-equivalence in Proposition~\ref{prop:der-T}  gives another proof of the smoothness of $\mathbf{D}_{\rm dg}^b(\Lambda\mbox{-{\rm mod}})$; see \cite[Theorem~3.7 and Remark~3.8]{ELS}.

By \cite[Proposition~3.10 c)]{Kel11} or \cite[Section~1]{Kel18}, $\mathbf{S}_{\rm dg}(\Lambda)$ is  also smooth. It follows from the quasi-equivalence in Proposition~\ref{prop:sing-L} that the dg Leavitt algebra $L_E(J)$ is smooth. We expect that a direct proof of this fact might be given by constructing an explicit bounded dg-projective bimodule resolution of $L_E(J)$ from the one in \cite[Proposition~7.5]{CLW}, via  the homological perturbation lemma.
\end{rem}

\subsection{The quiver case}
\label{ss:The quiver case}

In this subsection, we will explore Proposition~\ref{prop:sing-L} in the quiver case. Let $Q$ be a finite quiver. An ideal $I$ of the path algebra $\mathbb{K}Q$ is \emph{admissible} provided that there exists  $d\geq 2$ satisfying $\bigoplus_{n\geq d}\mathbb{K}Q_n\subseteq I\subseteq \bigoplus_{n\geq 2}\mathbb{K}Q_n$.

We fix $\Lambda=\mathbb{K}Q/I$ with $I$ an admissible ideal. Set $E=\mathbb{K}Q_0$, which is naturally a subalgebra of $\Lambda$. The Jacobson radical $J$ equals $\bigoplus_{n\geq 1}\mathbb{K}Q_n/I$. The decomposition
$$\Lambda=E\oplus J$$
allows us to identify $J$ with $\overline{\Lambda} = \Lambda/ E$.

The following notion is due to \cite[Section~3, Definition]{Sch}, in which it is called the basis-graph of $\Lambda$.

\begin{defn}
The \emph{radical quiver} $\widetilde Q$ of $\Lambda=\mathbb{K}Q/I$ is defined as follows: $\widetilde Q_0=Q_0$ and for any vertices $i$ and $j$, the number of arrows in $\widetilde Q$ from $i$ to $j$ equals the dimension of $e_jJe_i$.
\end{defn}

By the very definition and choosing  bases for $e_jJe_i$, we may identify $\mathbb{K}\widetilde Q_1$ with $J$ as a $\mathbb{K}Q_0$-$\mathbb{K}Q_0$-bimodule. The multiplication of $J$ yields an associative map
\begin{align}\label{align:guttproduct}
\mu\colon \mathbb{K}\widetilde Q_1\otimes_{\mathbb{K}\widetilde Q_0} \mathbb{K}\widetilde Q_1\longrightarrow \mathbb{K}\widetilde Q_1.
\end{align}
As in Section~\ref{section4}, we have the dg Leavitt path algebra
$$L(\widetilde Q^\circ)=(L(\widetilde Q^\circ), \partial)$$
 associated to $(\widetilde Q, \mu)$. Here, $\widetilde Q^\circ$ denotes the finite quiver without sinks obtained from $\widetilde Q$ by repeatedly removing sinks. We mention that the differential $\partial$ is determined by the explicit maps in  (\ref{equ:+}) and (\ref{equ:-}).

 The following  reformulation of Proposition~\ref{prop:sing-L} describes the singularity category of $\Lambda$, explicitly to some extent.

\begin{thm}\label{thm:quiveralgebra}
Let $\Lambda = \mathbb KQ/I$ be a finite dimensional $\mathbb K$-algebra with $\widetilde Q$ its radical quiver, and $L(\widetilde Q^\circ)$ be the dg Leavitt path algebra associated to $(\widetilde Q, \mu)$. Then there is an isomorphism in  $\mathbf{Hodgcat}$
 $$\mathbf{S}_{\rm dg}(\Lambda) \simeq  \mathbf{per}_{\rm dg}(L(\widetilde Q^\circ)).$$
Consequently, we have  triangle equivalences
$$\mathbf{D}_{\rm sg}(\Lambda)\simeq \mathbf{per}(L(\widetilde Q^\circ)) \quad \mbox{ and } \quad \mathbf{K}_{\rm ac}(\Lambda\mbox{-}{\rm Inj})\simeq \mathbf{D}(L(\widetilde Q^\circ)).$$
\end{thm}

\begin{proof}
By the identification $J=\mathbb{K}\widetilde Q_1$, we identify the dg Leavitt algebra $L_E(J)$ with the dg Leavitt path algebra $L(\widetilde Q^\circ)$; see Proposition~\ref{prop:iso-path}(2). Then the result follows immediately from Proposition~\ref{prop:sing-L}.
\end{proof}

\begin{rem}\label{rem:deformation}
(1) Assume that $\Lambda$ is radical square zero, or equivalently, $I=\bigoplus_{n\geq 2}\mathbb{K}Q_n$. Then $\widetilde Q=Q$ and the map $\mu$ is zero. It follows that the differential $\partial$ on $L(Q^\circ)$ is also zero. The above equivalences specialize to \cite[Theorem~7.2]{Smi} and \cite[Theorem~6.1]{CY}, respectively. We refer to \cite[Section~10]{AV} and \cite{Chen11, EL} for the study of the singularity category of $\Lambda$ from quite different perspectives.

(2) Let us provide a deformation-theoretic perspective for Theorem~\ref{thm:quiveralgebra}; compare \cite[Section~3]{Sch}.  Consider the radical-square-zero algebra $\widetilde{\Lambda} =\mathbb{K}\widetilde Q_0\oplus \mathbb{K}\widetilde Q_1$ and the Leavitt path algebra $L=L(\widetilde Q^\circ)$ with trivial differential. In view of \cite[Remark~5.19]{CLW}, the sequence of explicit $B_\infty$-quasi-isomorphisms in \cite[the second paragraph of Section~12]{CLW} induces an explicit quasi-isomorphism of dg Lie algebras of degree $-1$
\[
\Psi \colon \overline C_{{\sg}, R, E}^*(\widetilde \Lambda, \widetilde \Lambda) \longrightarrow  \overline{C}_E^*(L, L),
\]
where  $\overline C_{{\sg}, R, E}^*(\widetilde \Lambda, \widetilde \Lambda)$ is the $E$-relative singular Hochschild cochain complex of $\widetilde{\Lambda}$,  and $\overline{C}_E^*(L, L)$ is the $E$-relative normalized Hochschild cochain complex of $L$.

The associative product  $\mu$ in \eqref{align:guttproduct} and the differential $\partial$ in Theorem~\ref{thm:quiveralgebra}  can be respectively viewed as a Maurer--Cartan element in $\overline C_E^*(\widetilde \Lambda, \widetilde \Lambda)$ and in $\overline{C}_E^*(L, L)$; compare \cite[Theorem~7.42]{BW}.  We observe that, under the natural inclusion  $\overline C_E^*(\widetilde \Lambda, \widetilde \Lambda)\hookrightarrow \overline C_{{\sg}, R, E} ^*(\widetilde \Lambda, \widetilde \Lambda)$, the following equality holds.
$$\Psi(\mu) = \partial$$
For this,  we recall  by \cite[(13.4)]{CLW} that $\Psi(\mu)$ lies in $\overline{C}_E^{*, 1}(L, L) = \mathrm{Hom}_{\text{$E$-$E$}}(\overline L, L)$, the Hom complex between two graded $E$-$E$-bimodules $\overline{L}=L/E$ and $L$. Then it suffices to verify that the restriction of $\Psi(\mu)$ on $\mathbb{K}\widetilde{Q}^\circ_1$ and $\mathbb{K} (\widetilde{Q}^\circ_1)^*$ coincides with $\partial_{-}$  in (\ref{equ:-}) and $\partial_{+}$ in (\ref{equ:+}), respectively.  The verification is routine by \cite[Lemma~13.1]{CLW}.

The above observation actually motivates Theorem~\ref{thm:quiveralgebra}. Let us point out that presently,
there seems to be no general deformation theory for pretriangulated dg categories which contains our class of examples.

As observed by \cite{KeLo}, the Hochschild cochain complex of a dg category $\mathcal{A}$ with its Gerstenhaber bracket does not control the deformations of $\mathcal{A}$ among dg categories but among {\em curved} $A_\infty$-categories. The disadvantage of these is that their derived categories often vanish \cite{KellerLowenNicolas10}. In order to obtain deformations without curvature,
it is natural  to impose boundedness conditions on the homology of the morphism
complexes of $\mathcal{A}$; compare \cite{DeDekenLowen13, LowenVandenBergh15, Lur}. These do not
hold in our setting. In particular, Lurie's results in  \cite[Section~5.3]{Lur} do not apply directly.

Indeed, in the section mentioned, Lurie investigates deformations of $\infty$-categories. Via the dg nerve, each pretriangulated dg category gives rise to an $\infty$-category; compare   \cite[Construction~1.3.1.6]{LurieHA}. The main result of \cite[Section~5.3]{Lur} is \cite[Theorem~5.3.33]{Lur}. Here, Lurie shows that under suitable conditions, the deformation
theory of an $\infty$-category is controlled by its Hochschild cochain complex.
The conditions actually only concern the triangulated category $\ct$ associated with
a given  $\infty$-category (for example,  $\ct=H^0(\ca)$ for a pretriangulated
dg category $\ca$). They are as follows:
\begin{itemize}
\item[i)] $\ct$ is {\em tamely} compactly generated, i.e. it is compactly generated and for any
two compact objects $C$ and $D$, we have $\Ext_\mathcal{T}^n(C,D)=0$ for all $n\gg 0$;
\item[ii)] $\ct$ equals its localizing subcategory generated by a family of {\em unobstructible}
objects, i.e. objects $C$ such that $\Ext^n_\mathcal{T}(C,C)=0$ for all $n\geq 2$.
\end{itemize}
Clearly, these conditions are (almost) never satisfied for
singularity categories,  and in particular, Lurie's theorem does not apply in our setting.

(3) By the isomorphism in Theorem~\ref{thm:quiveralgebra},  the cohomology algebra $H^*(L(\widetilde Q^\circ))$ of the dg Leavitt path algebra $L(\widetilde Q^\circ)$ is isomorphic to
the {\it singular Yoneda algebra}
$$\underline{\mathrm{Ext}}_\Lambda^*(E, E): = \bigoplus_{i \in \mathbb Z} \mathrm{Hom}_{\mathbf{D}_{\rm sg}(\Lambda)}  (E, \Sigma^i(E)).$$
In general, $L(\widetilde Q^\circ)$ is not $A_\infty$-quasi-isomorphic to $\underline{\mathrm{Ext}}_\Lambda^*(E, E)$. In other words, $L(\widetilde Q^\circ)$ is not necessarily formal; see Proposition~\ref{prop:nonformal} below.

By the homotopy transfer theorem, we may endow $\underline{\mathrm{Ext}}_\Lambda^*(E, E)$ with an $A_\infty$-algebra structure so that it is $A_\infty$-quasi-isomorphic to $L(\widetilde Q^\circ)$; such an $A_\infty$-algebra structure is unique up to  $A_\infty$-quasi-isomorphism. We might call $\underline{\mathrm{Ext}}_\Lambda^*(E, E)$, endowed with this $A_\infty$-algebra structure, the {\it minimal $A_\infty$-model} of $L(\widetilde Q^\circ)$; compare \cite[Subsection~3.3]{Kel01}.
\end{rem}

\subsection{An example}\label{example:truncatedpoly}

In this final subsection, we will give an explicit example. Let $n\geq 1$ and $Q$ be the quiver with one vertex and one loop $x$. Denote $\Lambda_n = \mathbb K Q/(x^{n+1})$, which is a truncated polynomial algebra in  one variable.

The Jacobson radical $J$ of $\Lambda_n$ has a basis $\{x, x^2, \dotsc, x^n\}$. The radical quiver of $\Lambda_n$ coincides with the rose quiver $R_n$  in  Example~\ref{example:quiverone}, where the arrow $x_i$ corresponds to the basis element $x^i\in J$. The multiplication on $J$ is transferred to the associative product $\mu$ therein. Therefore, the corresponding  dg Leavitt path algebra  coincides with $L(R_n)=(L(R_n), \partial)$, which is explicitly given in Example~\ref{example:quiverone}.

Since $\Lambda_n $  is self-injective,  the singularity category $\mathbf{D}_{\rm sg}(\Lambda_n)$ is triangle equivalent to the stable module category $\Lambda_n\mbox{-\underline{\rm mod}}$; see \cite[Theorem~4.4]{Buc} or \cite[Theorem~2.1]{Ric89JPAA}. Combining this fact with  Theorem~\ref{thm:quiveralgebra}, we have a triangle equivalence
\[
\Lambda_n\mbox{-\underline{\rm mod}} \simeq \mathbf{per}(L(R_n)).
\]

\begin{rem}
If $n=1$ then $\Lambda_1=\mathbb{K}[\epsilon]$ is the algebra of dual numbers. We observe that $L(R_1)\simeq \mathbb{K}[y, y^{-1}]$ with $|y|=1$ and $\partial = 0$; in particular, $L(R_1)$ is formal. We recover the following well-known triangle equivalence
\[
\mathbb{K}[\epsilon]\mbox{-\underline{\rm mod}} \simeq \mathbf{per}(\mathbb{K}[y, y^{-1}]).
\]
\end{rem}

\begin{prop}\label{prop:nonformal}
 The dg Leavitt path algebra   $L(R_n)=(L(R_n), \partial)$ is not formal for any $n\geq 2$.
\end{prop}

\begin{proof}
In this proof, we write $L=(L(R_n), \partial)$. We first claim that the singular Yoneda algebra $\underline{\mathrm{Ext}}_{\Lambda_n}^*(\mathbb K, \mathbb K)$ is isomorphic to the graded commutative algebra $\mathbb K[\epsilon, u, u^{-1}]$ with $\epsilon^2=0$, which is graded by means of $|\epsilon| = 1$ and $|u| = 2$.

Indeed, as mentioned above, we identify $\mathbf{D}_{\rm sg}(\Lambda_n)$ with $\Lambda_n\mbox{-\underline{\rm mod}}$. Denote by $\Omega$ the syzygy endofunctor on $\Lambda_n\mbox{-\underline{\rm mod}}$. For each $i$, $\Sigma^i(\mathbb{K})$ corresponds to $\Omega^{-i}(\mathbb{K})$. In view of Remark~\ref{rem:deformation}(3), we have
$$\underline{\mathrm{Ext}}_{\Lambda_n}^*(\mathbb K, \mathbb K)\simeq \bigoplus_{i\in \mathbb{Z}} { \underline{\rm Hom}}_{\Lambda_n}(\mathbb{K}, \Omega^{-i}(\mathbb{K})),$$
where $\underline{\rm Hom}$ denotes the Hom spaces in $\Lambda_n\mbox{-\underline{\rm mod}}$.
The following fact is well known: if $i$ is even, we have $\Omega^i(\mathbb{K})\simeq \mathbb{K}$; if $i$ is odd,  we have $\Omega^i(\mathbb{K})\simeq \Lambda_n/(x^n)$. It follows that each homogeneous component of $\underline{\mathrm{Ext}}_{\Lambda_n}^*(\mathbb K, \mathbb K)$ is one dimensional, which is represented by the obvious morphism from $\mathbb{K}$ to $\mathbb{K}$, or from $\mathbb{K}$ to $\Lambda_n/(x^n)$, namely the map sending $1_\mathbb K$ to $x^{n-1}+(x^n)$. The multiplication of $\underline{\mathrm{Ext}}_{\Lambda_n}^*(\mathbb K, \mathbb K)$ is induced by the composition of morphisms in $\Lambda_n\mbox{-\underline{\rm mod}}$. Then the claim follows readily.

Recall that  $A_\infty$-quasi-isomorphisms preserve Hochschild cohomology. Therefore, for the required result, it suffices to show that the zeroth Hochschild cohomology of $L$ and $\mathbb K[\epsilon, u, u^{-1}]$ are not isomorphic.

 On one hand, we have
 $$\HH^0(L, L) \simeq \HH^0(\mathbf{per}_{\rm dg}(L), \mathbf{per}_{\rm dg}(L))\simeq \HH^0(\mathbf{S}_{\rm dg}(\Lambda_n), \mathbf{S}_{\rm dg}(\Lambda_n))\simeq  \HH^0_{\sg}(\Lambda_n, \Lambda_n),$$
 where the leftmost isomorphism follows from \cite[Theorem~4.6~c)]{Kel03}, the middle one follows by combining the isomorphism in  Theorem ~\ref{thm:quiveralgebra} and \cite[Theorem~4.6~b)]{Kel03}, and the last one follows from \cite[Theorem~1.1]{Kel18}. Here, $\HH^0_{\rm sg}$ denotes the zeroth singular Hochschild cohomology. Since $\Lambda_n$ is selfinjective, $ \HH^0_{\sg}(\Lambda_n, \Lambda_n)$ is isomorphic to the stable center of $\Lambda_n$; see \cite[Example~3.19]{LZZ}. In particular, it is finite dimensional.

 On the other hand, we claim that the zeroth Hochschild cohomology of
 $$\mathbb K[\epsilon, u, u^{-1}]=\mathbb{K}[\epsilon]\otimes \mathbb{K}[u, u^{-1}]$$
  is infinite dimensional. This will complete the proof.

  Indeed, by \cite[Subsection~6.6]{Kel94}, the normalized bar resolution of $\mathbb K[\epsilon]$ is given by
 \[
P:= \bigoplus_{i \geq 0} \mathbb K [\epsilon] \otimes (s\overline{\mathbb K [\epsilon]})^{\otimes i} \otimes \mathbb K[\epsilon] = \bigoplus_{i \geq 0} \mathbb K [\epsilon] \otimes (s\mathbb{K}\epsilon)^{\otimes i}\otimes \mathbb K[\epsilon],
 \]
whose differential is given by the external one, namely
$$d_{\rm ex}(1\otimes (s\epsilon)^{\otimes i}\otimes 1)=\epsilon\otimes (s\epsilon)^{\otimes i-1} \otimes 1-1\otimes (s\epsilon)^{\otimes i-1} \otimes \epsilon.$$
Here, we identify $s\overline{\mathbb K [\epsilon]}$ with $s\mathbb{K}\epsilon$, which is concentrated in degree zero. We have  the following well-known dg-projective bimodule resolution of $\mathbb K[u, u^{-1}]$
  \[
 Q:= \mathbb K[u, u^{-1}] \otimes s\mathbb K \otimes \mathbb K[u, u^{-1}]  \bigoplus   \mathbb K[u, u^{-1}] \otimes \mathbb{K}\otimes \mathbb K[u, u^{-1}],
  \]
  where $s\mathbb K$ denotes the one dimensional  space concentrated in degree $-1$. The differential $d$ on $Q$ is uniquely determined  by
  $$d(1 \otimes s \otimes 1)=1 \otimes 1 \otimes  1 - u \otimes 1 \otimes u^{-1} \mbox{ and } d(1 \otimes1\otimes  1)=0.$$

  We observe that $P \otimes Q$ is a dg-projective bimodule resolution of $\mathbb K[\epsilon, u, u^{-1}]$. Therefore, the Hochschild cohomology of $\mathbb K[\epsilon, u, u^{-1}]$ is computed by the following Hom complex
  \[
 \mathrm{Hom}_{{\mathbb K[\epsilon, u, u^{-1}]}^e} (P \otimes Q, \mathbb K[\epsilon, u, u^{-1}]),
  \]
  where ${\mathbb K[\epsilon, u, u^{-1}]}^e$ denotes the enveloping algebra. The above Hom complex is isomorphic to
  $$
  \mathrm{Hom}( \bigoplus_{i\geq 0}(s\mathbb{K}\epsilon)^{\otimes i}\otimes (s\mathbb{K}\oplus \mathbb{K}), \mathbb K[\epsilon, u, u^{-1}]);$$
  moreover, the induced differential on the latter complex is zero. It follows that the zeroth Hochschild cohomology of $\mathbb K[\epsilon, u, u^{-1}]$ is isomorphic to the zeroth  component
  \[
  \prod_{i \geq 0} \mathbb K \oplus \prod_{i \geq 0}  \mathbb K,
  \]
 of the above Hom complex,  which is clearly infinite dimensional. This proves the claim.
   \end{proof}

Using the homotopy deformation retract constructed in \cite[Subsection~5.6]{Dyc}, we may obtain the minimal $A_\infty$-model of $L(R_n)$ for $n \geq 2$:
\begin{align*}
 (\mathbb K[\epsilon, u, u^{-1}]; m_1=0, m_2, m_3,\dotsb)
 \end{align*}
 where $m_2$ is the  product  of  $\mathbb K[\epsilon, u, u^{-1}]$ and the only nonzero higher  product is \begin{align*}
m_{n+1}(\epsilon u^{k_1} \otimes \epsilon u^{k_2} \otimes \dotsb \otimes \epsilon  u^{k_{n+1}}) & = u^{k_1 + k_2+ \dotsb + k_{n+1}+1}
\end{align*}
for any $k_1, k_2, \dotsc, k_{n+1} \in \mathbb Z$.

The above $A_\infty$-algebra might be viewed as  a localization of the $A_{\infty}$-algebra structure on  the Yoneda algebra
$$\mathrm{Ext}_{\Lambda_n}^*(\mathbb K, \mathbb K) = \mathbb K[\epsilon, u]$$
with respect to the central element $u$; see \cite[Example~6.3]{LPWZ}. We refer to \cite{CLT} for a $\mathbb Z/2\mathbb Z$-graded version of the above $A_\infty$-algebra, which is obtained from the  category $\mathrm{MF}(\mathbb K [x], x^{n+1})$ of matrix factorizations.

\vskip 10pt

\noindent {\bf Acknowledgements.}\quad We are very grateful to Bernhard Keller for his encouragement, and to Bernhard Keller and Yu Wang for agreeing to write the appendix.  We thank Pere Ara for the reference \cite{CO}, and thank   Martin Kalck and Julian K\"{u}lshammer for helpful comments. This work is supported by the National Natural Science Foundation (Nos.\ 12131015, 	
12071137 and 12161141001) and the Alexander von Humboldt Stiftung.

\vskip 5pt

\appendix
\section{DG Leavitt path algebras for singularity categories, by Bernhard Keller and Yu Wang}
\label{app:A}

In this appendix, we present an alternative proof of Theorem~\ref{thm:quiveralgebra}.
 It is based on Koszul--Moore duality as described
in \cite{Kel03a} and on derived localizations as described in \cite{BCL18}.

\subsection{Modules and comodules}

Let $k$ be a field and $Q$ a finite quiver.
Let $A$ be the quotient $kQ/I$ of the path algebra $kQ$ by
an admissible ideal $I$ (cf.\ Subsection~\ref{ss:The quiver case} for the terminology). Let $R=kQ_0$
be the subalgebra of $A$ generated by the lazy idempotents $e_i$, $i\in Q_0$, and let
$J$ be the Jacobson radical of $A$ (which equals the ideal generated by the arrows).
We have the decomposition $A=R\oplus J$ in the category of $R$-bimodules and we
view $A$ as an augmented algebra in the monoidal category of $R$-bimodules with
the tensor product $\ten_R$. Notice that the vector space $kQ_1$ whose basis is formed
by the arrows of $Q$ is naturally an $R$-bimodule and that the path algebra $kQ$ identifies
with the tensor algebra $T_R(kQ_1)$.

For an $R$-bimodule $M$, we define the dual bimodule by
\[
M^\vee = \Hom_{R^e}(M,R^e).
\]
For example, for $M=kQ_1$, the dual bimodule $M^\vee$ canonically identifies with
$kQ_1^*$, where $Q^*$ is the quiver with the same vertices as $Q$ and whose
arrows are the $\alpha^*\colon j\to i$ for each arrow $\alpha\colon i\to j$ of $Q$.
Notice that for an arbitrary $R$-bimodule $M$, the underlying vector space
of $M^\vee$ identifies with the dual
\[
DM= \Hom_k(M,k)
\]
via the map taking an $R$-bilinear map $f: M \to R^e$ to the linear form
$t\circ f$, where $t: R^e \to k$ takes $e_i\ten e_j$ to $\delta_{ij}\in k$.

As an $R$-bimodule, the algebra $A$ is finitely generated projective so that
$C=A^\vee$ becomes a coalgebra in the category of $R$-bimodules.
We have $C=R\oplus J^\vee$ and the induced comultiplication
\[
J^\vee \to J^\vee \ten_R J^\vee
\]
is conilpotent because the Jacobson radical $J$ of $A$ is nilpotent. Thus, we
may view $C$ as an augmented cocomplete differential graded coalgebra
(in the sense of \cite[Section~2]{Kel03a}), which is moreover concentrated
in degree $0$.

Since $A$ is finitely generated projective as an $R$-bimodule, for each right
$R$-module $M$, we have natural isomorphisms
\[
\Hom_{R}(M\ten_R A, M) \iso M\ten_R A^\vee \ten_R \Hom_R(M,R)  \iso  \Hom_{R}(M, M\ten_R C).
\]
This allows us to convert right $A$-modules into right $C$-comodules.
More precisely, we have an isomorphism of categories
\[
\Mod A \iso \Com C \ko
\]
where $\Mod A$ denotes the category of all right $A$-modules
and $\Com C$ the category of all right $C$-comodules.
Clearly, this isomorphism restricts to an isomorphism
\[
\mod A \iso \com C
\]
between the categories of finite-dimensional modules respectively comodules.

\subsection{Koszul--Moore duality} We refer to  \cite[Section~4]{Kel03a} for
all non defined terminology and for proofs or references to proofs of the claims we make.

Let $\Omega C$ be the cobar construction of $C$ over $R$.
Thus, the underlying graded algebra of $\Omega C$ is the tensor algebra
$T_R(\Sigma^{-1} J^\vee)$ on the desuspension $\Sigma^{-1} J^\vee = J^\vee[-1]$ of $J^\vee=C/R$.
The differential of $\Omega C$ encodes the comultiplication $J^\vee \to J^\vee \ten_R J^\vee$.
The projection $C \to J^\vee$ composed with the inclusion $\Sigma^{-1} J^\vee \to \Omega C$
is the canonical {\em twisting cochain} $\tau\colon C \to \Omega C$. It is an $R$-bimodule
morphism of degree $1$ satisfying
\[
d(\tau) + \tau * \tau =0\ko
\]
where  $d(\tau) = d_{\Omega C} \circ \tau + \tau \circ d_C$ and $*$ is the convolution product
on $\Hom_{R^e}(C, \Omega C)$. For a dg $\Omega C$-module
$L$, the twisted tensor product $L\ten_\tau C$ is defined by twisting the differential
on $L\ten_R C$ using $\tau$, and similarly, for a cocomplete dg $C$-comodule
$M$, the twisted tensor product $M\ten_\tau (\Omega C)$ is defined by twisting
the differential on $M\ten_R (\Omega C)$. We get a pair
of adjoint functors
\[
F=?\ten_\tau (\Omega C) \colon \dgCom(C) \leftrightarrow \dgMod (\Omega C) \colon ?\ten_\tau C = G\ko
\]
where $\dgMod(\Omega C)$ denotes the category of dg right $\Omega C$-modules and
$\dgCom(C)$ the category of cocomplete dg right $C$-comodules. These functors
form in fact a Quillen equivalence for the standard Quillen model structure on
$\dgMod(\Omega C)$ and a suitable Quillen model structure on $\dgCom(C)$,
cf.~\cite{LH03}. Thus, they induce quasi-inverse equivalences
\[
F\colon \cd^c(C) \iso \cd(\Omega C) \colon G\ko
\]
where $\cd^c(C)$ is the coderived category of $C$ and $\cd(\Omega C)$ the
derived category of $\Omega C$. The equivalence $F$ sends $C$ to $R$ and
$R$ to $\Omega C$.

Although it is not necessary for the sequel, let us point out that $\cd^c(C)$ is
equivalent to the homotopy category $\ch(\Inj A)$ of complexes of injective
$A$-modules; cf. Proposition~\ref{prop:der-T}.

Indeed, we know from \cite{LH03} that the fibrant-cofibrant
objects of $\dgCom C$ are exactly the retracts of the cofree dg comodules
(which are automatically cocomplete since $C$ is conilpotent) and that two
morphisms between fibrant-cofibrant dg comodules are homotopic in the
model-theoretic sense if and only if they are homotopic in the classical sense. Thus,
the homotopy category $\cd^c(C)$ of the Quillen model category $\dgCom C$
is equivalent to the usual homotopy category of all complexes of
right $C$-comodules which are retracts of complexes of cofree comodules.
It is not hard to see that this homotopy category is equivalent to the
(slightly larger) homotopy category of complexes
of dg comodules with injective components.
Via the isomorphism $\Com C \to \Mod A$, the complexes of $C$-comodules
with injective components correspond exactly to the
complexes of $A$-modules with injective components so that
we get the equivalence
\[
\ch(\Inj A) \iso \cd^c(C).
\]

\subsection{Description of the singularity category}
Since $R$ generates $D^b(\com C)$ and $\Omega C$
generates $\per(\Omega C)$, the equivalence $\cd^c(C) \iso \cd(\Omega C)$
induces an equivalence
\[
\cd^b(\com C) \iso \per(\Omega C).
\]
By composition with the isomorphism $\cd^b(\mod A) \iso \cd^b(\com C)$, we get
an equivalence $\cd^b(\mod A) \iso \per(\Omega C)$ which sends $R$ to $\Omega C$.
In particular, we get an induced algebra isomorphism
\[
\Ext_A^*(R,R) \iso \Ext^*_{\Omega C}(\Omega C, \Omega C) \iso H^*(\Omega C).
\]
The isomorphism $\cd^b(\mod A) \iso \cd^b(\com C)$ sends the injective
cogenerator $DA \iso A^\vee$ to the cofree comodule $C$ and the equivalence
$\cd^b(\com C) \iso \per(\Omega C)$ sends $C$ to $R$. Thus, we get an induced
equivalence
\[
\cd^b(\mod A)/\thick(DA) \iso \per(\Omega C)/\thick(R).
\]
By composing with the duality functor
\[
D: \cd^b(\mod A^{\op})^{\op} \iso \cd^b(\mod A)
\]
we find an equivalence
\[
\sg(A^{\op})^{\op} \iso \per(\Omega C)/\thick(R) \ko
\]
where $\sg(A^{\op})$ denotes the singularity category of $A^{\op}$.

\subsection{Description of the singularity category as a derived localization}
\label{ss:derived localization}
We put $V=\Sigma^{-1} J^\vee$ so that $\Omega C = T_R(V)$
as a graded algebra.
We have an exact sequence of dg $\Omega C$-modules
\[
0 \to K \to \Omega C \to R \to 0.
\]
Its underlying sequence of graded modules identifies with
\[
0 \to V \ten_R T_R(V) \to T_R(V) \to R \to 0 \ko
\]
where the morphism $V\ten_R T_R(V) \to T_R(V)$ is
just multiplication. Notice that the differential on $K=V\ten_R T_R(V)$
is not $\id_V \ten d_{\Omega C}$ but is induced by that of
$\Omega C$ via the inclusion $K \to \Omega C$. It is not hard
to see that the cone over $K \to \Omega C$ is isomorphic
to $C \ten_\tau \Omega C$. This shows in particular that
this cone lies in $\pretr(\Omega C)$, the closure of $\Omega C$
under shifts and graded split extensions in the category of
dg modules. It follows that $K$ also lies in this category.
Thus, we wish to describe the localization of
$\per(\Omega C)$ with respect to the thick subcategory
generated by the cone over the morphism
\[
V\ten_R (\Omega C) = K \to \Omega C
\]
between two dg $\Omega C$-modules in $\pretr(\Omega C)$.

For this description, let us first recall the construction of universal localizations.
Let $B$ be a hereditary ring and  $S$ a set of morphisms $s: P_1 \to P_0$ between
finitely generated projective (right) $B$-modules. Let $B_S$ be the
universal localization of $B$ with respect to $S$ in the sense of Cohn
\cite{Cohn85}. Thus, the ring $B_S$ is endowed with
a morphism $B \to B_S$ which is universal among the ring morphisms $B \to B'$ such
that $s \ten_B B' : P_1\ten_B B' \to P_0 \ten_B B'$ is invertible for each $s\in S$. If $s$
is a morphism between finitely generated free modules
given by left multiplication by a $p\times q$-matrix $M$, then $B_s$ is obtained from
$B$ by formally adjoining the entries of a matrix $M'$ satisfying
$M M' = I_p$ and $M' M = I_q$. Of course, mutatis mutandis, these
constructions apply to graded hereditary rings and sets $S$ of homogeneous
morphisms of degree $0$.

In our setting, we apply the above to the graded
algebra $B=T_R(V)$. Let $L$ be the graded algebra obtained from $T_R(V)$
by adjoining all the matrix coefficients of a formal inverse of the morphism
(given by multiplication)
\[
V \ten_R T_R(V) \to T_R(V)
\]
between free graded $T_R(V)$-modules. Recall that
$\Sigma V$ is the $R^e$-module dual to the $R^e$-module $J=kQ_1$.
Thus, the $R^e$-module $V$ has the $R^e$-basis  $(\alpha^*)$ dual to the basis $(\alpha)$ of
$J$ formed by the arrows of $Q$ and each $\alpha^*$ is of degree $1$.
If $g$ is the formal inverse of the morphism $V\ten_R T_R(V) \to T_R(V)$, we can write
\[
g(1) = \sum_{\alpha} \alpha^*\ten \alpha.
\]
It is easy to check that requiring the two compositions to be the
respective identities amounts to imposing the Cuntz--Krieger relations; cf. \cite[Proposition~5.2]{Smi}. Thus, the algebra $L$ becomes isomorphic
to a graded Leavitt path algebra; cf. Proposition~\ref{prop:iso-path}(2).
We endow $L$ with the unique differential such that the canonical morphism
\[
\Omega C \to L
\]
becomes a morphism of dg algebras. Clearly, an induction along
the morphism $\Omega C \to L$ induces a triangle functor
\[
\per(\Omega C) \to \per(L),
\]
which annihilates the cone over $K \to \Omega C$ (indeed, the
image of this morphism in $\per(L)$ is invertible and hence its
cone becomes contractible) and thus induces a triangle functor
\[
\per(\Omega C)/\thick(R) \to \per(L).
\]
We claim that this functor is an equivalence. Indeed, this follows
by combining Theorem~4.36 with (a slight generalization with a
similar proof of) Corollary~4.15 in \cite{BCL18}. In the section below, we sketch an alternative, alas not
yet complete approach to the proof of the equivalence based
on a theorem of Neeman--Ranicki \cite{NeemanRanicki04}.

By composition, we obtain the desired equivalence in Theorem~\ref{thm:quiveralgebra}
\[
\sg(A^{op})^{op} \iso \per(L).
\]
It is clear how to obtain a similar description of $\sg(A)$ itself.

\newpage
\subsection{Conjectural approach via Neeman--Ranicki's theorem}

\subsubsection{For rings concentrated in degree $0$}

Let $R$ be a hereditary ring and  $S$ a set of morphisms $s: P_1 \to P_0$ between
finitely generated projective (right) $R$-modules. Let $R_S$ be the
universal localization of $R$ with respect to $S$ in the sense of Cohn
\cite{Cohn85}., cf.~section~\ref{ss:derived localization}.

Let $\cR$ be the localizing subcategory of the
derived category $\cs=\cd R$ generated by the cones $N_s$ over the morphisms $s\in S$.
Put $\ct=\cs/\cR$ so that we have an exact sequence of triangulated categories
\[
0 \to \cR \to \cs \to \ct \to 0.
\]
Clearly the extension of scalars functor $?\lten_R R_S: \cd(R) \to \cd(R_S)$
kills the $N_s$, $s\in S$, and commutes
with arbitrary coproducts. Thus, it kills $\cR$ and we have an induced
canonical triangle functor
\[
\ct=\cs/\cR \to \cd(R_S).
\]
The following theorem is an immediate consequence
of Neeman--Ranicki's work in \cite{NeemanRanicki04}.
\begin{theorem}[Neeman--Ranicki]  \label{thm:Neeman--Ranicki}
The canonical functor $\ct \to \cd(R_S)$ is
an equivalence.
\end{theorem}

\begin{proof}[Sketch of proof] Let $\ct=\cd(R)/\cn$ and let
$\pi: \cd(R) \to \cd(R)/\cn$ be the quotient functor. One first shows that
$\Hom_\ct(\pi(R), \Sigma^n \pi(R))$ vanishes for $n\neq 0$. Thus,
the image $\pi(R)$ is a tilting object in $\ct$ and we have a triangle
equivalence $\cd(E) \iso \ct$, where $E=\Hom_\ct(\pi(R),\pi(R))$.
Now one shows that the morphism $R \to E$ given by $\pi$ identifies with
the universal localization $R \to R_s$. The detailed arguments
are contained in \cite{NeemanRanicki04}.
\end{proof}

\begin{example} Let $k$ be a field and $Q$ a finite quiver.
For each vertex
$i$ of $Q$ which is the source of at least one arrow of $Q$, let
$s_i$ be the morphism
\[
P_i \to \bigoplus_{\alpha: i \to j} P_j
\]
where $P_i=e_i kQ$ and the component of the map associated with $\alpha$ is
the left multiplication by $\alpha$. Let $S$ be the (finite) set of the $s_i$.
Clearly the hypotheses of the theorem hold so that $\cd(R_S)$ identifies
with the quotient of $\cd(R)$ by the localizing subcategory generated by
the cokernels (equivalently: cones) of the $s_i$.
\end{example}

Let us observe that the above theorem easily generalizes
from rings to small categories and to small graded categories.
Of course, its analogue holds for small graded $k$-categories, where
$k$ is a field. So let $\cp$ be a small graded $k$-category
whose category of graded modules (i.e. the category of
$k$-linear graded functors with values in the category of
$\Z$-graded vector spaces) is hereditary. Let $S$ be a
set of morphisms of $\cp$ and $\cp_S$ the localization
of $\cp$ at the set $S$ in the sense of Gabriel--Zisman
\cite{GabrielZisman67}. For example, if $A$ is a graded
$k$-algebra and $S$ a set of homogeneous morphisms in the
category of finitely generated graded projective right
$A$-modules, then $\cp_S$ is Morita-equivalent to
the universal localization $A_S$ of $A$ at $S$; cf. \cite[Proposition~3.1]{CY}.

Now let us denote by $\cp_{hS}$ the localization of
$\cp$ as a dg category in the sense of Drinfeld \cite{Dri}.
By the main result of \cite{Dri}, the canonical functor
\[
\ct=\cd(\cp)/\cn \to \cd(\cp_{hS})
\]
is an equivalence. As a consequence, we obtain the
following variant of Neeman--Ranicki's theorem.

\begin{theorem}[Neeman--Ranicki] The canonical morphism $\cp_{hS} \to \cp_S$
is a quasi-equivalence.
\end{theorem}

\subsubsection{The differential graded case}

Let $k$ be a commutative ring and $\dgcat_k$ the category of
small dg $k$-categories.
Let $F:\ca\to\cb$ be a dg functor between small dg categories.
The functor $F$ is a {\em derived localization} if the induced
functor $F^*: \cd\ca \to \cd\cb$ is a Verdier localization.
Recall the two main model structures on the category of dg categories due
to Tabuada \cite{Tab, Tabuada05b}: the Dwyer-Kan model
structure of \cite{Tab}, whose weak equivalences are
the quasi-equivalences, and the Morita model structure of
\cite{Tabuada05b}, whose weak equivalences are the
Morita functors, \ie the dg functors $F:\ca\to\cb$ such that
$F^*: \cd\ca \to \cd\cb$ is an equivalence. Recall that these
are cofibrantly generated model structures whose sets
of generating cofibrations are identical and consist of
the following dg functors:
\begin{itemize}
\item[a)] the inclusion $\emptyset \to k$
of the empty dg category into the one-object dg category
given by the dg algebra $k$ and
\item[b)] the dg functors $S_n : \cc(n) \to \cp(n)$, defined for $n\in \Z$,
where $\cc(n)$ and $\cp(n)$ are the dg categories with two
objects $1,2$ respectively $3,4$ whose only non trivial morphism complexes
are $\cc(n)(1,2) = S^{n-1}$ respectively $\cp(n)=D^n$, where $S^{n-1}$
is $\Si^{n-1} k$ and $D^n$ the cone over the identity of $S^{n-1}$.
The dg functor $S_n$ maps $1$ to $3$ and $2$ to $4$ and
induces the inclusion $S^{n-1} \to D^n$ of $S^{n-1}$ into the cone
over its identity morphism.
\end{itemize}
A dg category is {\em finitely cellular} if it is obtained from the empty
dg category by a finite number of pushouts along functors in a) or b).
Equivalently, it is the path category of a graded quiver $Q$ such
that the set of arrows $Q_1$ admits a filtration
\[
\emptyset = F_0 Q_1 \subset F_1 Q_1 \subset \cdots \subset F_n Q_1 = Q_1
\]
such that the differential maps the graded path category of $(Q_0, F_p Q_1)$ to
that of $(Q_0, F_{p-1} Q_1)$ for each $1\leq p\leq n$.

From now on, let us assume that $k$ is a field and that $A$ is
a finitely cellular dg algebra given by a graded quiver $Q$ and
a differential on $kQ$. Let $\cp$ be the full subcategory of
the dg category of right dg $A$-modules whose objects
are the finite direct sums of the $\Si^p P_i = e_i A$, $i\in Q_0$, $p\in \Z$.
Let $S$ be a set of (closed) morphisms of $\cp$.

\begin{conjecture} The canonical dg functor $\cp_{hS} \to \cp_S$ is
a quasi-equivalence.
\end{conjecture}

We have not yet proved the conjecture but believe the following
strategy is promising. The given filtration on $A$ yields a filtration
on $\cp$ indexed by $\N$. The localization $\cp_S$ admits a
filtration indexed by $\Z$ such that the functor $\cp \to \cp_S$
becomes universal among the dg functors respecting the
filtration and making the elements of $S$ invertible. We would
like to describe the associated graded category $\gr(\cp_S)$.
Each morphism $s$ in $S$ is given by a matrix whose
entries are linear combinations of paths of $Q$. The
filtration degree $d$ of $s$ is the maximum of the
degrees of the paths appearing with non zero coefficients.
We write $\sigma(s)$ for the image of $s$ in the
$d$th graded component of $\gr(\cp)$ and we write
$\sigma(S)$ for the set of morphisms of $\gr(\cp)$ formed
by the $\sigma(s)$, $s\in S$. It is clear that the $\sigma(s)$
become invertible in $\gr(\cp_S)$.

\begin{lemma} The canonical morphism functor
\[
\gr(\cp)_{\sigma(S)} \to \gr(\cp_S)
\]
is invertible.
\end{lemma}

Recall that $\cp_{hS}$ is obtained from $\cp$ by adjoining,
for each $s: P_1 \to P_2$ in  $S$,
\begin{itemize}
\item a morphism $t: P_2 \to P_1$ of degree $0$,
\item endomorphisms $h_i$ of $P_i$ homogeneous of degree $-1$
such that $d(h_1)=ts-\id_{P_1}$ and $d(h_2)=st-\id_{P_2}$,
\item a morphism $u: P_1 \to P_2$ of degree $-2$ such that
$d(u) = h_2 s - s h_1$.
\end{itemize}
We see that $\cp_{hS}$ admits a $\Z$-indexed filtration such that
the canonical functor $\cp \to \cp_{hS}$ becomes universal
among the functors respecting the filtration.

\begin{lemma} The canonical functor
\[
\gr(\cp)_{h\sigma(S)} \to \gr(\cp_{hS})
\]
is invertible.
\end{lemma}

Clearly, we have a commutative square.
\[
\xymatrix{
	\mathrm{gr}(\cp)_{h\sigma(S)} \ar[d] \ar[r] &\mathrm{gr}(\cp_{hS}) \ar[d]\\
	\mathrm{gr}(\cp)_{\sigma(S)}  \ar[r] & \mathrm{gr}(\cp_{S})
}\]
By the two preceding lemmas, the horizontal functors are invertible.
By Neeman--Ranicki's theorem, the left vertical arrow is a quasi-equivalence.
Thus, the canonical functor
\[
\gr(\cp_{hS}) \to \gr(\cp_S)
\]
is a quasi-equivalence. We would like to conclude that
the canonical functor $\cp_{hS} \to \cp_S$ is a quasi-equivalence.
Unfortunately, this is not clear because the filtrations  are
indexed by $\Z$ rather than $\N$.

\bibliography{}

\begin{thebibliography}{9999}

\bibitem{AA} {\sc G. Abrams, and G. Aranda Pino}, {\em The Leavitt path algebra of a graph}, J. Algebra {\bf 293} (2) (2005), 319--334.


\bibitem{AAM} {\sc G. Abrams, P. Ara, and M. Siles Molina}, Leavitt path algebras, Lecture Notes in Math. {\bf 2191}, Springer-Verlag London, 2017.



\bibitem{ALPS} {\sc G. Abrams, A. Louly, E. Pardo,  and  C. Smith}, {\em Flow invariants in the classification of Leavitt path algebras}, J. Algebra {\bf 333} (2011), 202--231.


\bibitem{AGGP} {\sc P. Ara, M.A. Gonzalez-Barroso, K.R. Goodearl,  and E. Pardo}, {\em  Fractional skew monoid rings}, J. Algebra {\bf 278} (1) (2004), 104--126.



\bibitem{AMP} {\sc P. Ara, M.A. Moreno,  and E. Pardo}, {\em Nonstable $K$-theory for graph algebras}, Algebr. Represent. Theor. {\bf 10} (2) (2007), 157--178.


\bibitem{AV} {\sc L.L. Avramov, and O. Veliche}, {\em Stable cohomology over local rings}, Adv. Math. {\bf 213} (2007), 93--139.



\bibitem{BW}{\sc S. Barmeier, and Z. Wang}, {\em Deformations of path algebras of quivers with relations, } arXiv:2002.10001v4, 2021.





\bibitem{BSZ} {\sc R. Bautista, L. Salmeron, and R. Zuazua}, Differential tensor algebras and their module categories, London Math. Soc. Lecture Notes Ser. {\bf 362}, Cambridge Univ. Press, 2009.



\bibitem{Bei} {\sc A.A. Beilinson}, {\em Coherent sheaves on $\mathbf{P}^n$ and problems in linear algebra}, Funct. Anal.  Appl. {\bf 12} (1978), 214--216.



\bibitem{BRTV} {\sc A. Blanc, M. Robalo, B. T\"{o}en, and G. Vezzosi}, {\em Motivic realizations of singularity categories and vanishing cycles},  J. \'{E}c. polytech. Math. {\bf 5} (2018), 651--747.



\bibitem{BK}  {\sc A.I. Bondal, and M.M. Kapranov}, {\em Enhanced triangulated categories}, Mat. Sb. {\bf 181}:5 (1990), 669--683; translation: Math. USSR-Sb. {\bf 70}:1 (1991), 93--107.



\bibitem{BLL} {\sc A.I. Bondal, M. Larsen, and V.A. Lunts}, {\em Grothendieck ring of pretriangulated categories}, Inter. Math. Res. Not. IMRN {\bf 29} (2004), 1461--1495.

\bibitem{BCL18} {\sc C.~Braun, J.~Chuang, and A.~Lazarev}, {\em Derived localization of algebras and modules}, Adv. Math. {\bf 328} (2018), 555--622.

\bibitem{BrDy}{\sc M.K. Brown, and T. Dyckerhoff}, {\em Topological $K$-theory of equivariant singularity categories}, Homology Homotopy Appl.  {\bf 22} (2) (2020), 1--29.

\bibitem{Buc}{\sc R.O. Buchweitz}, {\em Maximal Cohen-Macaulay modules and Tate-cohomology over Gorenstein rings}, http://hdl.handle.net/1807/16682, Univ. Hannover, 1986.



\bibitem{CLT} {\sc A. C\u{a}ld\u{a}raru, S. Li, and J. Tu}, {\em Categorical primitive forms and Gromov--Witten invariants of $A_{n}$ singularities}, Inter. Math. Res. Not. IMRN {\bf 24} (2021), 18489--18519.



\bibitem{CO} {\sc T.M. Carlsen, and E. Ortega}, {\em Algebraic Cuntz-Pimsner rings}, Proc. London Math. Soc. (3) {\bf 103} (2011), 601--653.


\bibitem{CC} {\sc X. Chen, and X.W. Chen}, {\em Liftable derived equivalences and objective categories}, Bull. London Math. Soc. {\bf 52} (2020), 816--834.

\bibitem{Chen11-MathNach} {\sc X.W. Chen}, {\em  Relative singularity categories and Gorenstein-projective modules,} Math. Nachr. {\bf 284} (2-3) (2011), 199--212.


\bibitem{Chen11}{\sc X.W. Chen}, {\em The singularity category of an algebra with radical square zero}, Doc. Math. {\bf 16} (2011), 921--936.


\bibitem{CLW}{\sc X.W. Chen, H. Li, and Z. Wang}, {\em Leavitt path algebras, $B_\infty$-algebras and Keller's conjecture for singular Hochschild cohomology}, arXiv:2007.06895v3, 2020.

\bibitem{CLiuW} {\sc X.W. Chen, J. Liu, and R. Wang}, {\em Singular equivalences induced by bimodules and quadratic monomial algebras},  Algebr. Represent. Theor. (2021),  DOI:10.1007/s10468-021-10104-3.

\bibitem{CY} {\sc X.W. Chen, and D. Yang}, {\em Homotopy categories, Leavitt path algebras and Gorenstein projective modules}, Inter. Math. Res. Not. IMRN {\bf 10} (2015), 2597--2633.


\bibitem{Co}{\sc P.M. Cohn}, {\em  Some remarks on the invariant basis property}, Topology {\bf  5} (1966), 215--228.

\bibitem{Cohn85} {\sc P.M. Cohn}, Free rings and their relations, second ed., London
  Math. Soc. Monographs {\bf 19}, Academic Press Inc., London, 1985.


\bibitem{Cor} {\sc G. Corti\~nas}, {\em Classifying Leavitt path algebras up to involution preserving homotopy}, arXiv: 2101.05777v3, 2021.




\bibitem{CorMon20}{\sc G. Corti\~nas, and D. Montero}, {\em Homotopy classification of Leavitt path algebras}, Adv. Math. {\bf 362} (2020), 106961.


\bibitem{CorMon18}{\sc G. Corti\~nas, and D. Montero}, {\em Algebraic bivariant K-theory and Leavitt path algebras}, J. Noncommut. Geom. {\bf 15} (1) (2021), 113--146.



\bibitem{CQ} {\sc J. Cuntz, and D. Quillen}, {\em Algebra extensions and nonsingularity}, J. Amer. Math. Soc.  {\bf 8} (2) (1995), 251--289.


\bibitem{DeDekenLowen13} {\sc O. De~Deken, and W. Lowen}, \emph{On deformations of triangulated
  models}, Adv. Math. \textbf{243} (2013), 330--374.



\bibitem{Dri} {\sc V. Drinfeld}, {\em DG quotients of DG categories}, J. Algebra {\bf 272} (2) (2004), 643--691.

\bibitem{DK} {\sc Y.A. Drozd, and V.V. Kirichenko}, Finite dimensional algebras, with an appendix by V. Dlab, Springer-Verlag, Berlin Heidelberg, 1994.

\bibitem{Dyc} {\sc T. Dyckerhoff}, {\em Compact generators in categories of matrix factorizations}, Duke Math. J. {\bf 159} (2) (2011), 223--274.




\bibitem{EL} {\sc A. Elagin, and V.A. Lunts}, {\em Derived categories of coherent sheaves on some zero-dimensional schemes}, J. Pure Appl. Algebra {\bf 226} (6) (2022), 106939.


\bibitem{ELS}  {\sc A. Elagin, V.A. Lunts, and O.M. Schnuerer}, {\em Smoothness of derived categories of algebras},  Moscow Math. J. {\bf 20} (2) (2020), 277--309.


\bibitem{GabrielZisman67} {\sc P.~Gabriel and M.~Zisman}, Calculus of fractions and homotopy theory,
  Ergebnisse der Mathematik und ihrer Grenzgebiete, Band {\bf 35}, Springer-Verlag
  New York Inc., New York, 1967.


\bibitem{Haz}   {\sc R. Hazrat}, {\em The dynamics of Leavitt path algebras}, J. Algebra {\bf 384} (2013), 242--266.


\bibitem{IS} {\sc M.C. Iovanov, and A. Sistko}, {\em On the Toeplitz-Jacobson algebra and direct finiteness}, in: Groups, Rings, Group Rings, and Hopf Algebras, Contemporary Math. {\bf 688} (2017), 113--124.




\bibitem{Kel94} {\sc B. Keller}, {\em Deriving DG categories}, Ann. Sci. \'{E}cole Norm Sup. (4) {\bf 27}(1) (1994), 63--102.


\bibitem{Kel98} {\sc B. Keller},   {\em Invariance and localization for cyclic homology of dg algebras}, J. Pure Appl. Algebra {\bf 123} (1998), 223--273.



\bibitem{Kel99} {\sc B. Keller}, {\em On the cyclic homology of exact categories}, J. Pure Appl. Algebra {\bf 136}(1) (1999) 1--56.

\bibitem{Kel01} {\sc B.~Keller}, \emph{Introduction to A-infinity algebras and modules}, Homology, Homotopy and Appl. {\bf 3} (2001), 1--35.



\bibitem{Kel03} {\sc B.~Keller}, {\em Derived invariance of higher structures on the Hochschild complex}, preprint, 2003,  available at https://webusers.imj-prg.fr/bernhard.keller/publ/index.html.

\bibitem{Kel03a} {\sc B.~Keller}, {\em Koszul duality and coderived categories (after K.~Lef\`evre)}, preprint 2003, available at https://webusers.imj-prg.fr/bernhard.keller/publ/index.html.


\bibitem{Kel05} {\sc B. Keller}, {\em On triangulated orbit categories}, Doc. Math. {\bf 10} (2005), 551--581.



\bibitem{Kel06} {\sc B. Keller}, {\em On differential graded categories},  International Congress of Mathematicians. Vol. II, 151--190, Eur. Math. Soc., Z\"{u}rich, 2006.


\bibitem{Kel11} {\sc B. Keller}, {\em Deformed Calabi-Yau completions}, with an appendix by Michel Van den Bergh, J. Reine Angew. Math. {\bf 654} (2011), 125--180.



\bibitem{Kel18}{\sc B. Keller}, {\em Singular Hochschild cohomology via the singularity category},
C. R. Math. Acad. Sci. Paris {\bf 356}  (11-12) (2018), 1106--1111. {\em Corrections}, C. R. Math. Acad. Sci. Paris {\bf 357} (6) (2019), 533--536. See also arXiv:1809.05121v10, 2020.


\bibitem{KeLo}{\sc B. Keller, and W. Lowen}, {\em On Hochschild cohomology and Morita deformations}, Inter. Math. Res. Not. IMRN {\bf 2009} (17) (2009), 3221--3235.

\bibitem{KellerLowenNicolas10} {\sc B. Keller, W. Lowen, and P. Nicol\'{a}s}, \emph{On the
  (non)vanishing of some ``derived'' categories of curved dg algebras}, J. Pure Appl. Algebra \textbf{214} (7) (2010), 1271--1284.


\bibitem{Kra}{\sc H. Krause}, {\em The stable derived category of a noetherian scheme}, Compos. Math. {\bf 141} (5) (2005), 1128--1162.

\bibitem{Kuel} {\sc J. K\"{u}lshammer, } {\em In the bocs seat: quasi-hereditary algebras and representation type}, in: Representation theory--current trends and perspectives, 375--426, Eur. Math. Soc. Publishing House, 2017.



\bibitem{Lea} {\sc W.G. Leavitt}, {\em The module type of a ring}, Trans. Amer. Math. Soc. {\bf 103} (1962), 113--130.



\bibitem{LH03} {\sc K. Lef\`vre-Hasegawa}, {\em Sur les $A_\infty$-cat\'egories}, Th\`ese de doctorat,
{U}niversit\'e {D}enis {D}iderot -- Paris 7, November 2003, arXiv:math.CT/0310337.


\bibitem{Li} {\sc H. Li}, {\em The injective Leavitt complex}, Algebr. Represent. Theor. {\bf 21} (4) (2018), 833--858.



\bibitem{LZZ} {\sc Y. Liu, G. Zhou, and A. Zimmermann}, {\em Higman ideal, stable Hochschild homology and Auslander-Reiten conjecture}, Math. Z. {\bf 270} (2012), 759--781.



\bibitem{LoVa} {\sc W. Lowen, and M. Van den Bergh}, {\em Hochschild cohomology of abelian categories and ringed spaces}, Adv. Math. {\bf 198} (1) (2005), 172--221.


\bibitem{LowenVandenBergh15} {\sc W. Lowen, and M. Van~den Bergh}, \emph{The curvature problem for formal
  and infinitesimal deformations}, arXiv:1505.03698, 2015.


\bibitem{LPWZ}{\sc D.M. Lu, J.H. Palmieri, Q.S. Wu, and J.J. Zhang}, {\em $A$-infinity structure on $\mathrm{Ext}$-algebras}, J. Pure Appl. Algebra {\bf 213} (11) (2009), 2017--2037.



\bibitem{LO} {\sc V.A. Lunts, and D. Orlov}, {\em Uniqueness of enhancement for triangulated categories}, J. Amer. Math. Soc. {\bf 23} (3) (2010), 853--908.


\bibitem{LS16} {\sc V.A. Lunts, and O.M. Schnuerer}, {\em Matrix factorizations and motivic measures}, J. Noncommut. Geom. {\bf 10}(3) (2016), 981--1042.

\bibitem{Lur} {\sc J. Lurie}, Derived algebraic geometry X: formal moduli problems, preprint,  available at: https://www.math.ias.edu/$\sim$lurie/papers/DAG-X.pdf, 2011.

\bibitem{LurieHA} {\sc J. Lurie}, \emph{Higher algebra}, https://www.math.ias.edu/$\sim$lurie/papers/HA.pdf, 2017.


\bibitem{MacL} {\sc S. Mac Lane}, Homology, Reprint of the 1975 Edition, Springer-Verlag, Berlin Heidelburg, 1995.

\bibitem{NeemanRanicki04} {\sc A. Neeman, and A. Ranicki}, \emph{Noncommutative localisation in algebraic
  {$K$}-theory\; {I}}, Geom. Topol. \textbf{8} (2004), 1385--1425.


\bibitem{Orl} {\sc D. Orlov}, {\em Triangulated categories of singularities and $D$-branes in Landau-Ginzburg models},  Proc. Steklov Inst. Math. {\bf 246} (3) (2004), 227--248.


\bibitem{Ric89} {\sc J. Rickard}, {\em Morita theory for derived categories}, J. London Math. Soc. (2) {\bf 39} (1989), 436--456.


\bibitem{Ric89JPAA} {\sc J. Rickard}, {\em Derived categories and stable equivalence,} J. Pure Appl. Algebra {\bf 61} (1989), 303--317.



\bibitem{Sch} {\sc M. Schaps}, {\em Deformations of finite dimensional algebras and their idempotents}, Trans. Amer. Math. Soc. {\bf 307} (1988), 843--856.


\bibitem{Smi} {\sc S.P. Smith}, {\em Category equivalences involving graded modules over path algebras of quivers}, Adv. Math. {\bf 230} (2012), 1780--1810.


\bibitem{Swe} {\sc M.E. Sweedler}, {\em The predual theorem to the Jacobson-Bourbaki theorem}, Trans.
Amer. Math. Soc. {\bf 213} (1975), 391--406.


\bibitem{Tab} {\sc G. Tabuada}, {\em  Une structure de cat\'{e}gorie de mod\`{e}les de Quillen sur la cat\'{e}gorie des dg-cat\'{e}gories}, C. R. Math. Acad. Sci. Paris {\bf 340} (1) (2005), 15--19.

\bibitem{Tabuada05b} {\sc G. Tabuada}, \emph{Invariants additifs de dg-cat\'egories}, Inter. Math. Res. Not. IMRN  \textbf{53} (2005), 3309--3339.

\bibitem{Tab10} {\sc G. Tabuada}, {\em On Drinfeld's DG quotient}, J. Algebra {\bf 323} (2010), 1226--1240.

\bibitem{TT} {\sc J. Terilla, and T. Tradler}, {\em Deformations of associative algebras with inner products}, Homology Homotopy Appl. {\bf 8} (2) (2006), 115--131.


\bibitem{Toe}{\sc B. To\"en}, {\em The homotopy theory of dg-categories and derived Morita Theory}, Invent. Math. {\bf 167} (2007), 615--667.



\bibitem{ZF}{\sc Z. Wang}, {\em  Equivalence singuli\`ere \`a la Morita et la cohomologie de Hochschild singuli\`ere}, PhD thesis, Universit\'e Paris Diderot-Paris  7,  2016, available at https://www.theses.fr/2016USPCC203.



\bibitem{Wan1}{\sc Z. Wang}, {\em Gerstenhaber algebra and Deligne's conjecture on Tate-Hochschild cohomology}, Trans. Amer. Math. Soc. {\bf 374} (2021), 4537--4577.


\end{thebibliography}

\vskip 10pt

 {\footnotesize \noindent Xiao-Wu Chen\\
 Key Laboratory of Wu Wen-Tsun Mathematics, Chinese Academy of Sciences\\
 School of Mathematical Sciences, University of Science and Technology of China, Hefei 230026, Anhui, PR China\\
 xwchen@mail.ustc.edu.cn}

 \vskip 5pt

 {\footnotesize \noindent Bernhard Keller\\
Universit\'e de Paris, UFR de math\'ematiques, CNRS IMJ--PRG \\
8 place Aur\'elie Nemours, 75013 Paris, France\\
bernhard.keller$\symbol{64}$imj-prg.fr}

  \vskip 5pt

{\footnotesize \noindent Yu Wang\\
School of Mathematics and Statistics, Taiyuan Normal University, Jinzhong 030619, PR China\\
dg1621017$\symbol{64}$smail.nju.edu.cn\\
and\\
Universit\'e de Paris, UFR de math\'ematiques, CNRS IMJ--PRG \\
8 place Aur\'elie Nemours, 75013 Paris, France\\
yu.wang$\symbol{64}$imj-prg.fr}

\vskip 5pt

{\footnotesize \noindent Zhengfang Wang\\
Institute of Algebra and Number Theory, University of Stuttgart\\
Pfaffenwaldring 57, 70569 Stuttgart, Germany \\
 zhengfangw$\symbol{64}$gmail.com}

\end{document}